\documentclass[english,reqno,11pt]{amsart}
\usepackage{amsmath, amsthm, amsfonts}
\usepackage{hyperref} 
\hypersetup{hidelinks,backref=true,pagebackref=true,hyperindex=true,colorlinks=true,citecolor=red,linkcolor=blue,breaklinks=true,urlcolor=ocre,bookmarks=true,bookmarksopen=false,pdftitle={Title},pdfauthor={Author}}

\usepackage{amsmath,amsthm,amssymb}
\usepackage{amscd,indentfirst,epsfig}
\usepackage{latexsym}
\usepackage{times}
\usepackage{enumerate}
\usepackage{mathrsfs}
\usepackage{amsopn}
\usepackage{amsmath}
\usepackage{amssymb,dsfont}
\usepackage{amsfonts,bm}
\usepackage{amsbsy,amsmath}
\usepackage{amscd}
\usepackage{xcolor}
\begin{document}

\makeatletter
\def \leq {\leqslant}
\def \geq {\geqslant}
\def\e{\alpha}
\def\tend{\rightarrow}
\def\R{\mathbb R}
\def\S{\mathbb S}
\def\Z{\mathbb Z}
\def\N{\mathbb N}
\def\C{\mathcal C}
\def\D{\mathcal D}
\def\T{\mathcal T}
\def\E{\mathcal E}
\def\s{\sigma}
\def\k{\kappa}
\def \a{\beta}
 \def\be{\beta}
\def\t{\theta}
\def\l{\lambda}
\def \M {\mathcal{M}}
\def\L{\Lambda}
\def\O{\Omega}
\def\g{\gamma}
\def\G{\Gamma}
\def \mM {\mathfrak{m}}
\def\f{\varphi}
\def\o{\omega}
\def \d {\mathrm{d}}
\def \lm {\bm{m}}
\def \lM {\mathds{M}}
\def \lD {\mathds{D}}
\def\U{\Upsilon}
\def\z{\zeta}
\def \Q {\mathcal{Q}}
\def\over{\bm}
\def\b{\backslash}
\def\fet{f_{\ast}}
\def \get{g_{\ast}}
\def\fprim{f^{\prime}}
\def\fprimet{f^{\prime}_\ast }
\def\vet{v_{\ast}}
\def\vprim{v^{\prime}}
\def\vprimet{v^{\prime}_{\ast}}
\def\grad{\nabla}
\def\Log{\textrm{Log }}
\newtheorem{theo}{Theorem}[section]
\newtheorem{prop}[theo]{Proposition}
\newtheorem{cor}[theo]{Corollary}
\newtheorem{lem}[theo]{Lemma}
\newtheorem{hyp}[theo]{Assumptions}
\newtheorem{defi}[theo]{Definition}
\newtheorem{rmq}[theo]{Remark}

\def \dd {\bm{\varepsilon}}

\numberwithin{equation}{section}

\title[Long time dynamics for the Landau-Fermi-Dirac equation]{Long time dynamics for the Landau-Fermi-Dirac equation with hard potentials}

\author{Ricardo {\sc Alonso}}

\address{Departamento de Matem\'{a}tica, PUC-Rio, Rua Marqu\^{e}s de S\~ao Vicente 225, Rio de Janeiro, CEP 22451-900, Brazil.} \email{ralonso@mat.puc-rio.br}

 \author{V\'{e}ronique {\sc Bagland}}

 \address{Universit\'{e} Clermont Auvergne, LMBP, UMR 6620 - CNRS,  Campus des C\'ezeaux, 3, place Vasarely, TSA 60026, CS 60026, F-63178 Aubi\`ere Cedex,
 France.}\email{Veronique.Bagland@math.univ-bpclermont.fr}

 \author{Bertrand {\sc Lods}}

 \address{Universit\`{a} degli
Studi di Torino \& Collegio Carlo Alberto, Department of Economics and Statistics, Corso Unione Sovietica, 218/bis, 10134 Torino, Italy.}\email{bertrand.lods@unito.it}

\maketitle

\begin{abstract} 
In this document we discuss the long time behaviour for the homogeneous Landau-Fermi-Dirac equation in the hard potential case.  Uniform in time estimates for statistical moments and Sobolev regularity are presented and used to prove exponential relaxation of non degenerate distributions to the Fermi-Dirac statistics.  All these results are valid for rather general initial datum.  An important feature of the estimates is the independence with respect to the quantum parameter.  Consequently, in the classical limit the same estimates are recovered for the Landau equation.
 

\end{abstract}

%\tableofcontents

\section{Introduction}
\subsection{The model} We study in this document the long time behaviour of a particle gas satisfying the Pauli's exclusion principle in the Landau's grazing limit regime.  More specifically, we study the Landau-Fermi-Dirac (LFD) equation in the homogeneous setting for hard potential interactions described as
\begin{equation}\label{LFDeq}
\partial_{t} f(t,v) =  \Q_L(f)(t,v),\qquad (t,v)\in (0,\infty)\times\mathbb{R}^{3}\,, \qquad f(0)=f_{0}\,,
\end{equation}
where the collision operator is given by a modification of the Landau operator which includes the Pauli's exclusion principle
\begin{equation*}
\Q_L (f)(v)= {\grad}_v \cdot \int_{\R^3} \Psi(v-\vet) \, \Pi(v-\vet) 
\Big\{ \fet (1- \dd \fet) \grad f - f (1- \dd f) {\grad f}_\ast \Big\}
\, \d\vet\,.
\end{equation*}
We use the standard shorthand $f=f(t,v)$ and $\fet=f(t,\vet)$.  The matrix $\Pi(z)$ denotes the 
orthogonal projection on $(\R z)^\perp$, 
$$\Pi_{i,j}(z)=\delta_{i,j}-\frac{z_i z_j}{|z|^2}, \qquad 1\leq i , j\leq 3\,,$$
and  $\Psi(z)=|z|^{2+\g}$ is the kinetic potential.  The choice $\Psi(z)=|z|^{2+\g}$ corresponds 
to inverse power law potentials.  This document only considers the case of hard 
potentials, that is $0< \g\leq 1$.  We point out that the Pauli exclusion 
principle implies that a solution to \eqref{LFDeq} must satisfy the \emph{a priori} bound 
\begin{equation}\label{eq:Linf}
0\leq f(t,v)\leq \dd^{-1},
\end{equation}
where the \textit{quantum parameter} 
$$\dd:= \frac{(2\pi\hslash)^{3}}{m^{3}\beta} >0$$ 
depends on the reduced Planck constant $\hslash \approx 1.054\times10^{-34} \mathrm{m}^{2}\mathrm{kg\,s}^{-1}$, the mass $m$ and the statistical weight $\beta$ of the particles species, see \cite[Chapter 17]{chapman}.  Recall that the statistical weight is the number of independent quantum states in which the particle can have the same internal energy.  For example, for electrons $\beta=2$ corresponding to the two possible electron spin values.  In the case of electrons $m\approx 9.1\times10^{-31}$ kg, and therefore, $\dd \approx 1.93\times10^{-10}\ll 1$. The parameter $\dd$ encapsules the quantum effects of the model with the case $\dd=0$ corresponding to the classical Landau equation as studied in \cite{DeVi1,DeVi2}.

\subsection{Mathematical difficulty and our contribution in a nutshell}

Suitable modifications of classical kinetic equations, such as Boltzmann or Landau equations, that include quantum effects have been proposed in the literature since the pioneering works \cite{nordheim,UU}\footnote {Refer to \cite[Chapter 2E, Section 3]{villani} for a more detailed account.}.  In particular, the LFD equation is a natural modification of the Landau equation modelling a gas in the grazing collision regime.  The LFD equation has been introduced in several contexts \cite{chapman,daniel,kadomtsev,chavanis,lynden}.  The common feature to these kinetic models is that the relevant steady state is given by the \textit{Fermi-Dirac} statistics
\begin{equation}\label{FDs}
\M_{\dd}(v)=\frac{M_{\dd}(v)}{1+\dd\,M_{\dd}(v)}\,,
\end{equation}
where $M_{\dd}(v)$ is a suitable Maxwellian distribution which allows to recover, as $\dd \to 0$, a Maxwellian equilibrium. Of course, $\M_{\dd}$ satisfies \eqref{eq:Linf}.

\medskip
\noindent
There are several references studying the well posedness of the Cauchy problem \cite{bag1} and propagation of regularity of solutions \cite{Chen10,Chen11} to the LFD equation \eqref{LFDeq} as well as the precise form of steady states \cite{bag2}.  In addition, there are related references \cite{Escobedo,Lu,LW} treating Fermi-Dirac gases with the Boltzmann equation in the homogeneous setting.

\medskip
\noindent
A common feature of kinetic equation for particles satisfying Pauli's exclusion principle is that a suitable $L^{\infty}-$ \emph{a priori} estimate such as \eqref{eq:Linf} holds true. Such bound has been fully exploited to prove existence of solutions in the homogeneous and inhomogeneous setting \cite{alex,bag1,Dolbeault,Escobedo,Ma,Lions,Lu,LW}.  A negative issue related to this bound is that it brings a degeneracy in the set $\{(t,v):f(t,v)=\dd^{-1}\}$ which, in turn, makes time uniform estimates for the statistical moments and Sobolev regularity difficult to establish.  That this degeneracy is quite real is more evident from the fact that, besides the Fermi-Dirac statistics \eqref{FDs}, the distribution
\begin{equation}\label{eq:dege}
F_{\dd}(v)=\begin{cases}
\dfrac{1}{\dd}  \qquad \text{ if } \quad &|v|\leq \left(\dfrac{3\varrho\dd}{|\S^{2}|} \right)^{\frac{1}{3}}\\
0 \qquad \:\:\text{ if } \quad &|v|> \left(\dfrac{3\varrho\dd}{|\S^{2}|}\right)^{\frac{1}{3}}\,,\end{cases}
\end{equation}   
can be a stationary state with prescribed mass $\varrho=\int_{\R^{3}}F_{\dd}(v)\d v$, and where $|\S^{2}|$ is the volume of the unit sphere.  Such a degenerate, referred to as \emph{saturated Fermi-Dirac}, stationary state can occur for very cold gases where an explicit condition on the gas temperature can be found, see Appendix \ref{app:FD} for details.  For initial distributions near such a degenerate state, the regularization process can take a long time relative to those further away.  Of course, this impacts the transitory time of the particle relaxation.  

\medskip
\noindent
In this work we overcome such difficulty by controlling simultaneously the statistical moments and higher norms, see later for a more precise statement.  This can be done due to the strong elliptic nature of the equation.  A crucial point here is that all the estimates we obtain are independent of $\dd$, which means that the $L^{\infty}$--\textit{a priori} estimate \eqref{eq:Linf} is {\sc only}  used in our analysis through the bound $1-\dd f\leq 1$.  This allows to recover all estimates, in the hard potential setting, for the Landau equation in the classical limit $\dd \rightarrow 0$.

\medskip
\noindent
After uniform bounds have been found, the interesting question is the relaxation rate towards the Fermi-Dirac density profile. We perform a complete and direct study of the spectral properties of the linearized LFD operator rendering an explicit spectral gap for it. This approach requires some smallness assumption on $\dd$, but, it does not rely on a perturbative analysis with respect to linearized Landau operator $\dd=0$.  Only at the level of entropy - entropy production estimates we proceed in a perturbative setting, that is for $\dd$ sufficiently small.  In this respect, it would be quite interesting to obtain a type of Cercignani's conjecture result for the LFD collisional operator analog to that of Landau collisional operator.  The intrinsic difficulty of such result, even in the perturbative setting, is that the entropy production operators of LFD and Landau are essentially different; one admits at least two stationary states, the other only one.  So it happens that the Landau entropy operator is not the classical limit, $\dd\rightarrow0$, of the LFD entropy operator, see Section \ref{sec:entrop} for more details. 

\medskip
\noindent
Finally, combining the estimates on moments, higher norms, entropy dissipation, and spectral analysis, we present an exponential relaxation of general initial density towards the Fermi-Dirac density.  Of course, a central requirement is the non degeneracy of such initial state.  We also point out that a general strong convergence result without rate is presented, very much in the spirit of \cite{LW} for Boltzmann equation (see also \cite{CLR} for a quantum Fokker-Planck equation). This result is free from any condition on $\dd$.  For such result uniform in time estimates are again a key point.

\subsection{Notations} For $s \in \R $ and $ p\geq 1$, we define the space $L^{p}_{s}(\R^3)$ through the norm
$$\displaystyle \|f\|_{L^p_{s}} = \left(\int_{\R^3} \big|f(v)\big|^p \, 
\langle v\rangle^s \, \d v\right)^{1/p},$$
where $\langle v\rangle = \big(1+ |v|^2\big)^{1/2}$, that is, $L^{p}_{s}(\R^3)=\big\{f\::\R^{3} \to \R\;;\,\|f\|_{L^{p}_{s}} < \infty\big\}$. More generally, for any weight function $\varpi\::\:\R^{3} \to \R^{+}$, we define, for any $p \geq 1$
$$L^{p}(\varpi)=\Big\{f\::\:\R^{3} \to \R\;;\,\|f\|_{L^{p}(\varpi)}^{p}:=\int_{\R^{3}}\big|f\big|^{p}\,\varpi\d v  < \infty\Big\}\,.$$
With this notation one can write for example $L^{p}_{s}(\R^{3})=L^{p}\big(\langle \cdot \rangle^{s}\big)$, for $p \geq 1,\,s \geq 0$. We also define, for $k \in \N$,
$$H^{k}_{s}(\R^3)=\Big\{f \in L^{2}_{s}(\R^3)\;;\;\partial_{v}^{\beta}f \in L^{2}_{s}(\R^3) \:\;\forall\, |\beta| \leq k\Big\}$$
with the usual norm,
$$ \|f\|_{H^k_{s}} = \bigg( \sum_{0\leq |\beta| \leq k} 
\int_{\R^3} \big| \partial^{\beta}_v f(v)\big|^2\, 
\langle v\rangle^s \, \d v\bigg)^{\frac{1}{2}},$$
where  $\a=(i_1,i_2,i_3)\in \N^3$, $|\a|=i_1+i_2+i_3$ and 
$\partial^{\beta}_v f =\partial_1^{i_1}\partial_2^{i_2}\partial_3^{i_3} f$.
We shall also use the homogeneous norm 
$$ \|f\|_{\dot{H}^k_{s}} = \bigg( \sum_{ |\a| = k} 
\int_{\R^3} \big| \partial^{\beta}_v f(v)\big|^2\, 
\langle v\rangle^s \, \d v\bigg)^{\frac{1}{2}}.$$
\subsection{Main results}
Before describing in detail the main results of the present contribution, we state the key assumptions on the initial datum for equation \eqref{LFDeq}.
\begin{hyp}\label{ci} The initial datum $f(0)=f_{0}\geq0$ to \eqref{LFDeq} is such that
\begin{equation}\label{hypci}
0<\|f_0\|_{\infty} =:\dd_{0}^{-1} <\infty\qquad \text{ and }  \qquad S_{0}:=\mathcal{S}_{\dd_{0}}(f_0)>0
\end{equation}
where for any $\dd>0$ and $0\leq f \leq \dd^{-1}$ we introduce the Fermi-Dirac entropy as
\begin{equation}\label{eq:FDentro}
\mathcal{S}_{\dd}(f)= - \frac{1}{\dd}\int_{\R^3} \Big[\dd f\log(\dd f)+(1-\dd f)\log (1-\dd f)\Big] \, \d v\,.
\end{equation}
\end{hyp}
\begin{rmq}
Notice that $\mathcal{S}_{\dd}(f) \geq 0$ for any $0 \leq f \leq \dd^{-1}$.  Recalling the definition of the classical Boltzmann entropy $H(f):=\int_{\R^{3}}f\log f\d v$, one has $\mathcal{S}_{\dd}(f)=-\frac{1}{\dd}\big[H(\dd f)+H(1-\dd f)\big].$ 
\end{rmq}
\noindent
For initial datum $f_{0}$ satisfying the condition \eqref{hypci} we always consider the LFD equation \eqref{LFDeq} with quantum parameter $0 < \dd \leq \dd_{0}$ which guarantees estimate \eqref{eq:Linf} and the nonnegativity of the Fermi-Dirac entropy.  {\it A priori} estimates hold for the mass, momentum, and energy of solutions as a consequence of the conservation laws
\begin{eqnarray}
\hspace{-3cm}\forall\; t \geq 0, & \qquad & \int_{\R^3} f(t,v)\, \d v 
= \int_{\R^3} f_{0}(v) \, \d v =:M(f_0), \label{consmass}\\
&& \int_{\R^3} f(t,v)\, v\, \d v = \int_{\R^3} f_{0}(v) \,v\,  \d v \,, 
\label{momentum}\\
& &  \int_{\R^3} f(t,v)\, |v|^2 \, \d v
= \int_{\R^3} f_{0}(v) \,|v|^2 \,  \d v =:E(f_0)\,.\label{consenergy}
\end{eqnarray}
As mentioned, the Cauchy theory for \eqref{LFDeq} has been developed in \cite{bag1}, see Theorem \ref{results} for a precise statement. The question of the smoothness of the solution was then tackled in \cite{Chen10,Chen11}, where, as it occurs for the classical Landau equation, the parabolic nature of \eqref{LFDeq} was exploited. Recall that for the Landau equation, smoothness is immediately produced even if it does not initially exists and, for subsequent time, it is propagated uniformly in time. For the Landau equation, the propagation/appearance  of smoothness is strongly related to the propagation/appearance of moments, see \cite{DeVi1}.  

\medskip
\noindent
For the LFD equation, the analysis of \cite{bag1,Chen10,Chen11} proved the propagation of the moments and the associated smoothness but only on a finite interval of time and their appearance was left open.  An intrinsic difficulty, with respect to the Landau equation, is of course the strongest nonlinearity of the equation, trilinear vs. bilinear, which, at first sight, the additional $L^{\infty}$-bound is not able to compensate.  

\medskip
\noindent
The first main result fills this blank by showing the instantaneous appearance and, then, propagation of both $L^{1}$-moments and $L^{2}$-moments, uniformly in time.  In turn, this results in the appearance and propagation of smoothness.
\begin{theo}\label{smoothness}
Consider $0\leq f_0\in L^{1}_{s_{\gamma}}(\R^{3})$, with $s_{\g}=\max\{2+\tfrac{3\gamma}{2}, 4-\gamma\}$ satisfying \eqref{hypci}.  Then, for any $\dd \in (0,\dd_{0}]$ there exists a weak solution $f$ to \eqref{LFDeq}  such that:

\medskip
\noindent  
(i) \textit{\textbf{(Generation)}} For any $t_0>0$, $k\in\N$, and $s>0$, there exists a constant $C_{t_0}>0$ such that 
$$\sup_{t\geq t_0} \|f(t)\|_{H^k_s}\leq C_{t_0}.$$
The constant $C_{t_0}$ depends, in addition to $t_0$, on $M(f_0)$, $E(f_0)$, $S_{0}$, $k$, $s$, $\gamma$.  In particular,
$$f\in\C^\infty\big([t_0,+\infty); {\mathcal S}(\R^3)\big)\,,\qquad \forall\,t_0>0\,.$$

\noindent
(ii) \textit{\textbf{(Propagation)}} Furthermore, if $\|f_0\|_{H^k_s}<\infty$ and $f_{0}\in L^{1}_{s'}(\R^{3})$ for sufficiently large $s'>0$, the choice $t_0=0$ is valid with constant depending on such initial regularity. 
\end{theo}
\noindent
This result improves the regularity estimates obtained in \cite[Theorem 1.2 and Proposition 2.1]{Chen11} in two directions.  First, the assumptions on the initial datum are relaxed, and second, the estimates are uniform with respect to time and to the quantum parameter.  See the analog result for the Landau equation in \cite[Theorem 5]{DeVi1}.  The novelty of our approach, as mentioned previously, lies in the fact that we treat \emph{simultaneously} the appearance of $L^{1}$ and $L^{2}$ moments by studying the evolution of the functional $\mathcal{E}_{s}(t)=\lm_{s}(t)+\lM_{s}(t)$ where
$$\lm_{s}(t):=\int_{\R^{3}}f(t,v)\langle v\rangle^{s}\d v, \qquad \lM_{s}(t)=\int_{\R^{3}}f(t,v)^{2}\langle v\rangle^{s}\d v, \qquad s \geq 0.$$
This is reminiscent of some recent numerical investigations for the Boltzmann equation \cite[Section 4]{AGT} and appears quite natural since the evolution of $L^{1}$-moments involves $L^{2}$-moments and mixed terms of the form $\lm_{s+\g}(t)\lM_{s}(t)$, see Section \ref{sec:L1L2} for a complete proof. The fact that $L^{1}-L^{2}$ bounds translate in smoothness estimates is a standard procedure developed for the Landau equation. 

\medskip
\noindent
We emphasise the fact that special effort is put in the proof to \emph{not use} the $L^{\infty}$-bound \eqref{eq:Linf} so that the constants involved in   Theorem \ref{smoothness} are all independent of $\dd$. This is the key for the study of the long time behaviour since it is what allows to argue that solutions to \eqref{LFDeq} \emph{stay away} from the degenerate steady state \eqref{eq:dege}, see Corollary \ref{cor:Linfty}.  Having ruled out the steady state \eqref{eq:dege}, we can investigate  the question of the long time behaviour for solutions to \eqref{LFDeq} and prove that the Fermi-Dirac statistics attract the solutions to \eqref{LFDeq} obtained in Theorem \ref{smoothness}.  We can prove this result in a general \emph{non quantitative way} by using suitable compactness argument, see Theorem \ref{theo:conve}. The advantage of such a convergence result is that it applies to any solution to \eqref{LFDeq} and quantum parameter.  The drawback, of course, is that it does not provide any rate of convergence and no indication of the relaxation time.

\medskip
\noindent
In order to quantify the relaxation time, we make a quantitative study of the linearization of \eqref{LFDeq} around equilibrium and we prove the following theorem where $\mathscr{L}_{\dd}$ is the linearized operator around the unique steady state $\M_{\dd}$, see Sections \ref{sec:FDstat} and \ref{sec:line} for precise definitions.

\begin{theo}\label{theo:enlargeintrod} There exists some explicit $\dd^{\dagger} >0$ such that, for any $\dd \in (0,\dd^{\dagger})$  there exists $k_{\dd}^{\dagger}>0$ such that for any $k>k_{\dd}^{\dagger}$, the linearized operator $\mathscr{L}_{\dd}$ around the Fermi-Dirac statistics $\M_{\dd}$ generates a $C_{0}$-semigroup $(\bm{S}_{\dd}(t))_{t\geq0}$ in $L^{2}(\langle \cdot \rangle^{k})$ and for any $g\in L^{2}(\langle \cdot \rangle^{k})$,
\begin{equation*}
\big\|\bm{S}_{\dd}(t)g-\mathbb{P}_{\dd}g\big\|_{L^{2}(\langle \cdot \rangle^{k})} \leq C_{\dd}\,\exp(-\lambda_{\star}\,t)\big\|g-\mathbb{P}_{\dd}g\big\|_{L^{2}(\langle \cdot \rangle^{k})}, \qquad \forall\, t \geq0\,,
\end{equation*}
for some explicit constant $C_{\dd}>0$, and any $\lambda_{\star} \in (0,\lambda_\gamma(\dd))$, with explicit $\lambda_\gamma(\dd)>0$.  The operator $\mathbb{P}_{\dd}$ is the spectral projection on $\mathrm{Ker}\big(\mathscr{L}_{\dd}\big)$.
\end{theo}
\noindent
The fact that the linearization of the LFD equation \eqref{LFDeq} admits a positive spectral and decay in the natural Hilbert space $L^{2}(\mM)$ where $\mM=\M_{\dd}(1-\dd\,\M_{\dd})\,$
is a well-know fact, established in \cite{lemou} by a compactness argument. Notice that the space $L^{2}(\mM)$ is the space in which the linearized operator is symmetric.  We extend such a result of \cite{lemou} in two directions:

\medskip
\noindent
1.  We provide a quantitative estimate of the spectral gap for the linearized operator $\mathscr{L}_{\dd}$ in $L^{2}(\mM).$ This is done in two steps, first considering the case of Maxwell interactions $\g =0$, and then using a comparison argument between the Dirichlet forms corresponding to $\g = 0$ and $\g > 0$. For the study of Maxwell interactions, we adapt the method introduced in \cite{DeVi2, DesvJFA} for the study of the entropy-entropy production inequality for the Landau equation.  As far as we know, this is a novel approach since estimates for the spectral gap for Landau-type equation are usually obtained as the grazing limit of associated estimates for Boltzmann operator, see \cite{bara}.  This latter method seems non trivial and expensive in our context since, to the best of our knowledge, no explicit spectral gap estimates are available for the corresponding Boltzmann equation for Fermi-Dirac particles.

\medskip
\noindent
2.  We extend the spectral gap obtained in the symmetric space $L^{2}(\mM)$ to the space $L^{2}(\langle \cdot \rangle^{k})$ where solutions to \eqref{LFDeq} are known to belong thanks to Theorem \ref{smoothness}.  The technique used is the space enlargement argument introduced in \cite{GMM} and already used for the Landau equation in \cite{Kleber}. The adaptation of the results given in \cite{Kleber} to the LFD equation is straightforward and postponed to the Appendix \ref{app:kleb}.

\medskip

\noindent
It is worthwhile to remark that the restriction to the range of the quantum parameter $\dd \in (0,\dd^{\dagger})$ is of technical nature and it is introduced to make sure that the spectral gap $\lambda_\gamma(\dd)$ has a positive estimation 
$$\lambda_\gamma(\dd) >0\,,\qquad \forall\, \dd \in (0,\dd^{\dagger}].$$
It is not related to any kind of limiting procedure exploiting the existence of a spectral gap for $\dd=0$.  In other words, there is no perturbation argument involved around the classical case.  In fact, we are able to estimate these values as
\begin{equation}\label{estimatedagger}
\varepsilon^{\dagger} \sim E^{3/2}\rho^{-1}\,,\quad\text{and}\quad \lambda_\gamma(\varepsilon)\gtrsim 7.034 \times 10^{-6} \left(\frac{3\gamma E}{20 e}\right)^{\frac{\gamma}{2}}\varrho\,,
\end{equation} 
so that the classical limit $\varepsilon\rightarrow0$ follows.  It was shown in \cite[Proposition 3]{Lu}, that the condition
$$E > \frac{1}{5}\left(\frac{3\dd}{4\pi}\varrho\right)^{2/3},$$
is necessary for convergence to a Fermi-Dirac distribution.  Thus, estimate \eqref{estimatedagger} for $\dd^{\dagger}$ is natural.  The estimation of the spectral gap $\lambda_\gamma(\dd)$ is related to a Poincar\'{e}'s inequality constant \cite[Corollary 3.4]{lemou} and, in our estimation, likely far from optimal.   

\medskip
\noindent
Based on Theorems \ref{smoothness} and \ref{theo:enlargeintrod}, we are able to prove the following relaxation theorem.
\begin{theo}\label{theo:converg}
Consider $0\leq f_0\in L^{1}_{s_{\gamma}}(\R^{3})\cap L^2_{k}(\R^3)$, with $s_{\g}=\max\{\tfrac{3\gamma}{2}+2, 4-\gamma\}$ and $k>k_{\dd}^{\dagger}$, satisfying \eqref{hypci}. Let $\dd \in (0,\dd_{0}]$ and let $f$ be a weak solution to \eqref{LFDeq} given in Theorem \ref{smoothness}. Then, there exists $\dd^{\ddagger} \in (0,\dd^{\dagger})$ such that for any $\dd \in (0,\dd^{\ddagger})$
$$\|f(t)-\M_{\dd}\|_{L^{1}_{2}} \leq C_{0}\exp(-\lambda_{\star}t), \qquad \forall\, \lambda_{\star} \in \big(0,\lambda_\gamma(\dd)\big),\;t \geq 0\,,$$
where $\lambda_\gamma(\dd) >0$ is the explicit spectral gap of $\mathscr{L}_{\dd}$ given by Theorem \ref{theo:enlargeintrod}. The constant $C_{0}>0$ depends also on $M(f_0)$, $E(f_0)$, $S_0$, and $\g$, but not on $\dd.$
\end{theo} 

\noindent
This theorem is proved using standard combination of the close to equilibrium result Theorem \ref{theo:enlargeintrod} and the entropy-entropy production estimates. The proof of such entropy-entropy production estimates are technical because of the nature of Fermi-Dirac entropy $\mathcal{S}_{\dd}$.  It is in the entropy-entropy production estimates where a perturbation argument is used, exploiting the entropy-entropy production estimates available for the Landau equation and showing that, for $\dd$ sufficiently small, the entropy production associated to both problems are close. We remark that, even if the Fermi-Dirac entropy $\mathcal{S}_{\dd}(f)$ is not continuous with respect to $\dd \simeq 0$, that is
$$\lim_{\dd \to 0^{+}}\mathcal{S}_{\dd}(f) \neq H(f),$$
the \emph{relative Fermi-Dirac entropy} around $\M_{\dd}$ is continuous, i.e. 
$\lim_{\dd \to 0^{+}}\mathcal{S}_{\dd}(f)-\mathcal{S}_{\dd}(\M_{\dd})=H(f) - H(\M_{0})\,,$
where $\M_{0}$ is the Maxwellian distribution with same mass, momentum, and energy than $f$ and $\M_{\dd}$. Making this continuity quantitative and combining it with the close to equilibrium study lead to Theorem \ref{theo:converg}. Again, at the limit $\dd \to 0$, our result allows to recover the exponential convergence to equilibrium for the classical Landau equation, see \cite{DeVi2,Kleber}.

\subsection{Organization of the paper} The paper is organised as follows.  In Section \ref{sec:known}, we recall some known results about \eqref{LFDeq} such as well posedness and existence of stationary solution.  An equivalent formulation of \eqref{LFDeq} as a nonlocal nonlinear parabolic equation, see \eqref{LFD}, is presented and a proof of the uniform ellipticity of the diffusion matrix associated to such formulation is given.  This is in the line with \cite{DeVi1,bag1}, but a careful analysis is required to prove that the ellipticity of the matrix is uniform with respect to the quantum parameter $\dd$. Section \ref{sec:regula} is devoted to the proof of Theorem \ref{smoothness} and Corollary \ref{cor:Linfty}.  We prove the appearance and propagation of $L^{1}$ and $L^{2}$-moments and, then, deduce the smoothness estimates.  In Section \ref{sec:convno}, the non quantitative convergence result for solution to equation \eqref{LFDeq} is provided.  The spectral analysis of the linearized operator $\mathscr{L}_{\dd}$ and the associated semigroup is performed in Section \ref{sec:line}.  Section \ref{sec:entrop} combines the linearized analysis with new entropy-entropy production estimates resulting in a proof of Theorem \ref{theo:converg}. The paper ends with two appendices:  Appendix \ref{app:FD} is devoted to quantitative bounds on the Fermi-Dirac statistics, and Appendix \ref{app:kleb} presents the technicalities related to extension of the results given in \cite{Kleber} about enlargement of the space for the linearized study.
\section{Cauchy theory}\label{sec:known}

\subsection{Uniform ellipticity of the diffusion matrix} It is convenient to write \eqref{LFDeq} as a nonlinear parabolic equation. More precisely, for $ (i,j) \in  [\hspace{-0.7mm} [  1,3 ]\hspace{-0.7mm}] ^2 $, define
\begin{equation*}
\left\{
\begin{array}{rcl}
a(z) & = & \left(a_{i,j}(z)\right)_{i,j} \quad \mbox{ with }
\quad a_{i,j}(z) 
= |z|^{\g+2} \,\left( \delta_{i,j} -\frac{z_i  z_j}{|z|^2} \right),\medskip\\
 b_i(z) & = & \sum_k \partial_k a_{i,k}(z) = -2 \,z_i \, |z|^\g,  \medskip\\
 c(z) & = & \sum_{k,l} \partial^2_{kl} a_{k,l}(z) = -2 \,(\g+3) \, |z|^\g. \\
\end{array}\right.
\end{equation*}
For any $f \in L^{1}_{2+\g}(\R^{3})$, we define then the matrix-valued mappings $\bm{\sigma}[f]$ and $\bm{\Sigma}[f]$ given by
$$\bm{\sigma}[f]=\big(\bm{\sigma}_{ij}[f]\big)_{ij}:=\big(a_{ij}\ast f\big)_{ij}, \qquad \qquad \bm{\Sigma}[f]=\bm{\sigma}[f(1-\dd\,f)].$$
In the same way, we set $\bm{b}[f]\::\:v \in \R^{3} \mapsto \bm{b}[f](v) \in \R^{3}$ given by 
$$\bm{b}_{i}[f](v)=\big(b_{i} \ast f\big)(v), \qquad \forall\, v \in \R^{3},\qquad i=1,2,3.$$
We also introduce
$$\bm{B}[f]=\bm{b}[f(1-\dd\,f)], \qquad \text{ and } \qquad  \bm{c}[f]=c \ast f.$$
Notice that
\begin{equation}
\left\{
\begin{array}{ccl}
\label{eq:bc}
\left|\bm{c}[f](v)\right| \; \leq &8 \langle v\rangle^{\gamma}\|f\|_{L^{1}_{\gamma}} &\leq \; 8  \langle v\rangle^{\gamma}\|f\|_{L^{1}_{2}} \medskip\\
\left|\bm{b}[f](v)\right| \; \leq &2 \langle v\rangle^{\gamma+1}\|f\|_{L^{1}_{\gamma+1}} &\leq \; 2 \langle v\rangle^{\gamma+1}\|f\|_{L^{1}_{2}}, \qquad v \in \R^{3} \\
\end{array}\right.
\smallskip
\end{equation}
since $0 < \gamma \leq 1$.  With these notations, the LFD equation can then be written 
alternatively under the form
\begin{equation}\label{LFD}
\left\{
\begin{array}{ccl}
\;\partial_{t} f &= &\grad \cdot \big(\,\bm{\Sigma}[f]\, \grad f
- \bm{b}[f]\, f(1-\dd f)\big),\medskip\\
\;f(t=0)&=&f_{0}\,.
\end{array}\right.
\end{equation}
A key ingredient in the well posedness of the LFD equation \eqref{LFD}, as shown in \cite{bag1}, is the ellipticity of the matrix function $\bm{\Sigma}[f]$.  Recall that the analysis of \cite{bag1} is performed for $\dd=1$, so, we need to adapt the proof of \cite{bag1} to the case $\dd>0$.  Furthermore, we aim to prove that such ellipticity is uniform in terms of the parameter $\dd$.  We will need some preliminary results which are \emph{a priori} estimates for the entropy $\mathcal{S}_{\dd}$.
\begin{defi} 
Fix $\dd>0$, $f_0\in L^{1}_{2}(\R^{3})$ satisfying $0\leq f_0\leq \dd^{-1}$.  We say that $f \in \mathcal{Y}_{\dd}(f_{0})$ if $f\in L^{1}_{2}(\R^{3})$ satisfies $0\leq f\leq \dd^{-1}$ and 
$$M(f)=M(f_{0}), \qquad E(f)=E(f_{0}), \quad \text{ and } \quad \mathcal{S}_{\dd}(f) \geq \mathcal{S}_{\dd}(f_{0}).$$
\end{defi}
\begin{lem}\label{L1unif}
Fix $\dd > 0$, $f_0\in L^{1}_{2}(\R^{3})$ satisfying $0\leq f_0\leq \dd^{-1}$, and let $f \in \mathcal{Y}_{\dd}(f_{0})$.  Then, it holds that
\begin{equation}\label{upperentropy}
0 \leq \mathcal{S}_{\dd}(f) \leq 80\big(|\log \dd| + \log R\big)\int_{|v|\leq R}f(1-\dd f)\, \d v 
+ \frac{C(f_0)}{R} + |\log\dd|\frac{E(f_{0})}{R^{2}}\,,\quad \forall \,R>1\,.
\end{equation}
The constant $C(f_0)$ depends only on the energy $E(f_{0})$.
\end{lem}
\begin{proof}
For any $R > 1$, we have
\begin{multline}\label{eq:Saf}
0 \leq \mathcal{S}_{\dd}(f) \leq \int_{|v|\leq R}f|\log \dd f| \d  v + \frac{1}{\dd}\int_{|v|\leq R}(1-\dd f)|\log(1-\dd f)|\d  v  \\
+ \int_{|v|>R}f|\log \dd f|\d  v + \frac{1}{\dd}\int_{|v|>R}(1-\dd f)|\log(1-\dd f)|\d  v\,.
\end{multline}
First, we can estimate the last integral, see \cite[Eqs. (3.6)]{bag1}, to obtain
\begin{equation}\label{eq1:v>R}
\frac{1}{\dd}\int_{|v|>R}(1-\dd f)|\log(1-\dd f)|\d  v 
\leq \frac{3E(f_0)}{R^{2}}\,.
\end{equation}
Second, because $0\leq \dd f \leq 1$, 
\begin{multline*}
 \int_{|v|>R}f|\log \dd f|\d  v=- \int_{|v|>R}f \log \dd f\d  v 
 =  -\log \dd\,\int_{|v|>R}f\,\d  v \\
- \int_{\{|v|>R\}\cap\{f\leq 1\}} f \log f \,\d  v 
- \int_{\{|v|>R\}\cap \{f > 1\}} f \log f\, \d  v\,.
\end{multline*}
Consequently, dismissing the last integral which is nonpositive, we estimate the second integral as in \cite[Eqs. (3.5)]{bag1}, to get
\begin{eqnarray}
\int_{|v|>R}f|\log \dd f|\d  v 
& \leq & |\log \dd|\;\frac{E(f_0)}{R^{2}} + \int_{\{|v|>R\}\cap\{f\leq 1\}} f|\log f|\,\d  v \nonumber \\
& \leq &|\log \dd|\;\frac{E(f_0)}{R^{2}} + C_{0}\;\frac{E(f_0)^{{4}/{5}}}{R} \,,\qquad\qquad C_{0}>0.\label{eq2:v>R}
\end{eqnarray}
Then, to estimate the first integral we use the following slight improvement of \cite[Eqs. (3.2) and (3.3)]{bag1} which is valid for any $\delta \in (0,1/2]$
\begin{equation*}
\begin{split}
|\log r| &\leq 2|\log \delta|\,(1-r) \qquad \forall \,r \in (\delta,1)\,,\\ 
|\log(1-r)| &\leq 2|\log\delta|\,r \qquad \qquad \forall\, r \in (0,1-\delta)\,.
\end{split}
\end{equation*}
Then, one proceeds as in \cite[Lemma 3.1]{bag1} to find that for any $\theta \in (0,1)$ 
\begin{multline*}
\int_{|v|\leq R}f|\log \dd f|\d  v \leq 2|\log \delta|\int_{|v|\leq R}f(1-\dd f)\d  v + \frac{\delta^{\theta}}{\dd}\int_{|v|\leq R}(\dd f)^{1-\theta}|\log \dd f|\d  v\\
\leq  2|\log \delta|\int_{|v|\leq R}f(1-\dd f)\d  v + C(\theta)\frac{\delta^{\theta}}{\dd}R^{3}\,,
\end{multline*}
for some positive constant $C(\theta) >0$.  For the last inequality we used that the mapping $r \in [0,1] \mapsto r^{1-\theta}|\log r|$ is bounded. In the same way
\begin{equation*}
\frac{1}{\dd}\int_{|v|\leq R}(1-\dd f)|\log(1-\dd f)|\d  v \leq 2 |\log \delta|\int_{|v|\leq R}f(1-\dd f)\d  v + C(\theta)\frac{\delta^{\theta}}{\dd}R^{3}
\end{equation*}
valid for any $\delta \in (0,1/2]$, $\theta \in (0,1)$, and $R >1$. The value of $\theta $ is irrelevant and we fix it as $\theta=\tfrac{1}{5}$ for instance. Then, choosing $\delta >0$ such that $\frac{\delta^{\theta}}{\dd}R^{3}=\frac{1}{R}$, that is $\delta=\dd^{5}R^{-20}$, we obtain the existence of a positive constant $C_{1}>0$ such that
\begin{align}\label{eq:v<R}
\begin{split}
\int_{|v|\leq R}f|\log \dd f| \d  v + \frac{1}{\dd}\int_{|v|\leq R}&(1-\dd f)|\log(1-\dd f)|\d  v \\
&\leq 4|\log \delta|\int_{|v|\leq R}f(1-\dd f)\d  v + \frac{C_{1}}{R}\,.
\end{split}
\end{align}
Plugging \eqref{eq1:v>R}, \eqref{eq2:v>R} and \eqref{eq:v<R} in \eqref{eq:Saf}, we obtain \eqref{upperentropy} with $C(f_0):=C_{1}+C_0\,E(f_0)^{4/5}+3E(f_0)$.
\end{proof}
\begin{lem}\label{L2unif}
Let $0\leq f_0\in L^{1}_{2}(\R^{3})$ be fixed and bounded satisfying \eqref{hypci}.  Then, for any $\dd \in (0,\dd_{0}]$, $f \in \mathcal{Y}_{\dd}(f_{0})$, it holds that
\begin{equation}\label{e0}
\inf_{0<\dd\leq \dd_{0}}\int_{|v|\leq R(f_0)} f(1-\dd f)\, \d v \geq \eta(f_0)>0\,,
\end{equation}
for some $R(f_0)>0$ and $\eta(f_0)$ depending only on $\varrho(f_{0}), E(f_{0})$ and $H(f_{0})$ but not on $\dd$.
\end{lem}
\begin{proof}
Since   $f \in \mathcal{Y}_{\dd}(f_{0})$, we have 
\begin{equation*}
-\int_{\mathbb{R}^{3}} f_0 \log f_0\, \d v + |\log\dd|\,M(f_0) \leq  \mathcal{S}_{\dd}(f_{0}) \leq \mathcal{S}_{\dd}(f).
\end{equation*}
We know, thanks to Lemma \ref{L1unif}, that
\begin{equation}\label{e1}
S_{\dd}(f)\leq 80(|\log \dd| + \log R)\int_{|v|\leq R} f(1-\dd f)\d v + \frac{C(f_0)}{R} + |\log \dd| \;\frac{E(f_0)}{R^{2}}\,,\quad R>1\,.
\end{equation}
We conclude that
\begin{align*}
-\int_{\mathbb{R}^{3}} f_0 \log f_0\,\d v \leq 80(|\log \dd| &+ \log R)\int_{|v|\leq R} f(1-\dd f) \d v \\
& + \frac{C(f_0)}{R} - |\log \dd| \Big( M(f_0) - \frac{E(f_0)}{R^{2}}\Big)\,.
\end{align*}
Choosing $R := R_{1}(f_0):=\max\big\{1,\sqrt{2E(f_0)/M(f_{0})}\big\}$, we are led to
\begin{align*}
-\int_{\mathbb{R}^{3}} f_0 \log f_0\,\d v - &\frac{C(f_0)}{R_{1}(f_0)} + |\log\dd|\; \frac{M(f_0)}{2}\\
&\leq 80\big(|\log \dd| + \log R_{1}(f_0)\big)\int_{|v|\leq R_{1}(f_0)} f(1-\dd f)\d v\,.
\end{align*}
Thus, for any $0<\dd<\dd(f_0)$ where 
\begin{equation*}
\log\left(\frac{1}{\dd(f_0)}\right):=2\,\dfrac{\big|H(f_{0})\big| + 2\,C(f_0)/R_{1}(f_0)}{M(f_0)}\,,\end{equation*}
we have
$$\int_{|v|\leq R_{1}(f_0)} f(1-\dd f) \d v\geq \eta_{1}(f_{0})$$
with 
$$\eta_{1}(f_{0}):=\inf_{0<\dd\leq\dd(f_0)}\dfrac{-H(f_0) - C(f_0)/R_{1}(f_0) + |\log\dd| M(f_0)/2}{80(|\log \dd| + \log R_{1}(f_0))} >0.$$ 
Now, for $\dd<\dd_{0}$ it follows that $\mathcal{S}_{\dd}(f_0)>0$ since $0< f_0\ < \dd^{-1}$. Since $\mathcal{S}_{\dd_{0}}(f_{0}) > 0$ due to the continuity of the map $\dd \in (\dd(f_{0}),\dd_{0}) \mapsto \mathcal{S}_{\dd}(f_0)$ one is led to
\begin{equation*}
\eta_{2}(f_0):=\inf_{\dd(f_0)\leq \dd \leq \dd_{0}}S_{\dd}(f_0)>0\,.
\end{equation*}
Recalling estimate \eqref{e1} and using the fact that
$$|\log \dd| \leq |\log \dd(f_0)| + |\log \dd_0|=:\delta_0 \quad \text{for}\quad \dd(f_0) \leq \dd \leq \dd_{0}\,,$$
it holds that
\begin{align*}
\eta_{2}(f_0)\leq 80\big(\delta_0 + \log R\big)\int_{|v|\leq R} f(1-\dd f)\d v + \frac{C(f_0)}{R} + \delta_0\frac{E(f_0)}{R^{2}}\,.
\end{align*}
Thus, choosing $R:= R_{2}(f_0)=\frac{4}{\eta_{2}(f_0)}\max\Big\{C(f_0),\sqrt{\delta_0\,\eta_{2}(f_0)\,E(f_0)}\Big\}$, it follows that
\begin{align}\label{e2}
\begin{split}
0<\eta_{3}(f_0):&=\frac{\eta_{2}(f_0)}{160( \delta_0 + \log R_{2}(f_0))}\\
&\leq \int_{|v|\leq R_{2}(f_0)} f(1-\dd f) \d v\,,\quad \text{for}\quad\dd(f_0)\leq \dd \leq \,\dd_{0}.
\end{split}
\end{align}
In this way, choosing $
R(f_0):=\max\big\{R_{1}(f_0), R_{2}(f_0)\big\}$ and 
$\eta(f_0):=\min\big\{\eta_{1}(f_0), \eta_{3}(f_{0})\big\}\,$, estimate \eqref{e0} follows.\end{proof}

\begin{lem}
Let $0\leq f_0\in L^{1}_{2}(\R^{3})$ be fixed and bounded satisfying \eqref{hypci}.  Then, for any $\delta >0$ there exists $\eta(\delta)>0$ depending only on $M(f_0)$, $E(f_0)$ and $H(f_{0})$ such that,  for any $\dd \in (0,\dd_{0}]$, $f \in \mathcal{Y}_{\dd}(f_{0})$, and measurable $A\subset \R^3$, 
\begin{equation}\label{Lem6DV}
|A|\leq \eta(\delta) \Longrightarrow \int_A f(1-\dd f)\, \d v \leq \delta.
\end{equation}
\end{lem}
\begin{proof} For $f \in \mathcal{Y}_{\dd}(f_{0})$, expanding the inequality $\mathcal{S}_{\dd}(f_{0}) \leq \mathcal{S}_{\dd}(f) $ and using the conservation of mass, we obtain that 
\begin{multline*}
  -\int_{\R^3} f_0\log f_0\, \d v + \frac{1}{\dd} \int_{\R^3} (1-\dd f_0) |\log(1-\dd f_0)| \, \d v  \\
\leq -\int_{\R^3} f\log f \, \d v + \frac{1}{\dd} \int_{\R^3} (1-\dd f) |\log(1-\dd f)| \, \d v.
\end{multline*}
For $r\in(0,1)$ one has  $0\leq (1-r)|\log(1-r)| \leq r$.   As a consequence of these two observations and the conservation of mass it follows that 
$$\int_{\R^3} f\log f \, \d v \leq \int_{\R^3} f_0\log f_0\, \d v +M(f_0)\,.$$
The result then follows from \cite[Lemma 6]{DeVi1}, using the bound $0\leq 1-\dd f\leq 1$.
\end{proof}
\noindent
Since entropy increases along the flow of the LFD equation \eqref{LFD}, previous lemmata will apply to the solution $f=f(t,v)$ associated to the initial datum $f_{0}$.

\smallskip
\noindent
In addition, the uniform ellipticity of the matrix function $\bm{\Sigma}[f]$ is obtained by adapting the proof of \cite[Proposition 4]{DeVi1} with the help of previous lemmas.

\begin{prop}\label{diffusion}
Let $0\leq f_0\in L^{1}_{2}(\R^{3})$ be fixed and satisfying \eqref{hypci}. Then, there exists a 
positive constant $K_{0} > 0$ depending on $M(f_{0})$, $E(f_{0})$, and $S_{0}$, such that
$$\forall\, v,\, \xi \in \R^3, \qquad 
\sum_{i,j} \, \bm{\Sigma}_{i,j}[f](v) \, \xi_i \, \xi_j 
\geq K_{0} \langle v \rangle^{\g} \, |\xi|^2$$
holds for any $\dd >0$ and $f \in \mathcal{Y}_{\dd}(f_{0})$.  Recall that ${\bm{\Sigma}_{i,j}[f]}= a_{i,j} \ast \big( f(1-\dd f)\big)$.
\end{prop}

\subsection{Well-posedness and  Fermi-Dirac statistics}\label{sec:FDstat} We recall here the well-posedness result established in \cite{bag1} that we reformulate to take into account the quantum parameter. 
\begin{theo}\label{results} Consider an initial datum $f_{0}$ satisfying \eqref{hypci} and $\dd \in (0,\dd_{0}]$. Assume further 
that $f_{0}\in L^1_{s_0}(\R^3)$ for some $s_0>2$. Then, there 
exists a weak solution $f$ to \eqref{LFDeq} satisfying 
\eqref{consmass}, \eqref{momentum}, \eqref{consenergy} and 
$$
f(1- \dd f) \in L^1_{\mathrm{loc}}\big(\R_+; L^1_{s_0+\g}\big(\R^3\big)\big);\qquad
f \in L_{\mathrm{loc}}^\infty \big(\R_+; L^1_{s_0}\big(\R^3\big)\big) 
\cap L_{\mathrm{loc}}^2 \big(\R_+; H^1_{s_0}\big(\R^3\big)\big).
$$
If we also assume that $s_0 \geq 2+\g$, then the entropy $t \geq 0\longmapsto \mathcal{S}_{\dd}(f)(t)$ is 
a non-decreasing function and 
$$ S_{\dd}(f_{0}) \leq  \mathcal{S}_{\dd}(f)(t) \leq E(f_{0}) + \pi^{3/2}-\varrho \log \dd, \qquad \forall\, t \in \R_{+}.$$
Moreover, for $s_0>4\g+11$, such a solution is unique. 
\end{theo}
\noindent
It will be noticed later that the upper bound in $ \mathcal{S}_{\dd}(f)(t)$ is far from optimal in terms of $\dd$.  Due to the conservation laws \eqref{consmass}, \eqref{momentum}, \eqref{consenergy}, we will assume in the sequel that $f_0$ satisfies Assumptions \ref{ci} and 
\begin{equation}\label{hyp:mass}
\int_{\R^{3}}f_{0}(v)\left(\begin{array}{c}1 \\v \\|v|^{2}\end{array}\right)\d v=\left(\begin{array}{c}\varrho \\0 \\3\varrho\,E\end{array}\right)
\end{equation}
where $\varrho, E >0$ are given.  An important observation is given in \cite[Propositions 3]{Lu} where it is shown that the condition
\begin{equation}\label{eq:bound}
\dd  < \dfrac{4\pi}{3}\dfrac{(5E)^{\frac{3}{2}}}{\varrho}
\end{equation}
is necessary and sufficient to associate a \emph{unique} Fermi-Dirac statistics 
$$\M_{\dd}(v)=\frac{a\exp(-b|v|^{2})}{1+\dd\,a\exp(-b|v|^{2})}$$
with $a=a_{\dd} >0$ and $b=b_{\dd} >0$, such that
\begin{equation}\label{FD_moments}
\int_{\R^{3}}\M_{\dd}(v)\left(\begin{array}{c}1 \\v \\|v|^{2}\end{array}\right)\d v=\left(\begin{array}{c}\varrho \\0 \\3\varrho\,E\end{array}\right).
\end{equation}
Provided that the initial datum satisfies Assumptions \ref{ci}, \cite[Theorem 3]{bag2} shows that $\M_{\dd}$ is the unique steady state to \eqref{LFDeq} satisfying \eqref{FD_moments}. Moreover, under Assumptions \ref{ci} and \eqref{hyp:mass}, it follows from  \cite[Propositions 4]{Lu} that \eqref{eq:bound} holds, therefore, the solution $f(t,v)$ must converge to a Fermi-Dirac statistic in this regime.  In order to make explicit estimations for \textit{exponentially fast} convergence a stronger assumption on $\dd$ is needed in the form of $\dd<c\, E^{3/2}\rho^{-1}$ with constant $c<\frac{4\pi}{3}5^{\frac{3}{2}} \approx 46.832$.

\section{Moments and regularity}\label{sec:regula}
The final goal of this section is to prove Theorem \ref{smoothness}. This will be done by improving the approach of \cite[Theorem 5]{DeVi1} and \cite{Chen11} for which regularity estimates are deduced from estimates on $L^{1}$ and $L^{2}$ norms.  The novelty here consists in treating the propagation of these norms as a whole to be able to close the energy estimate. 
\subsection{$L^{1}$ and $L^{2}$ norms}\label{sec:L1L2}
The main result of this section shows the instantaneous appearance of both $L^{1}$ and $L^{2}$ norms as well as uniform estimates for such norms with respect to both time and the quantum parameter $\dd >0$. 
\begin{theo}\label{moments}
Consider $0\leq f_0\in L^{1}_{s_{\gamma}}(\R^{3})$, with $s_{\gamma}=\max\big\{\tfrac{3\gamma}{2}+2, 4-\gamma\big\}$, satisfying \eqref{hypci}.  Let $f=f(t,v)$ be a weak solution to the LFD equation given by Theorem \ref{results}.

\medskip
\noindent
(i) Then, for any $s\geq0$
$$\int_{t_0}^{T} \int_{\R^3} |\nabla f(t,v)|^{2}\langle v\rangle^{s+\g} \, \d v \, \d t <+\infty\,,\qquad \forall\,T>t_0>0\,.$$

\medskip
\noindent
(ii) There exists some positive constant $C_{t_0}$ depending on $M(f_0)$, $E(f_0)$, $S_{0}$, $s$ and $t_0$, but not on $\dd$, such that 
\begin{equation}\label{unif_L2}
\int_{\R^3} \big( f(t,v)+  f^2(t,v) \big)\, \langle v\rangle^{s} \, \d v \leq C_{t_0}\,,\qquad \forall\,s\geq0\,,\;t\geq t_0>0\,.
\end{equation}
Moreover, if
\begin{equation*}
\int_{\R^3} \big( f(0,v)+  f^2(0,v) \big)\, \langle v\rangle^{s} \, \d v < \infty\,,
\end{equation*}
then $t_0=0$ is a valid choice in the estimate \eqref{unif_L2} with constant depending on such initial quantity.

\end{theo}
\begin{rmq}
Theorem \ref{moments} is the analog to \cite[Theorem 3]{DeVi1} for the Landau equation and constitutes a noteworthy improvement of \cite[Lemma 3.2]{bag1}.
\end{rmq}
\noindent
We recall that, given $f(t,v)$ solution to \eqref{LFD}, we write
$$\lm_{s}(t)=\int_{\R^{3}}f(t,v)\langle v\rangle^{s}\d v, \qquad \lM_{s}(t)=\int_{\R^{3}}f^{2}(t,v)\langle v\rangle^{s}\d v, \qquad s \in \R.$$
We introduce also
$$\lD_{s}(t)=\left\|\nabla \left(f(t,\cdot)\langle v\rangle^{\frac{s}{2}}\right)\right\|_{2}^{2}, \qquad t \geq 0.$$
\begin{prop}\label{lem5}
Consider $0\leq f_0\in L^{1}_{s_0}(\R^{3})$, for some $s_0>2$, satisfying \eqref{hypci}.  Let $f=f(t,v)$ be a weak solution to \eqref{LFD} that preserves mass and energy. Then, for some constants $C_{s,1}$ and $K_s$ depending only on $M(f_0)$, $E(f_0)$, $s$, it holds
\begin{equation}\label{ineq_L1s}
\frac{\d}{\d t} \lm_{s}(t)  + K_s \lm_{s+\g}(t) \leq  K_{s} \lM_{s+\g}(t) + C_{s,1} \lm_{s}(t)\,,\qquad s>2\,.
\end{equation}
Also, for some constants $C_{s,2},C_{s,3}>0$ depending only on $\gamma$, $s$, $M(f_0)$, $E(f_0)$ and $S_{0}$, it follows that 
\begin{equation}\label{ineq_L2s}
\frac{1}{2}\frac{\d}{\d t} \lM_{s}(t) + K_{0} \lD_{s+\g}(t) \leq  C_{s,2} \lM_{s+\g}(t) +  C_{s,3} \lm_{2+\g}(t) \lM_{s+\g-2}(t), 
\end{equation}
where $K_{0}$ is given by Proposition \ref{diffusion}.  We remark that all constants are independent of $\dd>0$.
\end{prop}

\begin{proof}
Let $\Phi$ be a smooth convex function on $\R_+$.  Let us proceed in the spirit of \cite[Lemma 3.2]{bag1} by multiplying \eqref{LFD} by $\Phi \big(|v|^2\big)$ and integrating over $\R^3$ to obtain that  
$$\frac{\d}{\d t} \int_{\R^3} f(t,v)\, \Phi \big(|v|^2\big) \,\d v =  4 \iint_{\R^3\times \R^3}  f \fet (1-\dd\fet) \, |v-\vet|^\g \,\L^\Phi(v,\vet) \, \d v \, \d\vet\,,$$
where 
$$\L^\Phi(v,\vet) \leq \big(|\vet|^2 - (v\cdot \vet)\big)\big(\Phi'\big(|v|^2\big) - \Phi'\big(|\vet|^2\big)\big) + |v|^{2}|\vet|^{2} \Phi''\big(|v|^2\big)\,.\smallskip$$
Let $\Phi(r)=(1+r)^{s/2}$, with $s>2$. Since $(v \cdot \vet) \leq \langle v\rangle \langle\vet\rangle$, 
we deduce (with the notation $\L^s=\L^\Phi$) that
\begin{equation}\label{propagation}
\L^{s}(v,\vet) \leq \frac{s}{2} \bigg[\frac{s}{2}\,\langle v \rangle^{s-2}\,\langle\vet\rangle^2 -\langle\vet\rangle^{s} +\langle\vet\rangle^{s-2}\\
 +\langle v\rangle \langle\vet\rangle^{s-1}+ \langle\vet\rangle \langle v\rangle^{s-1}  \bigg].
\end{equation}
Since $s-1>1$, we use Young's inequality to obtain
$$ x^{s-2} \,y^2 
 =    x^{s-2} \, y^{(s-2)/(s-1)} \, y^{s/(s-1)}
 \leq  \frac{s-2}{s-1}\, x^{s-1}\,y + \frac{1}{s-1}\, y^{s}. $$
Substituting this inequality for $x=\langle v\rangle$, $y=\langle\vet\rangle$ 
into \eqref{propagation} yields 
$$\L^{s}(v,\vet) \leq \frac{s}{2} \left(\frac{s+2}{2}\,\langle v\rangle^{s-1}\,\langle\vet\rangle 
+\langle v\rangle \langle\vet\rangle^{s-1}+ \langle\vet\rangle^{s-2} - \frac{s-2}{2(s-1)}\langle\vet\rangle^{s}\right).$$
Since $|v-\vet|^\g \leq \langle v\rangle^\g \langle\vet\rangle^\g$ and 
$|v-\vet|^\g\geq 2^{-\g/2}\langle\vet\rangle^\g-\langle v\rangle^\g$, for $0<\gamma\leq 1$, we finally obtain
\begin{multline*}
\frac{\d}{\d t} \int_{\R^3} f(t,v)\, \langle v\rangle^{s} \,\d v 
+ 2^{-\g/2} \:\frac{s(s-2)}{s-1}\, M(f_0) \int_{\R^3} 
\, \langle\vet\rangle^{s+\g}\, \fet (1-\dd\fet) \, \d\vet\\
\leq  \:\frac{s(s-2)}{s-1} \iint_{\R^3\times \R^3}  f \fet (1-\dd\fet)\langle v\rangle^\g \langle\vet\rangle^{s}\, \d v \, \d\vet \\
+ s \iint_{\R^3\times \R^3}  f \fet \, 
\Big[(s+2)\,\langle v\rangle^{s-1+\g} \langle\vet\rangle^{1+\g} +2\langle v\rangle^{1+\g} \langle\vet\rangle^{s-1+\gamma}+2\langle v\rangle^\g \langle\vet\rangle^{s-2+\gamma}\Big] \, \d v \, \d\vet.
\end{multline*}
Since $0\leq f\leq \dd^{-1}$ and $0<\g\leq 1$, there exists a constant $C_{s,1}>0$ depending only on $s$, $M(f_0)$ and $E(f_0)$ such that 
\begin{equation*}
\frac{\d}{\d t} \int_{\R^3} f\, \langle v\rangle^{s} \,\d v  + K_s\int_{\R^3} \langle\vet\rangle^{s+\g}\, \fet(1-\dd\fet)  d\vet\leq C_{s,1} \int_{\R^3}  f \langle v\rangle^{s} \d v,
\end{equation*}
where $K_s:=2^{-\g/2} \:\frac{s(s-2)}{s-1} \,M(f_0)$. This proves \eqref{ineq_L1s}.

\medskip
\noindent
Let us now show \eqref{ineq_L2s}.  Multiplying \eqref{LFD} by $f\langle v\rangle^{s}$ and integrating over $\R^3$ lead to 
\begin{multline*}
\frac{1}{2}\frac{\d}{\d t} \int_{\R^3} f^2(t,v) \langle v\rangle^{s} \d v 
=  -\int_{\R^3}  \langle v\rangle^{s}(\bm{\Sigma}[f] \grad f) \cdot \grad f  \d v
-  \, s \int_{\R^3} f  \langle v\rangle^{s-2}(\bm{\Sigma}[f] \grad f)\cdot v \d v \\
+ \int_{\R^3}\langle v\rangle^{s} f(1-\dd f) \bm{b}[f] \cdot \grad f  \d v
+ \,s \int_{\R^3} (\bm{b}[f] \cdot v )\, f^2(1-\dd f)\langle v\rangle^{s-2} \d v.
\end{multline*}
Using the uniform ellipticity of the diffusion matrix $\bm{\Sigma}[f]$, recall Proposition \ref{diffusion}, we deduce that
$$\int_{\R^3}  \langle v\rangle^{s}(\bm{\Sigma}[f] \grad f) \cdot \grad f  \d v
 \geq K_{0} \int_{\R^{3}}\langle v\rangle^{s+\g}\,\left|\grad f\right|^{2}\d v\,.$$
Also, using \eqref{eq:bc} and the fact that $0 \leq 1-\dd f\leq 1$, we get
$$\int_{\R^3} (\bm{b}[f] \cdot v) \, {f}^2 (1-\dd f)  \langle v\rangle^{s-2}  \d v
\leq  2 \,\lm_{2}(t)\, \lM_{s+\g}(t),$$
and 
$$
\int_{\R^3}\langle v\rangle^{s} f(1-\dd f) \bm{b}[f] \cdot \grad f  \d v
 =  - \int_{\R^3} \left( \frac{1}{2} f^2 - \frac{\dd}{3} f^3 \right)
\grad \cdot \Big(\bm{b}[f] \langle v\rangle^{s}\Big) \d v 
 \leq   C \, \lm_{2}(t)\, \lM_{s+\g}(t),
$$ for some constant $C>0$ depending on $\g$ and $s$.  For the last inequality we used the fact that $\dd f^{3} \leq f^{2}$.  Finally, we write
$$- 2 \int_{\R^3} f  \langle v\rangle^{s-2}(\bm{\Sigma}[f] \grad f)\cdot v \d v  =  \int_{\R^3} f^2 \,\grad \cdot \Big[\langle v\rangle^{s-2}\bm{\Sigma}[f] \, v \,\Big] \,\d v\,.$$
For the last integral we expand
\begin{multline*}
\grad \cdot \Big[\langle v\rangle^{s-2}\bm{\Sigma}[f] \, v \, \Big] \\
=  \langle v\rangle^{s-4} \int_{\R^3} \fet(1-\dd\fet) |v-\vet|^\g\Big[|\vet|^2(2+s|v|^2) -(s-2)(v\cdot \vet)^2-2(v\cdot\vet) \langle v\rangle^2 \Big]\d\vet,
\end{multline*}
which leads to 
\begin{align*}
\int_{\R^3} f^2 \,\grad \cdot \Big[\langle v\rangle^{s-2} \bm{\Sigma}[f] \, v \, \Big] \, \d v  \leq s \int_{\R^3} \fet \langle\vet\rangle^{2+\g} \d\vet &\int_{\R^3} f^{2} \langle v\rangle^{s-2+\g} \d v \\
 & +    2 \|f\|_{L^1_2}\int_{\R^3} f^2 \langle v\rangle^{s-1+\g} \d v.
\end{align*}
Gathering the estimates together, one can find constants $C_{s}$, $\tilde{C}_{s} >0$ such that 
\begin{equation*}
\frac{1}{2}\frac{\d}{\d t} \lM_{s}(t) + K_{0} \int_{\R^{3}}\langle v\rangle^{s+\g}\,\left|\grad f\right|^{2}\d v
\leq  C_{s} \lM_{s+\g}(t) +  \tilde{C}_{s} \lm_{2+\g}(t) \lM_{s+\g-2}(t)\,.
\end{equation*} 
Notice that there exists $c_{s,\g} >0$ such that
$$\int_{\R^{3}}\langle v\rangle^{s+\g}\,\left|\grad f\right|^{2}\d v \geq \lD_{s+\g}(t) - c_{s,\g}\lM_{s+\g}(t)\,.$$
This proves the desired result.
\end{proof}
\noindent
It is important to control the ``mixed term'' $\lm_{2+\g}(t)\,\lM_{s+\g-2}(t)$ in the estimate \eqref{ineq_L2s} of Proposition \ref{lem5}.  This is done in the following lemma.  We continue with the assumptions and notations of Theorem \ref{moments}.
\begin{lem}\label{lem6} Fix $s \in \big(\tfrac{3\gamma}{2}+2, 9 -\gamma\big]$, $\delta >0$.   There exists a constant $C_{s}(\delta) >0$ depending only on $M(f_0)$, $E(f_0)$, $S_{0}$, $s$, $\gamma$ such that 
\begin{equation}\label{ineq_3}
\lm_{2+\g}(t)\,\lM_{s+\g-2}(t) \leq C_{s}(\delta) + \delta\,\lm_{s+\g}(t) + \delta\,\lD_{s+\g}(t), \qquad \forall\, t \geq 0. 
\end{equation}\end{lem}
\begin{proof} Using Littlewood's interpolation inequality, see for instance \cite[Theorem 5.5.1 (ii)]{garling}
$$\|g\|_{L^{2}(\mu)} \leq \|g\|_{L^{1}(\mu)}^{1-\theta}\|g\|_{L^{6}(\mu)}^{\theta}, \qquad \tfrac{1}{2}=(1-\theta) + \tfrac{\theta}{6} \quad \big(\text{that is }\, \theta=\tfrac{3}{5}\big)$$
with the measure $\d\mu(v)=\langle v\rangle^{-3}\d v$ and with $g=f\langle \cdot\rangle^{\frac{s+\g+1}{2}}$, 
we have that
\begin{equation*}
\big\|f \langle\cdot\rangle^{\frac{s+\gamma-2}{2}}\big\|_{L^2}\leq \|f\langle\cdot\rangle^{\frac{s+\gamma}{2} - \frac{5}{2}}\|^{1-\theta}_{L^1}\|f\langle \cdot\rangle^{\frac{s+\gamma}{2}}\|^{\theta}_{L^6}\,.
\end{equation*}
Estimating the last $L^{6}$-norm with Sobolev's inequality \cite[Theorem 12.4]{leoni}, we obtain that
\begin{equation*}
\big\|f\langle\cdot\rangle^{\frac{s+\gamma-2}{2}}\big\|_{L^2} \leq  C  \big\|f\langle\cdot\rangle^{\frac{s+\gamma}{2} - \frac{5}{2}}\big\|^{\frac{2}{5}}_{L^1}\,\big\|\nabla(f\langle\cdot\rangle^{\frac{s+\gamma}{2}})\big\|^{\frac{3}{5}}_{L^2}\,,
\end{equation*}
for some $C>0$. We deduce from this that, as soon as $s \leq 9-\g$ $\big(\text{that is}\; \frac{s+\gamma}{2} - \frac{5}{2} \leq 2\big)$,
$${\lM_{s+\g-2}(t)} \leq  C \lm_{2}(t)^{\frac{4}{5}}\lD_{s+\g}(t)^{\frac{3}{5}}.$$
Moreover, 
$$\lm_{2+\g}(t) \leq \lm_{2}(t)^{\frac{s-2}{s+\gamma-2}}\,\lm_{s+\gamma}(t)^{\frac{\gamma}{s+\gamma-2}}\,,$$
which leads us to deduce, from the conservation of mass and energy, that there is a positive constant $C_{0}$ depending only on $\lm_{2}(0)=M(f_0)+E(f_0)$ such that
$$\lm_{2+\g}(t)\,\lM_{s+\g-2}(t) \leq C_{0}\,\lm_{s+\gamma}(t)^{\frac{\gamma}{s+\gamma-2}}\,\,\lD_{s+\g}(t)^{\frac{3}{5}}\,,\qquad \forall\,t\geq0\,.$$
Using then Young's inequality, there exists $C_{1} >0 $ such that
$$\lm_{2+\g}(t)\,\lM_{s+\g-2}(t) \leq C_{1}\delta^{-3/2}\lm_{s+\gamma}(t)^{\theta}\,+\delta\,\lD_{s+\g}(t) \qquad \forall\,\delta >0$$
with $\theta=\theta(s,\g):=\frac{5}{2}\frac{\gamma}{s+\gamma-2}.$ Notice now that $\theta < 1$ for  $s > \tfrac{3\gamma}{2}+2$ and, from Young's inequality  again, 
$$\lm_{s+\gamma}(t)^{\theta} \leq  \eta^{-\tfrac{\theta}{1-\theta}}+\eta\,\lm_{s+\g}(t) \qquad \forall\, \eta >0\,.$$
Combining these last two inequalities with $\eta=\tfrac{1}{C_{1}}\delta^{\frac{5}{2}}$ gives the result.
\end{proof}
\begin{lem}\label{LemL2h} Fix $s \geq 4 -\g$, $\delta >0$.  There exists a constant $C_{s}(\delta) >0$ depending only on $M(f_0)$, $E(f_0)$, $S_{0}$, $\delta$, $\g$, $s$, such that 
\begin{equation}\label{ineq_4}
\lM_{s+\g}(t) \leq C_{s}(\delta) + \delta\,\lm_{s+\gamma}(t) + \delta\, \lD_{s+\g}(t)\, \qquad \forall\, t \geq 0. 
\end{equation}\end{lem}
\begin{proof} Using Littlewood's interpolation inequality and Sobolev inequality, we obtain that
\begin{equation}\label{eq:lMsg}
\lM_{s+\g}(t) \leq C\lm_{\frac{s+\g}{2}}(t)^{\frac{4}{5}}\,\lD_{s+\g}(t)^{\frac{3}{5}}\,.
\end{equation}
Thus, using Young's inequality,
\begin{equation}\label{eq:Msg}\lM_{s+\g}(t) \leq \delta^{-3/2}\lm_{\frac{s+\gamma}{2}}(t)^{2}  + \delta \lD_{s+\g}(t)\,,\qquad \delta>0\,.\end{equation}
Assuming that $s+\gamma\geq4$, it follows that
\begin{equation}\label{eq:lmsg2}
\lm_{\frac{s+\gamma}{2}}(t)^{2}\leq \lm_{2}(t)^{\frac{s+\gamma}{s+\gamma-2}}\,\lm_{s+\gamma}(t)^{\frac{s+\gamma-4}{s+\gamma-2}}\,,\qquad t \geq 0\,.
\end{equation}
Using conservation of mass and energy, there exists $C_{s} >0$ depending only on $s$, $\g$ and $\lm_{2}(0)$ such that
$$\lm_{\frac{s+\g}{2}}(t)^{2}\leq C_{s,\gamma}(f_0)\lm_{s+\gamma}(t)^{\frac{s+\gamma-4}{s+\gamma-2}}\,,$$
and then, Young's inequality implies that, for $\eta >0$, there is  $C_{s,\g}(\eta)$ such that
$$\lm_{\frac{s+\g}{2}}(t)^{2} \leq C_{s,\g}(\eta) + \eta\,\lm_{s+\gamma}(t)\,,\quad  \quad \forall t >0\,.$$
This, together with \eqref{eq:Msg}, gives the conclusion choosing $\eta >0$ such that $\delta^{-3/2}\eta=\delta$.
\end{proof}

\begin{proof}[Proof of Theorem \ref{moments}] Recall that $s_{\g}=\max\big\{\tfrac{3\gamma}{2}+2, 4-\gamma\big\}$ and define 
$$\mathcal{E}_{s}(t)=\lm_{s}(t) + \lM_{s}(t), \qquad t \geq 0, \quad s \in (s_{\g},9-\g]\,.$$
Adding \eqref{ineq_L1s} to \eqref{ineq_L2s}, there exist positive constants $c_{0},c_{1}$ depending on $M(f_0)$, $E(f_0)$, $S_{0}$, $\g$, $s$, such that
$$\dfrac{\d}{\d t}\mathcal{E}_{s}(t) + c_{0}\bigg[\lm_{s+\g}(t)+\lD_{s+\g}(t)\bigg] \leq c_{1}\bigg[\lM_{s+\g}(t)+\lm_{s}(t)+\lm_{2+\g}(t)\,\lM_{s+\g-2}(t)\bigg]\,, \quad t \geq0\,,$$
and, using the result of Lemmas \ref{lem6} and \ref{LemL2h} it holds  that
$$
\dfrac{\d}{\d t}\mathcal{E}_{s}(t) + c_{0}\bigg[\lm_{s+\g}(t)+\lD_{s+\g}(t)\bigg] \leq c_{1}\bigg[C_{s}(\delta)+\delta\,\lm_{s+\g}(t) + \delta\lD_{s+\g}(t)+\lm_{s}(t)\bigg]\,,\quad \delta>0\,,$$
for some positive constant  $ C_{s}(\delta) >0$ depending on $M(f_0)$, $E(f_0)$, $S_{0}$, $s$, $\g$, $\delta$.  Since 
\begin{equation}\label{eq:lmsep}
\lm_{s}(t) \leq C(s,\delta) + \delta\,\lm_{s+\g}(t), \qquad \forall\, \delta >0
\end{equation}
for some positive constant $C(s,\delta)$ depending on $s$, $\delta$, and $M(f_{0})$, we can choose $\delta >0$ sufficiently small to obtain, for any $s\in(s_\gamma,9-\gamma)$, that
\begin{equation}\label{eq:Et}
\dfrac{\d}{\d t}\mathcal{E}_{s}(t) + \frac{c_{0}}{2} \bigg[\lm_{s+\g}(t)+\lD_{s+\g}(t)\bigg] \leq \tilde{C}, \quad \forall\, t \geq 0,
\end{equation}
for some positive constant $\tilde{C} >0$ depending on $M(f_0)$, $E(f_0)$, $S_{0}$, $\g$, $s$.

\medskip
\noindent 
Let us now control $\mathcal{E}_{s}(t)$ in terms of $\lm_{s+\g}(t)$ and $\lD_{s+\g}(t)$.  Using \eqref{eq:lMsg}, \eqref{eq:lmsg2}, and the conservation of mass and energy 
$$\lM_{s+\g}(t) \leq C_{s,\g}\,\lm_{s+\g}(t)^{\frac{2}{5}\nu}\,\lD_{s+\g}(t)^{\frac{3}{5}}, \quad \nu=\tfrac{s+\g-4}{s+\g-2}\,, \quad t \geq 0\,,$$
with $C_{s,\g}$ a positive constant depending only on $s,\g$ and $\lm_{2}(0)$.  This results in the estimate
\begin{equation*}
\begin{split}
\frac{1}{2}\lm_{s+\g}(t) + \lD_{s+\g}(t) &\geq \frac{1}{2}\lm_{s+\g}(t) + c_{s,\g}\lM_{s+\g}(t)^{\frac{5}{3}}\,\lm_{s+\g}(t)^{-\frac{2}{3}\nu}\\
&\geq \min_{z >0}\left(\frac{z}{2}+c_{s,\g}\lM_{s+\g}(t)^{\frac{5}{3}}\,z^{-\frac{2}{3}\nu}\right)
\end{split}
\end{equation*}
with $c_{s,\g}=C_{s,\g}^{-\frac{5}{3}}$.  Computing such minimum, we obtain that for a constant $\bm{c}_{s,\g}>0$ depending only on $s,\g$ and $\lm_{2}(0)$ it holds that
\begin{equation}\label{eq:sum}
\frac{1}{2}\lm_{s+\g}(t) + \lD_{s+\g}(t) \geq \bm{c}_{s,\g}\,\lM_{s+\g}(t)^{\nu_{0}}
\geq \bm{c}_{s,\g}\, \lM_{s}(t)^{\nu_{0}}\,,\quad \nu_{0}=\tfrac{5}{3+2\nu} >1\,.
\end{equation}
Furthermore,
$$\lm_{s+\g}(t) \geq \lm_{2}(t)^{-\frac{\g}{s-2}}\,\lm_{s}(t)^{\frac{s+\g-2}{s-2}}\,,$$
thus, we deduce from \eqref{eq:sum} that for two positive constant $\bm{c}_{1}, \bm{c}_{2} >0$ depending only on $s,\g$ and $\lm_{2}(0)$ we have
$$\lm_{s+\g}(t) + \lD_{s+\g}(t) \geq \bm{c}_{1}\,\mathcal{E}_{s}(t)^{\beta} - \bm{c}_{2}, \qquad \forall\,t \geq0\,,$$
where $\beta=\min\left(\frac{s+\g-2}{s-2},\nu_{0}\right) >1$. Plugging this into \eqref{eq:Et} yields
\begin{equation}\label{eq:propgen*}
\dfrac{\d}{\d t}\mathcal{E}_{s}(t) + \overline{\bm{c}}\,\mathcal{E}_{s}(t)^{\beta} \leq \bm{C}', \qquad \forall\, t \geq0\,,
\end{equation}
for positive constants $\overline{\bm{c}}$ and $\bm{C}'$ depending only on $s,\g$ and $\lm_{2}(0)$.  Estimate \eqref{eq:propgen*} proves that, for any $s \in (s_{\g},9-\g]$ there exists $\bm{C}_{s} >0$ depending only on $M(f_{0})$, $E(f_{0})$, $S_{0}$, $s$, $\g$, such that
$$\mathcal{E}_{s}(t)\leq \bm{C}_{s}\left(1+t^{-\frac{1}{\beta-1}}\right) \qquad \forall\, t >0\,.$$
In particular, for any $t_{0} >0$, 
\begin{equation}\label{Es}
\sup_{t \geq t_{0}}\mathcal{E}_{s}(t)=\bm{C}(s,t_{0}) <\infty, \qquad \forall\, s \in (s_{\g},9-\g]
\end{equation}
depends only on $t_{0} >0$, $M(f_{0})$, $E(f_{0})$, $S_{0}$, $s$, $\g$.  Observe that estimate \eqref{eq:propgen*} also implies that if $\mathcal{E}_{s}(0)$ is finite, then $\sup_{t\geq0}\mathcal{E}_{s}(t)$ is finite as well proving the propagation of $\mathcal{E}_{s}(t)$.

\medskip
\noindent
Of course, \eqref{Es} remains true for $s \leq 9-\g$.  This means that, for any $t \geq t_{0}$, one can replace \eqref{ineq_L2s} with
$$\frac{1}{2}\frac{\d}{\d t} \lM_{s}(t)   
+ K_{0} \lD_{s+\g}(t)
\leq   
 C_{s}(t_{0})\lM_{s+\g}(t), \quad t \geq t_{0}$$
where $C_{s}(t_{0})$ is a finite constant depending on $t_{0},s,S_{0}$ and $\bm{m}_{2}(0)$ \footnote{namely, $C_{s}(t_{0})=C_{s,2}+C_{s,3}\sup_{t\geq t_{0}}\lm_{2+\g}(t)$}. This shows that \eqref{unif_L2}  holds for any $s >s_{\g}$ since we used the constraint $s \leq 9-\g$ only to estimate $\lm_{2+\g}(t)$.  More precisely, we obtain that
$$\frac{\d}{\d t}\mathcal{E}_{s}(t) + c_{0}\bigg[\lm_{s+\g}(t)+\lD_{s+\g}(t)\bigg] \leq c_{1} \big[\lm_{s}(t) +\lM_{s+\g}(t)\big], \qquad \forall t \geq t_{0}$$
with $c_{0}$ and $c_{1}$ depending on $M(f_{0})$, $E(f_{0})$, $S_{0}$, $s$. Using then \eqref{ineq_4} and \eqref{eq:lmsep} for $\delta >0$ small enough, we obtain
$$\frac{\d}{\d t}\mathcal{E}_{s}(t) + \frac{c_{0}}{2}\bigg[\lm_{s+\g}(t)+\lD_{s+\g}(t)\bigg] \leq  \bm{C}(t_{0}) \qquad \forall s \geq s_{\g},\; t \geq t_{0}$$
for some positive constant $\bm{C}(t_{0})$ depending only on $t_{0}$, $M(f_{0})$, $E(f_{0})$, $S_{0}$, $s$. We can repeat the argument here above, using \eqref{eq:sum}, to obtain now
\begin{equation}\label{eq:propgen**}
\dfrac{\d}{\d t}\mathcal{E}_{s}(t) + \overline{\bm{c}}\,\mathcal{E}_{s}(t)^{\beta} \leq \bm{C}', \qquad \forall\, t \geq t_{0}\,,
\end{equation}
where $\overline{\bm{c}}$ and $\bm{C}'$ depends also on $t_{0}$. This concludes the proof of generation of the norms.  Propagation follows the same idea assuming $\mathcal{E}_{s}(0)$ finite, with $s\geq s_{\gamma}$.  One proceeds from \eqref{eq:propgen*} to arrive to \eqref{eq:propgen**} copycatting the procedure.  Furthermore, integrating in $t\in(t_0,T)$, with $T>t_0$, estimate \eqref{eq:Et} shows that $f\in L^2_{\mathrm{loc}}((t_0,+\infty);H^{1}_{s+\gamma}(\R^3))$ due to generation of the $L^{1}$ and $L^{2}$ norms which proves \it{(i)}.
\end{proof}
\subsection{Regularity estimates} We now prove Theorem \ref{smoothness} with the help of the following proposition. 
\begin{prop}\label{res_H1}
Consider $0\leq f_0\in L^{1}_{s_{\gamma}}(\R^{3})$, where $s_{\gamma}=\max\{2+\tfrac{3\gamma}{2}, 4-\gamma\}$, satisfying \eqref{hypci}.  Let $f=f(t,v)$ be a weak solution to \eqref{LFD} given by Theorem \ref{results}.  Then,  
\begin{equation}\label{eq:grad*}
\int_{\R^3} |\grad f(t,v)|^2 \langle v\rangle^{s} \, \d v \leq C_{t_0}\,,\qquad s\geq0\,, \quad \forall\; t_0>0\,,
\end{equation}
for a constant $C_{t_0}>0$ depending on $M(f_0)$, $E(f_0)$, $S_{0}$, $s$, and $t_0$.  Moreover, if
\begin{equation*}
\int_{\R^3} |\grad f(0,v)|^2 \langle v\rangle^{s} \, \d v <+\infty\,,
\end{equation*}
then, the choice $t_0=0$ is valid in \eqref{eq:grad*} with constant depending on such initial regularity. 
\end{prop}

\begin{proof} Let $i\in\{1,2,3\}$. Differentiating \eqref{LFD} with respect to the $v_i$ variable and setting $g_i=\partial_i f$, we get that
$$\partial_{t} g_i = \grad\cdot\big(\bm{\Sigma}[f]\grad g_i +\partial_i \bm{\Sigma}[f] \grad f -\partial_i \bm{b}[f] f(1-\dd f) - \bm{b}[f] (1-2\dd f)g_i\big).$$
Multiply the equation by $g_i\langle v\rangle^{s}$ and integrate over $\R^3$.  It follows that
$$\frac{1}{2}\frac{\d}{\d t}\int_{\R^3} g_i^2 \langle v\rangle^{s}\, \d v = - I_1+ I_2+I_3+I_4+I_5+I_6+I_7+I_8\,.$$
The terms $I_{j}$, $j=1,\ldots,8$ are estimated thanks to Proposition \ref{diffusion} and Theorem \ref{moments}.  More precisely, we have for $\delta>0$,  
\begin{equation*}\begin{split}
 I_1 & := \int_{\R^3}\langle v\rangle^{s} \bm{\Sigma}[f]\, \grad g_i \grad g_i \, 
\d v \geq K_{0} \int_{\R^3} |\grad g_i|^2 \, \langle v\rangle^{s+\g}\, \d v,\\
I_2 &:= -s  \int_{\R^3} g_{i}\langle v\rangle^{s-2}\left(\bm{\Sigma}[f]\, \grad g_i \cdot v\right) \d v =\frac{s}{2}\int_{\R^3} g_i^2\, \grad \cdot\left(\, \langle v\rangle^{s-2}\bm{\Sigma}[f]\,\,v \right) \, \d v \leq C\,\|g_i\|^2_{L^2_{s+\g}},\end{split}\end{equation*}
while
\begin{equation*}\begin{split}
I_3& :=   - \int_{\R^3} \left(\partial_i\bm{\Sigma}[f]\,\grad f\right)\cdot \grad g_i \,
\langle v\rangle^{s} \,\d v \\
&\hspace{1cm} \leq  C \,\|\grad f\|_{L^2_{s+2+\g}} 
\,\|\grad g_i\|_{L^2_{s+\g}}  \leq  \delta \,\|\grad g_i\|^2_{L^2_{s+\g}} 
+C_\delta \,\|\grad f\|^2_{L^2_{s+2+\g}},\end{split}\end{equation*}
and
\begin{equation*}\begin{split}
I_4& := -s \int_{\R^3} \left(\partial_i\bm{\Sigma}[f]\,
\grad f \cdot \,v \right)\,g_{i} \langle v\rangle^{s-2} \,\d v \leq C \,
\|\grad f\|_{L^2_{s+\g}}\, \| g_i\|_{L^2_{s+\g}},\\
I_5 &:= \int_{\R^3} f(1-\dd f)\,
\langle v\rangle^{s} \,\left(\partial_i \bm{b}[f] \cdot \grad g_i \right)\d v  \leq  C \, \| \grad g_i\|_{L^2_{s+\g}} \, 
\| f\|_{L^2_{s+\g}} \leq    \delta \,\| \grad g_i\|^2_{L^2_{s+\g}} 
+{C_\delta(t_{0})}\,.\end{split}\end{equation*}
Also,
\begin{equation*}
\begin{split}
I_6& :=   s \int_{\R^3} g_i \,f(1-\dd f) \,\langle v\rangle^{s-2}\left(\partial_i \bm{b}[f] \cdot v\right)
 \,\d v \leq C \,\| g_i\|_{L^2_{s}} 
\, \| f\|_{L^2_{s-2+2\g}} \leq C_{t_{0}} \, \| g_i\|_{L^2_{s}},\\
 I_7 &:= \int_{\R^3}  \,(1-2\dd f) \,g_i \,
\langle v\rangle^{s} \left(\bm{b}[f]\cdot \grad g_i\right)\,\d v\\
 &\hspace{1cm}\leq  C  \,\| \grad g_i\|_{L^2_{s+\g}}  \, \| g_i\|_{L^2_{s+2+\g}}
\leq  \delta \,\|\grad g_i\|^2_{L^2_{s+\g}} 
+C_\delta \,\| g_i\|^2_{L^2_{s+2+\g}},\\
 I_8& :=  s \int_{\R^3}  \,(1-2\dd f) \,g_i^{2} \,
\langle v\rangle^{s-2} \left(\bm{b}[f]\cdot v\right) \d v \leq C \,\|g_i\|^2_{L^2_{s+\g}}\,.\end{split}\end{equation*}
Here, and in the rest of the proof, $C$ denotes a positive constant depending only on $M(f_{0})$, $E(f_{0})$, $S_{0}$, $s$, but not $\dd$, which may change from line to line.  Gathering the above estimates, summing over $i\in\{1,2,3\}$ and recalling that $g_i=\partial_i f$, we obtain
$$ \frac{1}{2}\frac{\d}{\d t}\|\grad f\|^2_{L^2_{s}}
+(K-3\delta) \|f\|^2_{\dot{H}^2_{s+\g}} \leq C \,\left( 1+ \|\grad f\|^2_{L^2_{s+2+\g}}\right)\,.$$
Since, an integration by parts leads to  
$$\|\grad f \|^2_{L^2_{s+2+\g}} =  -\int_{\R^3}  f \, \Delta f \, \langle v\rangle^{s+2+\g}\, \d v -(s+2+\g) \int_{\R^3} f\,\langle v\rangle^{s+\g} \grad f \cdot v \, \d v, $$
we deduce from Young's inequality that
$$ \|\grad f \|^2_{L^2_{s+2+\g}} 
\leq \frac{\delta}{2} \|f\|^2_{\dot{H}^2_{s+\g}}
+ C_\delta \|f\|^2_{L^2_{s+4+\g}} 
+ \frac{1}{2} \|\grad f \|^2_{L^2_{s+2+\g}}  
+C\, \|f\|^2_{L^2_{s+\g}}.$$
Thus, estimate \eqref{unif_L2} imply that 
$$\|\grad f\|^2_{L^2_{s+2+\g}} \leq  \delta  \|f\|^2_{\dot{H}^2_{s+\g}} + C\,.$$
Therefore, it follows that 
$$\frac{1}{2}\frac{\d}{\d t} \|\grad f\|^2_{L^2_{s}} 
+(K-4\delta)  \|f\|^2_{\dot{H}^2_{s+\g}}\leq {C}\,.$$
Using the interpolation inequality and estimate \eqref{unif_L2}
\begin{equation*}
\|\nabla f\|^{2}_{L^{2}_{s+\gamma}}=\|f\|^2_{\dot{H}^1_{s+\g}} \leq \|f\|_{L^2_{s+\g}}\|f\|_{\dot{H}^2_{s+\g}} \leq C\|f\|_{\dot{H}^2_{s+\g}}\,, \qquad \forall t\geq t_0>0.
\end{equation*}
Therefore, choosing $4\delta = K/2$, it follows from the previous two estimates that
$$\frac{1}{2}\frac{\d}{\d t} \|f \|^2_{\dot{H}^1_{s}} +\frac{K}{2}  \|f\|^4_{\dot{H}^1_{s}}\leq {C}\,,\quad s\geq0\,,\quad t>t_{0}\,.$$
From this estimate all statements of the proposition follow.
\end{proof}

\begin{proof}[Proof of Theorem \ref{smoothness}]
The proof of Theorem \ref{smoothness} is a generalisation of the steps given in the proof of Proposition \ref{res_H1}.  Consider $g=\partial_v^{\beta} f$ for some $\beta\in\N^3$ with $|\beta|=k$.  Differentiating the LFD equation \eqref{LFD} $\beta$ times, we get the equation satisfied by $g$.  We multiply such equation by $g\langle v\rangle^s$ and integrate over $\R^3$.  Estimating the terms as in the proof of \cite[Proposition 2.1]{Chen11} one obtains that  for any $\delta>0$
\begin{equation}\label{H_k+1}
\frac{1}{2}\frac{\d}{\d t } \|f\|^2_{\dot{H}^{k}_s } +(K_{0}-\delta) \|f\|^2_{\dot{H}^{k+1}_{s+\g}} 
\leq  C_\delta\big( 1+  \|f\|^2_{\dot{H}^{k}_{s+2+\g }}\big),
\end{equation}
for a constant $C_\delta>0$ and $K_{0}>0$ given by Proposition \ref{diffusion}.  Using the interpolation inequality
\begin{equation*}
\|f\|^2_{\dot{H}^{k}_{s'}} \leq \|f\|_{\dot{H}^{k-1}_{2s'}} \|f\|_{\dot{H}^{k+1}}  + \|f\|_{\dot{H}^{k-1}_{2s'}} \|f\|_{L^{2}}\,,\quad s'\geq0\,, 
\end{equation*}
to control the right hand term of \eqref{H_k+1}, it follows that
\begin{equation*}\label{H_k+2}
\frac{1}{2}\frac{\d}{\d t } \|f\|^2_{\dot{H}^{k}_s } +\frac{K_{0}}{2}\|f\|^2_{\dot{H}^{k+1}_{s+\g}} 
\leq  C\big( 1+  \|f\|^2_{\dot{H}^{k-1}_{2s+4+2\g }}\big)\,,\quad s\geq0\,,\quad t\geq t_0>0\,. 
\end{equation*}
Here we used Theorem \ref{moments} to control the $L^{2}$-norm.  The interpolation inequality
\begin{equation*}
\|f\|^2_{\dot{H}^{k}_{s'}} \leq \|f\|_{\dot{H}^{k+1}_{s'}} \|f\|_{L^{2}_{s'}} \leq C\,\|f\|_{\dot{H}^{k+1}_{s'}}\,,
\end{equation*}
leads to, choosing $s'=s+\gamma$,
\begin{equation*}
\frac{1}{2}\frac{\d}{\d t } \|f\|^2_{\dot{H}^{k}_s } + \tilde{K}_0\|f\|^4_{\dot{H}^{k}_{s}} 
\leq  C\big( 1+  \|f\|^2_{\dot{H}^{k-1}_{2s+4+2\g }}\big)\,,\quad s\geq0\,,\quad t\geq t_0>0\,. 
\end{equation*}
Starting with $k=2$ and Proposition \ref{res_H1}, repeat this estimate to obtain the result.  Note that in each repetition twice the number of moments is needed.  This explain the condition $f_0\in L^{1}_{s'}$ with $s'$ sufficiently large ($s'\approx 2^{k}s$ ). 
\end{proof}

\noindent
We deduce from this the following important consequence.
\begin{cor}\label{cor:Linfty}
Consider $0\leq f_0\in L^{1}_{s_{\gamma}}(\R^{3})$ satisfying \eqref{hypci}. Then, for any solution $f(t)=f_{\dd}(t)$ to \eqref{LFD} given by Theorem \ref{results}., it holds
$$\sup_{t \geq t_{0}}\left\|f(t)\right\|_{\infty} \leq C_{t_{0}}\,,\qquad \forall\, t_{0}>0.$$
The constant $C_{t_{0}}$ only depends on $M(f_0)$, $E(f_0)$, $S_{0}$, $s$, and $t_0$.

\smallskip
\noindent
Consequently, for any $\kappa_{0} \in (0,1)$ there exists $\dd_{\star}>0$ depending only on $\kappa_{0}$, $M(f_{0})$, $E(f_{0})$, and $S_{0}$, such that
\begin{equation}\label{eq:lower*}
\inf_{v \in \R^{3}}\big(1-\dd\,f(t,v)\big) \geq \kappa_{0}, \qquad \forall\, \dd \in (0,\dd_{\star}),\,   t \geq 1.
\end{equation}\end{cor}

\begin{proof} The first statement is a direct consequence of Theorem \ref{smoothness} and the Sobolev embedding of $H^{s}(\R^{3})$ into $L^{\infty}(\R^{3})$ for any $s > \tfrac{3}{2}$, see for instance \cite[Theorem 12.46]{leoni}.  The lower bound \eqref{eq:lower*} is, thus, clear as one can choose $\dd >0$ so that $f(t,v) \leq C_{t_{0}=1} \leq \frac{1-\kappa_{0}}{\dd}$.\end{proof}

\begin{rmq}
Of course the choice $t \geq 1$ is arbitrary. The result is valid for any time sufficiently large.  This lower estimate rules out any degeneracy as the set $\{v: f=\dd^{-1}\}$ is empty after sufficiently large time. In particular, solutions to \eqref{LFD} remain \emph{uniformly away} from the saturated Fermi-Dirac statistics \eqref{eq:dege}.
\end{rmq}

\section{Convergence to equilibrium: non quantitative result}\label{sec:convno}
 
From the emergence of smoothness and moments of the previous sections a non-quantitative result of convergence to equilibrium can be inferred, see \cite{CLR}. The proof will exploit several results from \cite{bag1,bag2} and will resort on the property of the dissipation of entropy functional established in Section \ref{sec:entropy}.  
We begin with a preliminary lemma. 
\begin{lem}\label{prelim}
 Let $f_{0}\in L^1_{s}(\R^3)$ with $s> 17+6\gamma $ satisfying Assumption \ref{ci}, and let $f$ be the weak solution to \eqref{LFD} given by Theorem \ref{results}. Then, 
\begin{multline*}
\int_0^\infty\left(\int_{\R^6}|v-\vet|^{\frac{\g+2}{2}}\Big|\Pi(v-\vet)\big(f_{\ast}(1-\dd f_{\ast})\nabla f - f(1-\dd f)\nabla f_{\ast}\big)\Big|\d v\d \vet\right)^2\d t \\
\leq 2{\mathcal S}_{\dd}(\M_{\dd}) \left( \int_{\R^3} f_{0} \d v\right)^2,
\end{multline*}
where $\M_{\dd}$ is the Fermi-Dirac statistics with same mass, momentum and energy as $f_{0}$.
\end{lem}
\begin{proof} 
We work with the regularized problem introduced in \cite[Section 4.1]{bag1} to make all computations rigorous. Let $T>0$. As in the proof of \cite[Theorem 2.2]{bag1}, we denote by $(f^{k}_{0})_{k\geq 1}$ a sequence of smooth functions that converges to $f_{0}$ in $L^1_{s}(\R^3)$ and by $f^k$ the solution to the regularized problem with initial datum  $f^{k}_{0}$.  For any $k\geq 1$, the function $f^k$ is smooth on $ [0,T]\times \R^3$ and satisfies $0<f^k(t,v)<\dd^{-1}$ for any $(t,v)\in [0,T]\times \R^3$. Moreover, by \cite[Lemma 4.15]{bag1}, a subsequence of $(f^k)_{k\geq 1}$, not relabelled, converges to $f$ in $L^2((0,T);L^1(\R^3))$ and a.e. on $(0,T)\times \R^3$. Actually, one can also prove that $\nabla f^k$ converges to $\nabla f$ in $L^2((0,T); L^2_{q-2}(\R^3))$ with $q=s-10-4\gamma$. Indeed, using the uniform (in $k$) ellipticity estimate given by \cite[Corollary 4.10]{bag1} and performing the same computations as in the proof of \cite[Theorem 5.2]{bag1}, we obtain that, for any $t\in[0,T]$, 
\begin{multline*}
\|f^k(t)-f^{\ell}(t)\|^2_{L^2_{q}} + \kappa \int_0^t \|\nabla f^k(\tau)- \nabla f^{\ell}(\tau)\|^2_{L^2_{q-2}} \d \tau\\
\leq  \|f^{k}_{0}-f^{\ell}_{0}\|^2_{L^2_{q}} 
+\int_0^t D(\tau) \|f^k(\tau)-f^{\ell}(\tau)\|^2_{L^2_{q}} \d \tau,
\end{multline*}
for some function $D$ and where $\kappa$ is given by \cite[Corollary 4.10]{bag1}. Observe that the assumption  $f_{0}\in L^1_{s}(\R^3)$ with $s> 17+6\gamma $ ensures that $D\in L^2(0,T)$ with some bound independent of $k$ and $\ell$. Hence, the Gronwall Lemma implies that, for any $t\in[0,T]$,
$$\|f^k(t)-f^{\ell}(t)\|^2_{L^2_{q}} \leq  \|f^{k}_{0}-f^{\ell}_{0}\|^2_{L^2_{q}} \exp\left(\int_0^t D(\tau) \d \tau\right),$$
and thus, 
$$\kappa \int_0^T \|\nabla f^k(\tau)- \nabla f^{\ell}(\tau)\|^2_{L^2_{q-2}}\,\d\tau \leq \|f^{k}_{0}-f^{\ell}_{0}\|^2_{L^2_{q}} \exp\left(\int_0^T D(\tau) \d \tau\right). $$
This proves that $(\nabla f^k)_{k\geq 1}$ is a Cauchy sequence in $L^2((0,T);L^2_{q-2}(\R^3))$ and thus converges to some $g\in L^2((0,T);L^2_{q-2}(\R^3)) $. Necessarily, we have $g=\nabla f$. 
Note that mass and momentum are preserved for the regularized problem but not the energy. More precisely, for any $t\in[0,T]$ we have, see \cite[Lemma 4.8]{bag1}
$$\int_{\R^3} f^k(t,v)\d v=\int_{\R^3}f^{k}_{0}(v)\d v, \qquad \int_{\R^3} f^k(t,v)v\d v=\int_{\R^3}f^{k}_{0}(v)v\d v $$
but $$ \int_{\R^3} f^k(t,v)|v|^2 \d v=\int_{\R^3}f^{k}_{0}(v) |v|^2 \d v +\frac{6t}{k} \int_{\R^3}f^{k}_{0}(v)\d v.$$
Let us introduce 
\begin{equation}\label{eq:xi}
\bm{\xi}[f^k](v,\vet)=\Pi(v-\vet)\left(f^k_{\ast}(1-\dd f^k_{\ast})\nabla f^k - f^k(1-\dd f^k)\nabla f^k_{\ast}\right) \in \R^{3}.
\end{equation}
The Cauchy-Schwarz inequality ensures that 
\begin{multline*}
\int_{\R^{6}}|v-\vet|^{\frac{\g+2}{2}}\left|\bm{\xi}[f^k](v,\vet)\right|\d v \d\vet \\
\leq \bigg(\int_{\R^{6}}|v-\vet|^{\g+2}\frac{\left|\Pi(v-\vet)\left(f^k_{\ast}(1-\dd f^k_{\ast})\nabla f^k - f^k(1-\dd f^k)\nabla f^k_{\ast}\right)\right|^{2}}{f^k\,f^k_{\ast}(1-\dd f^k)(1-\dd f^k_{\ast})}\d v\d\vet\bigg)^{1/2}
\\
\times\bigg(\int_{\R^{3}}f^k\,f^k_{\ast}(1-\dd f^k)(1-\dd f^k_{\ast})\d v\d\vet\bigg)^{1/2}\,.
\end{multline*}
One checks that, with the notations introduced in Section \ref{sec:entropy}, the first integral on the right-hand side coincides with $2\mathscr{D}_{\dd}(f^k)$.  Using also the bound $0 \leq 1-\dd f^k \leq 1$ it follows that
\begin{equation*}
\left(\int_{\R^{6}}|v-\vet|^{\frac{\g+2}{2}}\left|\bm{\xi}[f^k](v,\vet)\right|\d v \d\vet\right)^{2} \leq  2\mathscr{D}_{\dd}(f^k)\,\left(\int_{\R^3} f^{k}_{0}(v)\d v\right)^{2}\,.
\end{equation*}
On the other hand, we deduce from the entropy identity for $f^k$, see the proof of \cite[Lemma 4.8]{bag1}, that  
$$\int_0^T \mathscr{D}_{\dd}(f^k)(t) \d t \leq {\mathcal S}_{\dd}(f^k)(T) \leq {\mathcal S}_{\dd}(\M_{\dd}^{k,T})$$
where $\M_{\dd}^{k,T}$ is the Fermi-Dirac statistics with same mass, energy and momentum as $f^k(T)$ and it maximises the entropy among the class of functions with prescribed mass, momentum and energy. Hence, we deduce that  
$$\int_0^T \left(\int_{\R^{6}}|v-\vet|^{\frac{\g+2}{2}}\left|\bm{\xi}[f^k](v,\vet)\right|\d v \d\vet\right)^{2} \d t  \leq 2 {\mathcal S}(\M_{\dd}^{k,T})\,\left(\int_{\R^3} f^{k}_{0}(v)\d v\right)^{2}.$$
It only remains to let first $k\to+\infty$ and then $T\to+\infty.$ 
\end{proof}
\begin{theo}\label{theo:conve} Consider $0\leq f_0\in L^{1}_{s_\gamma}(\R^{3})$ satisfying \eqref{hypci} and let $f=f(t,v)$ be a weak solution to \eqref{LFD} given by Theorem \ref{results}. Then, 
$$\lim_{t\to\infty}\|f(t)-\M_{\dd}\|_{L^{1}_{2}}=0$$
where $\M_{\dd}$ is the Fermi-Dirac statistics with same mass, energy and momentum of $f_{0}$.\end{theo} 

\begin{proof} Let $\dd \in (0,\dd_{0}]$ be given. Fix $t_0>0$. From the propagation and appearance of smoothness and moments established in Theorem \ref{smoothness}, we get that
\begin{equation*}
\sup_{t \geq t_0}\left(\|f(t)\|_{L^{1}_{s}}+\|f(t)\|_{H^{k}_{q}}\right) < \infty, \qquad \forall\, s \geq 0, \; q \geq 0, \; k \in \mathbb{N}.\end{equation*}
In particular, by Sobolev embedding, the family 
$$\{f(t)\}_{t \geq t_0} \qquad \text{ is relatively compact in } H^{1}_{p}(\R^3) \text{ for any } p\geq0\,.$$
Consider then a sequence $\{t_{n}\}_{n\in \N}$ of positive real numbers with $\lim_{n}t_{n}=\infty$. One can extract from it a subsequence, still denoted $\{t_{n}\}_{n}$, and $F_{\infty} \in H^{1}_{p}(\R^{3})$ such that 
$$\lim_{n\to\infty}\|f(t_{n})-F_{\infty}\|_{H^{1}_{p}}=0.$$
We introduce then
$$f_{n}(t)=f(t+t_{n}), \qquad t \in [0,1], \qquad n \in \N$$
and denote by $(F(t))_{t\geq0}$ the unique solution to \eqref{LFD} with initial datum $F(0)=F_{\infty}$ given by Theorem \ref{results}.

\smallskip
\noindent
Let us choose $p > 4\gamma+11$ and apply an analog stability result to \cite[Theorem 5.2]{bag1} with $q=p-(2\gamma+4)$. Since $\sup_{t \in [0,1]}\|f_{n}(t)+F(t)\|_{H^{1}_{p}} < \infty$ according to Theorem \ref{smoothness}, we get that
$$\frac{\d}{\d t}\|f_{n}(t)-F(t)\|_{L^{2}_{q}}^{2}+K_{0}\|\nabla (f_{n}(t)-F(t))\|_{L^{2}_{q+\gamma}}^{2} \leq \overline{C}_{0}\|f_{n}(t)-F(t)\|_{L^{2}_{q}}^{2}, \qquad \forall\, t \in [0,1]$$
for some positive constant $\overline{C}_{0} >0.$ In particular,
$$\sup_{t \in [0,1]}\|f_{n}(t)-F(t)\|_{L^{2}_{q}} \leq \exp\big(\tfrac{\overline{C}_{0}}{2}\, t\big)\|f(t_{n})-F_{\infty}\|_{L^{2}_{q}}\,,$$
and therefore,
\begin{equation}\label{eq:limnF}
\lim_{n\to\infty}\sup_{t \in [0,1]}\|f_{n}(t)-F(t)\|_{L^{2}_{q}}=0 \quad \text{ and } 	 
\quad \lim_{n\to\infty}\int_{0}^{1}\|\nabla f_{n}(t)-\nabla F(t)\|_{L^{2}_{q+\gamma}}^{2}\d t=0.
\end{equation}
Notice that, up to a subsequence, $\lim_{n}f_{n}(t,v)=F(t,v)$ for a.e. in $v \in \R^{3}$.  Thus, one still has
$$0 \leq F(t,v) \leq \dd^{-1}\,, \quad \text{a.e. in }\, v \in \R^{3}.$$
Moreover, a simple use of Cauchy-Schwarz inequality implies that the convergence \eqref{eq:limnF} transfers to
\begin{equation}\label{eq:limnFL1}
\lim_{n\to\infty}\sup_{t \in [0,1]}\|f_{n}(t)-F(t)\|_{L^{1}_{s}}=0 \quad \text{ and } 	\quad \lim_{n\to\infty}\int_{0}^{1}\|\nabla f_{n}(t)-\nabla F(t)\|_{L^{1}_{\tau}}\d t=0
\end{equation}
as soon as 
$$0 \leq s <\frac{q-1}{2}-1=\frac{1}{2}(p-2\g-7) \quad \text{ and } \quad 0\leq \tau < \frac{q+\g-1}{2}-1=\frac{1}{2}(p-\g-7).$$
In particular, since $p >4\gamma+11$, one notices that $s=\tau=2$ are both admissible.  Since $t_0>0$, we deduce from Lemma \ref{moments} that $f(t_0)\in L^1_{q}(\R^3)$ with $q>17+6\gamma$. Applying Lemma \ref{prelim} to $f(t_0+t)$ we get that the mapping
$$(t,v,\vet) \in [t_0,\infty) \times \R^{6} \mapsto |v-\vet|^{\frac{\g+2}{2}}\big|\bm{\xi}[f(t)](v,\vet)\big|\smallskip$$
lies in $L^{2}((t_0,\infty),L^{1}(\R^{6}))$ with
$$\int_{t_0}^{\infty}\bigg(\int_{\R^{6}}|v-\vet|^{\frac{\g+2}{2}}\big|\bm{\xi}[f(t)](v,\vet)\big|\d v\d\vet\bigg)^{2}\d t \leq 2M(f_0)^{2}\mathcal{S}_{\dd}(\M_{\dd}).$$
In particular, since
\begin{multline*}
\int_{0}^{1}\left(\int_{\R^{6}}|v-\vet|^{\frac{\g+2}{2}}\big|\bm{\xi}[f_{n}(t)](v,\vet)\big|\d v \d\vet\right)^{2}\d t\\
=\int_{t_{n}}^{t_{n}+1} \left(\int_{\R^{6}}|v-\vet|^{\frac{\g+2}{2}}\left|\bm{\xi}[f(t)](v,\vet)\right|\d v\d\vet\right)^{2}\d t \qquad \forall n \in \N\,,
\end{multline*}
we get that
\begin{equation}\label{eq:xifn}
\lim_{n\to\infty}\int_{0}^{1}\left(\int_{\R^{6}}|v-\vet|^{\frac{\g+2}{2}}\big|\bm{\xi}[f_{n}(t)](v,\vet)\big|\d v \d\vet\right)^{2}\d t=0.
\end{equation}
By virtue of \eqref{eq:limnF} and because $q >0$ is large enough to transfer the $L^{2}_{q}$ convergence into a $L^{1}_{\frac{\g+2}{2}}$ convergence 
\begin{align*}
\lim_{n\to\infty}\int_{0}^{1}\bigg(\int_{\R^{6}}|v-\vet|^{\frac{\g+2}{2}}&\big|\bm{\xi}[f_{n}(t)](v,\vet)\big|\d v \d\vet\bigg)^{2}\d t\\
&=\int_{0}^{1}\left(\int_{\R^{6}}|v-\vet|^{\frac{\g+2}{2}}\left|\bm{\xi}[F(t)](v,\vet)\right|\d v \d\vet\right)^{2}\d t.
\end{align*}
Thus,
$$\int_{0}^{1}\left(\int_{\R^{6}}|v-\vet|^{\frac{\g+2}{2}}\left|\bm{\xi}[F(t)](v,\vet)\right|\d v \d\vet\right)^{2}\d t=0$$
from which we readily deduce that 
$$\bm{\xi}[F(t)](v,\vet)=0 \qquad \text{ a. e. } t \in (0,1), v,\vet \in \R^{6}.$$
From the definition of $\bm{\xi}$, see equation \eqref{eq:xi}, and \cite[Theorem 4]{bag2} we notice that if
\begin{equation}\label{eq:measure}
\left|\left\{v \in \R^{3}\,;\; 0 < F(t,v) < \dd^{-1}\right\}\right| \neq 0\,,
\end{equation}
then, the density $F(t)$ is a Fermi-Dirac statistics
$$F(t,v)=\frac{a(t)\exp(-b(t)|v-\bm{v}(t)|^{2})}{1+\dd a(t)\exp(-b(t)|v-\bm{v}(t)|^{2})}, \qquad v \in \R^{3}$$
for some suitable $a(t),b(t) >0$. Now, \eqref{eq:measure} is clearly satisfied since, according to  Lemma \ref{L2unif}, there exists $\eta_{\ast} >0$ and $R_{\ast} >1$ depending only on $M(f_{0})$, $E(f_0)$, and $S_{0}$, such that
$$\int_{B_{R_{\ast}}}f(t,v) (1-\dd f(t,v))\d v \geq  \eta_{\ast} \qquad \forall\, t \geq t_0.$$
According to \eqref{eq:limnF}, this readily translates in 
$$\int_{B_{R_{\ast}}}F(t,v) (1-\dd F(t,v))\d v \geq  \eta_{\ast} \qquad \forall t \in [0,1]$$
which proves \eqref{eq:measure}. Therefore, $F(t,v)$ is a (time-dependent) Fermi-Dirac statistics.  Using the fact that the convergence of $f_{n}(t)$ to $F(t)$ occurs at least in $L^{1}_{2}(\R^{3})$, we observe that
$$\int_{\R^{3}}F(t,v)\left(\begin{array}{c}1 \\v \\|v|^{2}\end{array}\right)\d v=\int_{\R^{3}}f(t,v)\left(\begin{array}{c}1 \\v \\|v|^{2}\end{array}\right)\d v=\int_{\R^{3}}\M_{\dd}(v)\left(\begin{array}{c}1 \\v \\|v|^{2}\end{array}\right)\d v\,.$$
Therefore $F(t,v)=\M_{\dd}(v)$ for a.e. $(t,v) \in [0,1] \times \R^{3}$.   In particular, $\M_{\dd}$ is the only possible cluster point of $\{f(t)\}_{t\geq t_0}$.  The theorem is proved.
\end{proof}

\section{Linearized theory and estimates for the spectral gap}\label{sec:line}

\subsection{Existence of a spectral gap}\label{sec:spectral} 
Recall the non quantitative linearized theory performed in \cite{lemou} for $\dd =1$.  Fix $\dd>0$ and let $\M_{\dd}$ be the Fermi-Dirac statistics with same mass, momentum and energy as $f_{0}$.  Define
$$\mM(v)=\mM_{\dd}(v)=\M_{\dd}(v)(1-\dd \M_{\dd}(v)), \qquad v \in \R^{3}.$$
Perform a linearization around such statistics by writting
\begin{equation}\label{eq:ansatz}
f(t,v)=\M_{\dd}(v)+ \mM(v)h(t,v)
\end{equation}
and plugging into \eqref{LFD}.  We get that
\begin{equation}\label{eq:ansatzpt}
\partial_{t}h=\bm{L}_{\dd}(h) + \Gamma_{2,\dd}(h) + \Gamma_{3,\dd}(h)\end{equation}
where
\begin{multline}\label{eq:Lddhm}
\bm{L}_{\dd}h(v)=\mM^{-1}(v)\nabla \cdot \int_{\R^{3}}|v-\vet|^{\g+2}\Pi(v-\vet)\left[\mM_{\ast}\nabla (\mM\,h) + h_{\ast}\mM_{\ast}(1-2\dd\M_{\dd})_{\ast}\nabla \M_{\dd}\right.\\
\left. - \mM\,(\nabla (\mM\,h))_{\ast}-h\mM\,(1-2\dd\M_{\dd})(\nabla \M_{\dd})_{\ast}\right]\d \vet
\end{multline}
is a linear operator and
\begin{multline*}
\Gamma_{2,\dd}h(v)=\mM^{-1}(v)\nabla \cdot \int_{\R^{3}}|v-\vet|^{\g+2}\Pi(v-\vet)\left[\dd (\mM h)^{2}(\nabla \M_{\dd})_{\ast}-\dd (\mM h)_{\ast}^{2}\nabla\M_{\dd}\right.\\
+\left.(\mM\,h)_{\ast}(1-2\dd\M_{\dd})_{\ast}\nabla (\mM\,h) - (\mM\,h)(1-2\dd\M_{\dd})(\nabla (\mM\,h))_{\ast} \right]\d \vet
\end{multline*}
is a quadratic operator, and
\begin{multline*}
\Gamma_{3,\dd}h(v)=\dd \mM^{-1}(v)\nabla \cdot \int_{\R^{3}}|v-\vet|^{\g+2}\Pi(v-\vet)\left[(\mM h)^{2}(\nabla (\mM h))_{\ast}- (\mM h)^{2}_{\ast} \nabla (\mM\,h)\right]\d \vet
\end{multline*}
collects the cubic terms.

\medskip
\noindent
Noticing that $\nabla \mM=(1-2\dd\M_{\dd})\nabla \M_{\dd}$, one may rewrite the linear part as 
\begin{multline*}
\bm{L}_{\dd}h(v)=\mM^{-1}(v)\nabla \cdot \int_{\R^{3}}\mM_{\ast}\,\mM\,|v-\vet|^{\g+2}\Pi(v-\vet)\left[\nabla h -(\nabla h)_{\ast}\right]\d \vet\\
+\mM^{-1}(v)\nabla \cdot \int_{\R^{3}}\left[h_{\ast}(1-2\dd\M_{\dd})_{\ast} + h(1-2\dd\M_{\dd})\right]\\ |v-\vet|^{\g+2}\Pi(v-\vet)\left[\mM_{\ast}\,\nabla \M_{\dd}-\mM (\nabla \M_{\dd})_{\ast}\right]\d \vet.
\end{multline*}
It is easily seen that $\mM_{\ast}\,\nabla \M_{\dd}-\mM (\nabla \M_{\dd})_{\ast}$ is proportional to $(\vet-v)\mM\,\mM_{\ast}$, thus, the second integral vanishes and
\begin{equation}\label{eq:Ld}
\bm{L}_{\dd}h(v)=\mM^{-1}(v)\nabla \cdot \int_{\R^{3}}\mM(v_{\ast})\,\mM(v)\,|v-\vet|^{\g+2}\Pi(v-\vet)\left[\nabla h(v) -(\nabla h)(\vet)\right]\d \vet.\end{equation}
Using the same kind of relations on the quadratic part, it follows that
\begin{multline*}
\Gamma_{2,\dd}(h)=\mM^{-1}(v)\nabla \cdot \int_{\R^{3}}\mM(v_{\ast})\,\mM(v)\,|v-\vet|^{\g+2}\Pi(v-\vet)\\
\bigg[\dd h^{2}\frac{\mM}{\mM_{\ast}}(\nabla \M_{\dd})_{\ast}-
\dd  h_{\ast}^{2}\frac{\mM_{\ast}}{\mM}\nabla\M_{\dd}
\big((1-2\dd\M_{\dd})_{\ast}h_{\ast}\nabla h -(1-2\dd\M_{\dd})\,h(\nabla h)_{\ast}\big) \bigg]\d \vet\,.
\end{multline*}
Notice that, from \eqref{eq:ansatz}, we expect $h$ to satisfy
\begin{equation}\label{eq:ortho}
\int_{\R^{3}}h(v)\mM(v)\left(\begin{array}{c}1 \\v \\|v|^{2}\end{array}\right)\d v=\left(\begin{array}{c}0 \\0 \\0\end{array}\right).\end{equation}
A natural space to study the operator $\bm{L}_{\dd}$ is the Hilbert space $L^{2}(\mM)$.  In this space, the natural domain of $\bm{L}_{\dd}$ is 
$$\mathscr{D}(\bm{L}_{\dd})=\left\{h \in L^{2}(\mM)\;;\;\int_{\R^{3}}\langle v\rangle^{\g+2}\,|\nabla h(v)|^{2}\mM(v)\d v < \infty\right\}.$$
We denote by $\langle\cdot,\cdot\rangle_{2}$ the inner product in $L^{2}(\mM)$. For any $g,h \in \mathscr{D}(\bm{L}_{\dd})$ one has
\begin{equation*}\begin{split}
\langle g,\bm{L}_{\dd}h\rangle_{2}&=\int_{\R^{3}}g(v)\nabla\cdot \int_{\R^{3}}|v-\vet|^{\g+2}\Pi(v-\vet)\left[\nabla h-\nabla h_{\ast}\right]\mM\,\mM_{\ast}\d v\d \vet\\
&=-\int_{\R^{6}}\mM\,\mM_{\ast}|v-\vet|^{\gamma+2}\Pi(v-\vet)\left(\nabla h-\nabla h_{\ast}\right)\nabla g\d v\d\vet\\
&=-\frac{1}{2}\int_{\R^{6}}\mM\,\mM_{\ast}|v-\vet|^{\gamma+2}\Pi(v-\vet)\left(\nabla h-\nabla h_{\ast}\right)\left(\nabla g-\nabla g_{\ast}\right)\d v\d\vet.
\end{split}\end{equation*}
In particular, $\bm{L}_{\dd}$ is symmetric and the associated Dirichlet form
reads $$\mathcal{D}_{\dd,\gamma+2}(h)=-\langle h,\bm{L}_{\dd}h\rangle_{2}=\frac{1}{2}\int_{\R^{6}}\mM\,\mM_{\ast}|v-\vet|^{\gamma+2}\big|\Pi(v-\vet)\left(\nabla h-\nabla h_{\ast}\right)\big|^{2}\d v\d\vet \geq 0.$$
The spectral analysis of $\bm{L}_{\dd}$ has been performed in \cite{lemou}.  We remark that the linearization used there is slightly different, but the results are easily adapted to our linearization \eqref{eq:ansatz}.  The following theorem holds.
\begin{theo}\label{theo:lemou}
There exists $\lambda_{\dd} >0$ such that
$$\mathcal{D}_{\dd,\gamma+2}(h) \geq \lambda_{\dd}\,\|h\|_{L^{2}(\mM)}^{2}, \quad \text{ for any } h \in L^{2}(\mM) \text{ satisfying } \eqref{eq:ortho}.$$
\end{theo}
\noindent
The parameter $\lambda_{\dd} >0$ obtained in \cite{lemou} is not explicit since Weyl's Theorem is used in the argument. 
\begin{rmq} The Dirichlet form $\mathcal{D}_{\dd, \gamma+2}$ associated to the linearized Landau-Fermi-Dirac operator is very similar to the one associated to the classical linearized Landau operator in $L^{2}(\M_{0})$ given by
$$\mathcal{D}_{0,\gamma+2}(g)=\frac{1}{2}\int_{\R^{6}}\M_{0}(v)\,\M_{0}(\vet)|v-\vet|^{\gamma+2}\left|\Pi(v-\vet)\left(\nabla h-\nabla h_{\ast}\right)\right|^{2}\d v\d\vet$$
where $\M_{0}$ is the Maxwellian distribution with same mass, energy and momentum as $f_{0}$. 
Recall also that there exists an \textit{explicit} $\lambda_{0} >0$ such that 
\begin{equation}\label{eq:D0h}
\mathcal{D}_{0,\gamma+2}(h) \geq \lambda_{0}\,\|h\|_{L^{2}(\M_{0})}^{2}\end{equation} for any $h \in L^{2}(\M_{0})$ orthogonal to $\mathrm{Span}(1,v,|v|^{2})$ in $L^{2}(\M_{0}).$
\end{rmq}
\noindent
In the rest of the section, we give an explicit estimate of the spectral gap of the linearized operator $\bm{L}_{\dd}$ in Theorem \ref{theo:lemou}. We begin with the following lemma which can be easily deduced from \cite[Theorem 3.2 \& Corollary 3.4]{lemou} where we recall that 
\begin{equation}\label{eq:mMab}
\mM(v)=\frac{M(v)}{\left(1+\dd\,M(v)\right)^{2}} \quad \text{ with } \quad M(v)=a_{\dd}\exp(-b_{\dd}|v|^{2}), \qquad v \in \R^{3}\,,
\end{equation}
for $a_{\dd},b_{\dd} > 0$ such that $\M_{\dd}=\frac{M}{1+\dd\,M}$ has the same mass, momentum and energy as $f_{0}$.

\begin{lem}\label{lem:Bras-Lieb} For any $\dd >0$ the following Poincar\'e inequality holds
\begin{equation}\label{eq:Bras-Lieb}
\int_{\R^{3}}\left|\nabla h(v)\right|^{2}\mM(v)\d v \geq C_{\mathrm{P}}(\dd)\int_{\R^{3}}|h(v)|^{2}\mM(v)\d v 
\end{equation}
for any $h \in L^{2}(\mM)$ with $\int_{\R^{3}}h\,\mM\, \d v=0$ and with $C_{\mathrm{P}}(\dd)=2b_{\dd}(1+\dd\,a_{\dd})^{-4}$. 
\end{lem}
\begin{proof} The proof is given in \cite[Corollary 3.4]{lemou}.  Since $$(1+\dd\,a)^{-2}M \leq \mM \leq M$$
then, using the notations of \cite{lemou}, $C_{1}/C_{2}=(1+\dd\, a_{\dd})^{-2}$.  Moreover, $U_{0}=-\log M$ is such that
$$\mathrm{Hess}(U_{0})=2\,b_{\dd}\,\mathbf{Id}$$
where $\mathrm{Hess}(U_{0})$ is the Hessian matrix of $U_{0}.$  Therefore, according to \cite[Corollary 3.4]{lemou}, we obtain the result with $C_{\mathrm{P}}(\dd)=2b_{\dd}(C_{1}/C_{2})^{2}$.
\end{proof}
\subsection{Spectral gap estimate for Maxwell potential case} For Maxwell molecules $\g=0$ we denote by $\mathcal{D}_{\dd,2}(h)$ the Dirichlet form
$$\mathcal{D}_{\dd,2}(h)=\frac{1}{2}\int_{\R^{6}}\mM\,\mM_{\ast}|v-\vet|^{2}\left|\Pi(v-\vet)\left(\nabla h-\nabla h_{\ast}\right)\right|^{2}\d v\d\vet, \qquad h \in L^{2}(\mM).$$
We can make explicit the spectral gap obtained in Theorem \ref{theo:lemou} with the following proposition.
\begin{prop}\label{maxwell_gap} There exists an explicit $\dd^{\dagger} >0$ such that for any $\dd \in (0,\dd^{\dagger})$ 
$$\mathcal{D}_{\dd,2}(h) \geq \lambda_{2}(\dd)\,\left\|h\right\|^2_{L^{2}(\mM)}$$
for any function $h \in L^{2}(\mM)$ satisfying \eqref{eq:ortho}.  Here $\lambda_{2}(\dd) >0$ is explicit and depends only on $\dd,\varrho$ and $E$.  An estimation for the range of the quantum parameter and a lower bound for the spectral gap is $\dd^{\dagger} \approx  0.6011 E^{\frac{3}{2}}\varrho^{-1}$ and $\lambda_{2}(\dd)\gtrsim 4.686 \times 10^{-4} \varrho$.
\end{prop}
\begin{proof} Our proof uses some arguments used in \cite{DeVi2} for estimating the production of entropy associated to the classical Landau equation for Maxwell molecules, see also \cite{DesvJFA}.  Fix $h \in L^{2}(\mM)$ satisfying \eqref{eq:ortho} and write
$$
R_{h}(v,\vet):=
\Pi(v-\vet)\left(\nabla h - \nabla h_{\ast}\right)
=\nabla h - \nabla h_{\ast}-\lambda_{h}\,(v-\vet)$$
for some suitable real function $\lambda_{h}=\lambda_{h}(v,\vet)$ so that
$$\mathcal{D}_{\dd,2}(h)=\frac{1}{2}\int_{\R^{6}}\mM\,\mM_{\ast}|v-\vet|^{2}\left|R_{h}(v,\vet)\right|^{2}\d v\d\vet.$$
For any \emph{fixed} circular permutation $(i,j,k)$ of $(1,2,3)$, one has
$$
\big((v-\vet) \wedge R_{h}(v,\vet)\big)_{k}=(v-\vet)_{j}\left(\partial_{i} h - \partial_{i} h_{\ast}\right)
-(v-\vet)_{i}\left( {\partial_{j} h }- {\partial_{j} h_{\ast}}\right).$$
We multiply this vectorial identity by $\varphi^{\ell}(\vet)$ and integrate over $\R^{3}$ to get
\begin{equation*}
\left(v_{j} {\partial_{i} h}-v_{i}{\partial_{j} h}\right)(\mathbf{U}_{0})_{\ell}-{\partial_{i} h}(\mathbf{U}_{j})_{\ell}+ {\partial_{j} h}(\mathbf{U}_{i})_{\ell}=
\mathbf{R}_{\ell} + \mathbf{A}_{\ell}+\mathbf{B}_{\ell}v_{j}+\mathbf{C}_{\ell}v_{i}
\end{equation*}
where we introduced the vectors $\mathbf{R},$ $\mathbf{U}_{0}$, $\mathbf{U}_{p}$, with $p \in \{i,j\}$, $\mathbf{A},$ $\mathbf{B}$ and $\mathbf{C}$ defined by
\begin{multline*}
\begin{cases}
\mathbf{R}_{\ell}:=\int_{\R^{3}}\left((v-\vet) \wedge R_{h}(v,\vet)\right)_{k}\varphi^{\ell}_{\ast} \d\vet, \qquad
(\mathbf{U}_{0})_{\ell}:=\int_{\R^{3}}\varphi^{\ell}_{\ast}\d\vet,
\\
\\
(\mathbf{U}_{p})_{\ell}=\int_{\R^{3}}v^{p}_{\ast}\varphi^{\ell}_{\ast}\d\vet, \qquad
\mathbf{B}_{\ell}=\int_{\R^{3}}\partial_{i}h\,\varphi^{\ell} \d v,\\ 
\\
\mathbf{C}_{\ell}=-\int_{\R^{3}}\partial_{j}h\,\varphi^{\ell} \d v \qquad \mathbf{A}_{\ell}=\int_{\R^{3}}v_{i}\partial_{j}h\,\varphi^{\ell}\d v- \int_{\R^{3}}v_{j}\partial_{i}h\,\varphi^{\ell}\d v, \qquad \ell=1,2,3.
\end{cases}
\end{multline*}

\medskip
\noindent
Thanks to Cramer's rule, we can solve the above linear system of equations to find
$${\partial_{j} h}=\frac{\mathrm{det}(\mathbf{U}_{0},-\mathbf{U}_{j},\mathbf{R}+ \mathbf{A}+\mathbf{B} v_{j}+\mathbf{C} v_{i})}{\mathrm{det}(\mathbf{U}_{0},-\mathbf{U}_{j},\mathbf{U}_{i})}.$$
Let us pick  
$$\varphi^{1}=\M_{\dd}, \qquad \varphi^{2}=-v_{j}\M_{\dd}, \qquad \varphi^{3}=v_{i}\M_{\dd}\,.$$
We recall that
$$\int_{\R^{3}}\M_{\dd}(v)\left(\begin{array}{c}1 \\v \\|v|^{2}\end{array}\right)\d v=\left(\begin{array}{c}\varrho \\0 \\3\varrho\,E\end{array}\right)\,,$$
and, by a symmetry argument
$$\int_{\R^{3}}\M_{\dd}(v)v_{i}^{2}\d v=\varrho\,E, \qquad i=1,2,3.$$
Then one can check that
$$\mathbf{U}_{0}=\left(\begin{array}{c}\varrho \\0 \\0\end{array}\right), \qquad \mathbf{U}_{j}=\left(\begin{array}{c}0 \\-\varrho E \\0\end{array}\right), \qquad \mathbf{U}_{i}=\left(\begin{array}{c}0 \\0 \\\varrho\,E\end{array}\right)\,,$$
which results in
$$\mathrm{det}(\mathbf{U}_{0},-\mathbf{U}_{j},\mathbf{U}_{i})=\varrho^{3}E^{2}.$$
The matrix $(\mathbf{U}_{0},-\mathbf{U}_{j},\mathbf{R}  + \mathbf{A} +\mathbf{B} v_{j}+\mathbf{C} v_{i})$ is then upper triangular.  Thus,
$$\mathrm{det}(\mathbf{U}_{0},-\mathbf{U}_{j},\mathbf{R}  + \mathbf{A} +\mathbf{B} v_{j}+\mathbf{C} v_{i})=\varrho^{2}E\left(\mathbf{R}_{3} + \mathbf{A}_{3}+\mathbf{B}_{3}v_{j}+\mathbf{C}_{3}v_{i}\right)$$
and then,
$$\partial_{j}h -\frac{1}{\varrho\,E}\left(\mathbf{A}_{3}+\mathbf{B}_{3}v_{j}+\mathbf{C}_{3}v_{i}\right)=\frac{1}{\varrho\,E}\mathbf{R}_{3}\,.$$
Taking the square and multiplying by $\M_{\dd}$ we get that
$$\int_{\R^{3}}\M_{\dd}\left(\partial_{j}h -\frac{1}{\varrho\,E}\left(\mathbf{A}_{3}+\mathbf{B}_{3}v_{j}+\mathbf{C}_{3}v_{i}\right)\right)^{2}\d v=\frac{1}{\varrho^{2}\,E^{2}}\int_{\R^{3}}\M_{\dd}(v)\mathbf{R}_{3}^{2}\d v.$$
Now, one has that $\big|\left[(v-\vet) \wedge R_{h}(v,\vet)\right]_{k}\big| \leq |v-\vet|\,|R_{h}(v,\vet)|$. Also, write $\M_{\dd}=\sqrt{M}\,\sqrt{\mM}$ where $\mM$ and $M$ were defined in \eqref{eq:mMab}.  We have that
\begin{multline*}
\mathbf{R}_{3}^{2}=\left(\int_{\R^{3}}\left[(v-\vet) \wedge R_{h}(v,\vet)\right]_{k}\vet^{i}\M_{\dd}(\vet)\d \vet\right)^{2}\\
\leq \left(\int_{\R^{3}}|v-\vet|^{2}\,|R_{h}(v,\vet)|^{2}\mM(\vet)\d\vet\right)\,\left(\int_{\R^{3}}v_{i}^{2}M(v)\d v\right).\end{multline*}
Note that
$$\int_{\R^{3}}v_{i}^{2}M(v)\d v=C_{a,b}=\frac{a_{\dd}}{2b_{\dd}}\left(\frac{\pi}{b_{\dd}}\right)^{\frac{3}{2}}$$ is independent of $i$, and since there is $c_{\dd}=(1+\dd\,a_{\dd})$ such that $\M_{\dd} \leq c_{\dd}\mM$, it follows that 
$$\int_{\R^{3}}\M_{\dd}(v)\mathbf{R}_{3}^{2}\d v \leq C_{a,b}c_{\dd}\int_{\R^{6}}|v-\vet|^{2}\,|R_{h}(v,\vet)|^{2}\mM(\vet)\mM(v)\d v\d\vet=2C_{a,b}c_{\dd}\mathcal{D}_{\dd,2}(h)\,.$$
Consequently,
$$\mathcal{D}_{\dd,2}(h) \geq \nu\,\mathcal{I}_{i,j}, \quad \forall\, i \neq j$$
where $\nu^{-1}=\dfrac{2C_{a,b}c_{\dd}}{\varrho^2\,E^2}$ and
$$\mathcal{I}_{i,j}=\int_{\R^{3}}\M_{\dd}\left(\partial_{j}h -\frac{1}{\varrho\,E}\left(\mathbf{A}_{3}+\mathbf{B}_{3}v_{j}+\mathbf{C}_{3}v_{i}\right)\right)^{2}\d v.$$
Summing up over all possible couples $(i,j)$ with $i\neq j$ we get
\begin{equation}\label{eq:Ddd2}
\mathcal{D}_{\dd,2}(h) \geq \frac{\nu}{6}\sum_{i}\sum_{j\neq i}\mathcal{I}_{i,j}.\end{equation}
Let us estimate now $\mathcal{I}_{i,j}$. We expand the square
\begin{multline*}
\mathcal{I}_{i,j}=\int_{\R^{3}}\M_{\dd}\left(\partial_{j}h\right)^{2}\d v + \frac{1}{\varrho^{2}E^{2}}\int_{\R^{3}}\M_{\dd}(v)\left(\mathbf{A}_{3}+\mathbf{B}_{3}v_{j}+\mathbf{C}_{3}v_{i}\right)^{2}\d v\\
-\frac{2}{\varrho\,E}\int_{\R^{3}}\M_{\dd}(v)\partial_{j}h(v)\left(\mathbf{A}_{3}+\mathbf{B}_{3}v_{j}+\mathbf{C}_{3}v_{i}\right)\d v\\
=\int_{\R^{3}}\M_{\dd}\left(\partial_{j}h\right)^{2}\d v + \frac{1}{\varrho^{2}E^{2}}\int_{\R^{3}}\M_{\dd}(v)\left(\mathbf{A}_{3}+\mathbf{B}_{3}v_{j}+\mathbf{C}_{3}v_{i}\right)^{2}\d v\\
+\frac{2\mathbf{B}_{3}}{\varrho\,E}\int_{\R^{3}}h\M_{\dd}\d v + \frac{2}{\varrho\,E}\int_{\R^{3}}h(v)\left(\mathbf{A}_{3}+\mathbf{B}_{3}v_{j}+\mathbf{C}_{3}v_{i}\right)\partial_{j}\M_{\dd}\d v.
\end{multline*}
Notice that
\begin{equation}\label{eq:pjM}
\partial_{j}\M_{\dd}(v)=-2\,b_{\dd}\,v_{j}\,\mM(v)
\end{equation}
thus, using \eqref{eq:ortho},
\begin{equation*}
\int_{\R^{3}}h(v)\left(\mathbf{A}_{3}+\mathbf{B}_{3}v_{j}+\mathbf{C}_{3}v_{i}\right)\partial_{j}\M_{\dd}\d v=
-2b_{\dd}\mathbf{B}_{3}\bm{e}_{j}-2b_{\dd}\mathbf{C}_{3}\int_{\R^{3}}h(v)\mM(v)v_{i}v_{j}\d v
\end{equation*}
where
$$\bm{e}_{j}:=\int_{\R^{3}}h(v)v_{j}^{2}\mM(v)\d v.$$
Now, using the fact that $\M_{\dd}$ is radially symmetric, it follows that
\begin{multline*}
\int_{\R^{3}}\M_{\dd}(v)\left(\mathbf{A}_{3}+\mathbf{B}_{3}v_{j}+\mathbf{C}_{3}v_{i}\right)^{2}\d v=\mathbf{A}^{2}_{3}\int_{\R^{3}}\M_{\dd}(v)+\mathbf{B}_{3}^{2}\int_{\R^{3}}v_{j}^{2}\M_{\dd}(v)\d v\\
+\mathbf{C}_{3}^{2}\int_{\R^{3}}v_{i}^{2}\M_{\dd}(v)\d v=\varrho\mathbf{A}^{2}_{3}+ \varrho\,E\,\left(\mathbf{B}^{2}_{3}+\mathbf{C}^{2}_{3}\right).
\end{multline*}
Collecting all the computations, we get
\begin{multline}\label{eq:Iij}
\mathcal{I}_{i,j}=\int_{\R^{3}}\M_{\dd}\left(\partial_{j}h\right)^{2}\d v +\frac{\mathbf{A}^{2}_{3}}{\varrho\,E^{2}}+\frac{1}{\varrho\,E}\left(\mathbf{B}^{2}_{3}+\mathbf{C}^{2}_{3}\right)
+\frac{2\mathbf{B}_{3}}{\varrho\,E}\int_{\R^{3}}h\M_{\dd}\d v\\
-\frac{4b_{\dd}\mathbf{B}_{3}}{\varrho\,E}\bm{e}_{j}-\frac{4b_{\dd}\mathbf{C}_{3}}{\varrho\,E}\int_{\R^{3}}h(v)\mM(v)v_{i}v_{j}\d v.
\end{multline}
Let us compute $\mathbf{A}_{3},\mathbf{B}_{3},\mathbf{C}_{3}$.  Using \eqref{eq:pjM}, we check that
$$\mathbf{A}_{3}=\int_{\R^{3}}v_{j}\M_{\dd}\,h\d v, \qquad \mathbf{C}_{3}=-2b_{\dd}\int_{\R^{3}}h\,v_{i}\,v_{j}\,\mM\d v,$$
and  
\begin{equation}\label{eq:B3}\mathbf{B}_{3}=-\int_{\R^{3}}h\M_{\dd}\d v+2b_{\dd}\bm{e}_{i}.
\end{equation}
With this at hand, identity \eqref{eq:Iij} reads
\begin{eqnarray*}
\mathcal{I}_{i,j}&=&\int_{\R^{3}}\M_{\dd}\left(\partial_{j}h\right)^{2}\d v+\frac{\mathbf{A}^{2}_{3}}{\varrho\,E^{2}}+\frac{3\mathbf{C}^{2}_{3}}{\varrho\,E}+\frac{1}{\varrho\,E}\left(\mathbf{B}_{3}^{2}+2\mathbf{B}_{3}\int_{\R^{3}}h\M_{\dd}\d v\right)-\frac{4b_{\dd}\mathbf{B}_{3}}{\varrho\,E}\bm{e}_{j}\\
&\geq & \int_{\R^{3}}\M_{\dd}\left(\partial_{j}h\right)^{2}\d v+\frac{1}{\varrho\,E}\left(\mathbf{B}_{3}^{2}+2\mathbf{B}_{3}\int_{\R^{3}}h\M_{\dd}\d v\right)-\frac{4b_{\dd}\mathbf{B}_{3}}{\varrho\,E}\bm{e}_{j}.
\end{eqnarray*}
Using \eqref{eq:B3}, one deduce that
$$\mathcal{I}_{i,j}\geq \int_{\R^{3}}\M_{\dd}\left(\partial_{j}h\right)^{2}\d v-\frac{1}{\varrho\,E}\left(\int_{\R^{3}}h\M_{\dd}\d v\right)^{2}- \frac{4b_{\dd}^{2}\bm{e}_{i}^{2}}{\varrho\,E}+\frac{4b_{\dd}}{\varrho\,E}\bm{e}_{j}\int_{\R^{3}}h\M_{\dd}\d v -\frac{8b_{\dd}^{2}}{\varrho\,E}\bm{e}_{i}\bm{e}_{j}.$$
Using \eqref{eq:ortho} we see that $\sum_{i=1}^{3}\bm{e}_{i}=\int_{\R^{3}}h\mM\,|v|^{2}\d v=0$, which implies that
$$\sum_{i}\sum_{j\neq i}\bm{e}_{j}=0, \qquad \sum_{i}\sum_{j\neq i}\bm{e}_{i}\bm{e}_{j}=-\sum_{i=1}^{3}\bm{e}_{i}^{2}\,.$$
And, therefore 
\begin{equation}\label{eq:sIij}
\sum_{i}\sum_{j\neq i}\mathcal{I}_{i,j} \geq 2\int_{\R^{3}}\M_{\dd}\left|\nabla h(v)\right|^{2}\d v-\frac{6}{\varrho\,E}\left(\int_{\R^{3}}h\M_{\dd}\d v\right)^{2}.
\end{equation}
Recalling that $\int_{\R^{3}}h\,\mM=0$, with $\mM=\M_{\dd}(1-\dd\,\M_{\dd})$, we observe that actually the negative term in \eqref{eq:sIij} is of order $\dd^{2}$. Namely,
$$\sum_{i}\sum_{j\neq i}\mathcal{I}_{i,j} \geq 2\int_{\R^{3}}\M_{\dd}\left|\nabla h(v)\right|^{2}\d v-\frac{6\dd^{2}}{\varrho\,E}\left(\int_{\R^{3}}h\M_{\dd}^{2}\d v\right)^{2}.$$
Since $\M_{\dd}^{2}=\sqrt{\mM}\,\frac{M^{3/2}}{1+\dd\,M}$, we get that
$$\left(\int_{\R^{3}}h\M_{\dd}^{2}\d v\right)^{2} \leq \left(\int_{\R^{3}}h^{2}\mM \d v\right)\,\left(\int_{\R^{3}}\frac{M^{3}}{(1+\dd\,M)^{2}}\;\d v\right) \leq a_{\dd}^{3}\left(\frac{\pi}{3b_{\dd}}\right)^{3/2}\left(\int_{\R^{3}}h^{2}\mM \d v\right)\,,$$
because 
$$\int_{\R^{3}}\frac{M^{3}}{(1+\dd\,M)^{2}}\;\d v \leq \int_{\R^{3}}M^{3}\d v =a_{\dd}^{3}\int_{\R^{3}}\exp(-3b_{\dd}|v|^{2})\d v=a_{\dd}^{3}\left(\frac{\pi}{3b_{\dd}}\right)^{3/2}.$$
Then, from \eqref{eq:Ddd2} and $\M_{\dd} \geq \mM$, one finally arrives to
$$\mathcal{D}_{\dd,2}(h) \geq \frac{\nu}{3}\int_{\R^{3}}\mM(v)\left|\nabla h(v)\right|^{2}\d v - \frac{\nu\dd^{2}}{\varrho\,E}a_{\dd}^{3}\left(\frac{\pi}{3b_{\dd}}\right)^{3/2}\int_{\R^{3}}h^2(v)\mM(v)\d v\,.$$
Using Poincar\'e inequality with constant $C_{\mathrm{P}}(\dd)$, see Lemma \ref{lem:Bras-Lieb}, it holds that
$$\mathcal{D}_{\dd,2}(h) \geq {\nu}\left(\frac{C_{\mathrm{P}}(\dd)}{3}-\frac{\dd^{2}}{\varrho\,E}a_{\dd}^{3}\left(\frac{\pi}{3b_{\dd}}\right)^{3/2}\right)\,\int_{\R^{3}}h^2(v)\mM(v)\d v$$
where we recall that $C_{\mathrm{P}}(\dd)=2b_{\dd}(1+\dd\,a_{\dd})^{-4}$.  Set
\begin{equation}\label{eq:lambda2}
\lambda_{2}(\dd):={\nu}\left(\frac{C_{\mathrm{P}}(\dd)}{3}-\frac{\dd^{2}}{\varrho\,E}a_{\dd}^{3}\left(\frac{\pi}{3b_{\dd}}\right)^{3/2}\right)\,,\end{equation}
and notice that $\lambda_{2}(\dd) >0$ if 
$$\varrho\,E > \frac{3\dd^{2}}{C_{\mathrm{P}}(\dd)}a_{\dd}^{3}\left(\frac{\pi}{3b_{\dd}}\right)^{3/2}=\frac{3\dd^{2}}{2b_{\dd}^{5/2}}\left(\frac{\pi}{3}\right)^{3/2}a_{\dd}^{3}(1+\dd\,a_{\dd})^{4}.$$ 
According to the results of Appendix \ref{app:FD}, there exists an explicit $\dd^{\dagger} >0$ such that the previous estimates holds for any $\dd \in (0,\dd^{\dagger})$.  Refer to Lemma \ref{lem:abeq2} and remarks \ref{rmq:abeq2} and \ref{rmq:spec} for details.
\end{proof}
\begin{rmq} Notice that $\dd^{\dagger}$ is proportional to $E^{\frac{3}{2}}\varrho^{-1}$ precisely as the sharp requirement \eqref{eq:bound} on $\dd$ to obtain Fermi-Dirac statistics.
\end{rmq}
\subsection{Spectral gap estimates for Hard potential case} When $\g \in (0,1]$, we compare the operator to that of Maxwell potential.  Indeed, the fact that there is a Maxwellian density $M$ and $\kappa_{\dd} >0$ such that
$$\kappa_{\dd} M \leq \mM \leq M $$ 
readily implies that 
$$\kappa_{\dd}^{2} \mathbf{D}_{M,\g+2}(h) \leq \mathcal{D}_{\dd,\g+2}(h) \leq \mathbf{D}_{M,\g+2}(h)\,, \quad \forall\, h \in L^{2}(M ) \simeq L^{2}(\mM)$$
where
$$\mathbf{D}_{M ,\g+2}(h)=\frac{1}{2}\int_{\R^{6}}M (v)\,M (\vet)|v-\vet|^{\g+2}\left|\Pi(v-\vet)\left(\nabla h-\nabla h_{\ast}\right)\right|^{2}\d v\d\vet.$$
One can deduce from \cite[Theorem 1.2]{bara} that there is an explicit constant $C_{\g,\dd} =\frac{1}{8}\left(\frac{\gamma}{8eb_{\dd}}\right)^{\frac{\gamma}{2}}>0$ such that
$$\mathbf{D}_{M ,\g+2}(h) \geq C_{\g,\dd}\,\mathbf{D}_{M ,2}(h) \geq C_{\g,\dd}\mathcal{D}_{\dd,2}(h)\,, \qquad \forall\, h \in L^{2}(M ) \simeq L^{2}(\mM)$$
with no orthogonality conditions needed for this inequality. Therefore,
$$\mathcal{D}_{\dd,\g+2}(h) \geq \kappa^2_{\dd}\,C_{\g,\dd}\,\mathcal{D}_{\dd,2}(h) \qquad \forall h \in L^{2}(\mM)\,.$$
We can exploit, then, the orthogonality condition \eqref{eq:ortho} and the result for $\mathcal{D}_{\dd,2}$ given in Proposition \ref{maxwell_gap} to obtain that
$$\mathcal{D}_{\dd,\g+2}(h) \geq \kappa^2_{\dd}\,C_{\g,\dd}\lambda_{2}(\dd)\,\int_{\R^{3}}\mM(v)h^{2}(v)\d v \qquad \text{for any $h$ satisfying } \eqref{eq:ortho}.$$
\begin{theo}\label{theo:spectral}
There exists $\dd^{\dagger} >0$ such that for any $\dd \in (0,\dd^{\dagger})$
\begin{equation*}
\mathcal{D}_{\dd,\gamma+2}(h):=-\langle h,\bm{L}_{\dd}h\rangle_{2} \geq \,\lambda_{\g}(\dd)\left\|h-\mathbb{P}_{\dd}(h)\right\|^2_{L^2(\mM)}, \qquad \forall h \in \mathscr{D}(\bm{L}_{\dd})
\end{equation*}
where $\langle \cdot,\cdot \rangle_{2}$  denotes the inner product of $L^{2}(\mM)$ and $\mathbb{P}_{\dd}$ is the projection over the null space of $\bm{L}_{\dd}$ given by $\mathrm{Span}\left\{1,v_{1},v_{2},v_{3},|v|^{2}\right\}$.

\smallskip
\noindent
The constant $\lambda_{\g}(\dd) >0$ is given by $\lambda_{\g}(\dd)=\kappa^2_{\dd}\,C_{\g,\dd}\lambda_{2}(\dd)$ where $\lambda_{2}(\dd)$ is given by Proposition \ref{maxwell_gap}.
\end{theo}
\begin{rmq} It is possible to sharpen the spectral gap given in Theorem \ref{theo:spectral} arguing as in \cite{coerci}.  In fact, there exists a positive and explicitly computable constant $\eta=\eta(\g,\dd) >0$ such that
$$\mathcal{D}_{\dd,\gamma+2}(h) \geq \eta\left(\|h-\mathbb{P}_{\dd}h\|_{H^{1}_{\g}(\mM)}^{2} + \|h-\mathbb{P}_{\dd}h\|_{L^{2}_{\g+2}(\mM)}^{2}\right).$$
 \end{rmq}

\subsection{Spectral gap estimates in $L^{2}$-spaces with polynomial weights}\label{sec:spec-L2} Introduce the operator $\mathscr{L}_{\dd}$
$$\mathscr{L}_{\dd}(h\mM):=\mM\,\bm{L}_{\dd}(h).$$
Our goal is to prove that the linearized semigroup associated to $\mathscr{L}_{\dd}$, which relaxes exponentially fast in $L^{2}(\mM)$, admits similar decay in the larger $L^{2}$ space with polynomial weights.  A suitable approach for proving such extension uses enlargement techniques developed in \cite{GMM} and has been applied to the Landau equation in \cite{Kleber}.

\smallskip
\noindent
Observe that for any $\dd > 0$ such that $\dd\,a_{\dd} <1$ it follows that $1-2\dd\M_{\dd}(v) >0$.  Consequently,
\begin{equation}\label{eq:zeta}
\bm{\zeta}_{\dd}:=\frac{2b_{\dd}}{3}\int_{\R^{3}}|w|^{2}\mM(w)\left(1-2\dd\,\M_{\dd}(w)\right)\d w >0.\end{equation}
Denote also for such $\dd$,
$$k_{\dd}^{\dagger}:=\max\bigg({\varrho(\g+3)}\bm{\zeta}_{\dd}^{-1},2\g+7\bigg),$$
which quantify the amount of moments needed for exponential relaxation.
\begin{theo}\label{theo:enlarge} For any $\dd \in (0,\dd^{\dagger})$ such that $\dd\,a_{\dd} < 1$ the following holds.  For any  $k > k_{\dd}^{\dagger}$ the operator $\mathscr{L}_{\dd}\::\mathscr{D}(\mathscr{L}_{\dd}) \subset L^{2}(\langle \cdot \rangle^{k}) \to L^{2}(\langle \cdot \rangle^{k})$ generates a $C_{0}$-semigroup $(\bm{S}_{\dd}(t))_{t\geq0}$ in $L^{2}(\langle \cdot \rangle^{k})$ and 
\begin{equation*}
\big\|\bm{S}_{\dd}(t)g-\mathbb{P}_{\dd}g\big\|_{L^{2}(\langle \cdot \rangle^{k})} \leq C_{\dd}\,\exp(-\lambda_{\star}\,t)\big\|g-\mathbb{P}_{\dd}g\big\|_{L^{2}(\langle \cdot \rangle^{k})}, \qquad \forall t \geq0\,,
\end{equation*}
for some explicit constant $C_{\dd} >0$ and any $\lambda_{\star} \in (0,\lambda_{\gamma}(\dd))$.
\end{theo}
\begin{proof} We point out the important steps since similar arguments are given in \cite{Kleber}.  The reader can find the details in Appendix \ref{app:kleb}.  From \eqref{eq:Lddhm}, one has
\begin{multline*}\label{eq:Lddg}
\mathscr{L}_{\dd}(g)=
\nabla \cdot \int_{\R^{3}}a(v-\vet)\left[\mM_{\ast}\nabla g - \mM\,(\nabla g)_{\ast}\right]\d \vet \\
+ \nabla \cdot \int_{\R^{3}}a(v-\vet)\bigg[ g_{\ast}\,(1-2\dd\M_{\dd})_{\ast}\nabla \M_{\dd} - g\,(1-2\dd\M_{\dd})(\nabla \M_{\dd})_{\ast}\bigg]\d \vet\,.
\end{multline*} 
The Landau bilinear operator is given by 
\begin{equation}\label{eq:QL}
\Q(G,F)=\nabla \cdot \int_{\R^{3}} a(v-\vet)\left[G_{\ast}\nabla F-G (\nabla F)_{\ast}\right]\d\vet
\end{equation}
for any suitable functions $G$, $F$.  It is not difficult to check that
$$\mathscr{L}_{\dd}(g)=\Q(\M_{\dd},g)+\Q(g,\M_{\dd})-2\dd\,\Q(g\,\M_{\dd},\M_{\dd})-\dd\,\Q(\M_{\dd}^{2},g)$$
where we used that $\mM=\M_{\dd}-\dd\,\M_{\dd}^{2}$. 
For a smooth nonnegative function $\chi \in \mathscr{C}_{c}^{\infty}(\R^{3})$ such that $0 \leq \chi \leq 1$, $\chi(v)=1$ for $\{|v| \leq 1\}$ and $\chi(v)=0$ for $\{|v| >2\}$, we define
$$\chi_{R}(v)=\chi(R^{-1}v), \qquad \qquad R >1, \:v \in \R^{3}.$$
Then, for suitable constant $\U >0$ to be chosen later, we set
$$\bm{A}_{\dd}g=\Q(g,\M_{\dd}) + \U\chi_{R}g-2\dd\Q(g\M_{\dd},\M_{\dd})$$  and
$$\bm{B}_{\dd}g=\mathscr{L}_{\dd}g-\bm{A}_{\dd}g=\sum_{i,j=1}^{3}\bm{\sigma}_{i,j}[\M_{\dd}]\,\partial^2_{ij}g-\bm{c}[\M_{\dd}]\,g -\dd\,\Q(\M_{\dd}^{2},g)-\U\chi_{R}g\,.$$
Setting $\mathcal{E}=L^{2}(\langle \cdot \rangle^{k})=L^{2}_{k}(\R^{3})$ and $E=L^{2}(\M_{\dd}^{-1})$, one sees from Proposition \ref{prop:Badiss} that, for any $a \in \R,$ one can choose $\U,R$ large enough such that $\bm{B}_{\dd}-a$ is dissipative in $E$. Moreover,  $\bm{A}_{\dd} \in \mathscr{B}(\mathcal{E})$ thanks to Lemma \ref{lem:Abounded} and  $E \subset \mathcal{E}$. Using Lemma \ref{lem:AU}, it is not difficult to notice that the splitting $\mathscr{L}_{\dd}=\bm{A}_{\dd}+\bm{B}_{\dd}$ meets all the properties of \cite[Theorem 2.4]{Kleber}, see also \cite[Theorem 2.13]{GMM}.  Therefore, it is possible to conclude as in \cite[Theorem 2.1]{Kleber}.
\end{proof}
\begin{rmq} Theorem \ref{theo:enlarge} is valid in any $L^{p}(\langle \cdot \rangle^{k})$ spaces, with $p \geq 1$, since the results in Appendix \ref{app:kleb} can be extended to such spaces following \cite{Kleber}.
\end{rmq}
\begin{rmq} We can show that, actually, there is $C >0$ independent of $\dd$ such that $C_{\dd} \leq C$ as soon as $\dd\,a_{\dd} <1$.  {Notice that the case $\dd=0$ has been considered in \cite{Kleber}.} 
\end{rmq}

\section{Quantitative convergence to equilibrium} \label{sec:entrop}
We prove here Theorem \ref{theo:converg}.  This proof will resort on a combination of previous spectral analysis and suitable entropy production estimates.
\subsection{Close to equilibrium estimates}
Consider an initial datum  $0\leq f_0\in L^{1}_{s_0}(\R^{3})$, with $s_0>2$, satisfying \eqref{hypci}--\eqref{hyp:mass}.  Let $\dd \in (0,\dd_{0}]$ and $f=f(t,v)$ be a weak solution to \eqref{LFDeq} given by Theorem \ref{results}, and let $\M_{\dd}$ be the unique Fermi-Dirac statistics satisfying \eqref{hyp:mass}. We introduce the fluctuation
$$g=f - \M_{\dd}$$
which satisfies
\begin{equation}\label{eq:partgt} 
\partial_{t}g = \mathscr{L}_{\dd}g + \bm{\Gamma}(g)\,, \quad t > 0, \qquad  g_{0}=f_{0} - \M_{\dd},
\end{equation}
where, with the notations of Section \ref{sec:spectral},
$$\bm{\Gamma}(g)=\mM\,\Gamma_{2,\dd}(\mM^{-1}g)+\mM\,\Gamma_{3,\dd}(\mM^{-1}g).$$
One can  check that
\begin{equation}\label{eq:GQ}
\bm{\Gamma}(g)=\Q(g(1-2\dd\M_{\dd}),g)-\dd\,\Q(g^{2},\M_{\dd}+g)\end{equation}
where $\Q(G,F)$ is the bilinear Landau operator defined in \eqref{eq:QL}. We have the following estimates for $\bm{\Gamma}$.
\begin{lem}\label{lem:estG} For any $k \geq 0$, there exists $C=C(k,\g) >0$ such that
\begin{equation*}% \begin{split}
\left\|\bm{\Gamma}(g)\right\|_{L^{2}(\langle \cdot \rangle^{k})} 
\leq C\,\left(\|g\|_{L^{1}_{\g+2}}+\dd\|g^{2}\|_{L^{1}_{\g+2}}\right) \|g\|_{H^{2}_{{k+2\g+4}}}+ C \dd\|g^{2}\|_{L^{1}_{\g+2}}
% \left\|\bm{\Gamma}(g)\right\|_{L^{2}(\langle \cdot \rangle^{k})} &\leq C\,\left(\|g\|_{L^{1}(\langle \cdot \rangle^{\g+2})}+\dd\|g^{2}\|_{L^{1}(\langle \cdot \rangle^{\g+2})}\right)
% \left(1+\|\nabla^{2}g\|_{L^{2}(\langle \cdot \rangle^{k+\g+2})}+\|g\|_{L^{2}(\langle \cdot \rangle^{k+\g})}\right)\\
% &\leq C\,\left(\|g\|_{L^{1}(\langle \cdot \rangle^{\g+2})}+\dd\|g^{2}\|_{L^{1}(\langle \cdot \rangle^{\g+2})}\right)\left(1+\|g\|_{H^{2}_{{2(k+\g+2)}}}\right)
% .\end{split}
\end{equation*}
\end{lem}
\begin{proof} From \cite[Proposition 3.1]{Kleber}, there is $C_{k,\g} >0$ such that
$$\|\Q(G,F)\|_{L^{2}(\langle \cdot \rangle^{k})} \leq C_{k,\g} \left(\left\|G\right\|_{L^{1}(\langle \cdot \rangle^{\g+2})}\|\nabla^{2}F\|_{L^{2}(\langle \cdot \rangle^{k+2\g+4})}+\|G\|_{L^{1}(\langle \cdot \rangle^{\g})}\,\|F\|_{L^{2}(\langle \cdot \rangle^{k+2\g})}\right),$$
from which we deduce first that
$$\|\Q(g(1-2\dd\M_{\dd}),g)\|_{L^{2}(\langle \cdot \rangle^{k})} \leq 2C_{k,\g}\|g\|_{L^{1}(\langle \cdot \rangle^{\g+2})}\left(\|\nabla^{2}g\|_{L^{2}(\langle \cdot \rangle^{k+2\g+4})}+\|g\|_{L^{2}(\langle \cdot \rangle^{k+2\g}}\right)$$
since $\|1-2\dd\M_{\dd}\|_{\infty} \leq 2$.  Noticing that, see Lemma \ref{lem:HessM}, 
$$\sup_{\dd \in (0,1)}\left(\|\nabla^{2}\M_{\dd}\|_{L^{2}(\langle \cdot \rangle^{k+2\g+4})} + \|\M_{\dd}\|_{L^{2}(\langle \cdot \rangle^{k+2\g})}\right) < \infty,$$
one finds a positive constant $C_{k,\g}'>0$ depending only on $k,\g$ such that
$$\|\Q(g^{2},\M_{\dd})\|_{L^{2}(\langle \cdot \rangle^{k})} \leq C_{k,\g}'\|g^{2}\|_{L^{1}(\langle \cdot \rangle^{\g+2})}\,.$$
And also,
$$\|\Q(g^{2},g)\|_{L^{2}(\langle \cdot \rangle^{k})} \leq C_{k,\g}\|g^{2}\|_{L^{1}(\langle \cdot \rangle^{\g+2})}\left(\|\nabla^{2}g\|_{L^{2}(\langle \cdot \rangle^{k+2\g+4})}+\|g\|_{L^{2}(\langle \cdot \rangle^{k+2\g})}\right).$$
One concludes from \eqref{eq:GQ}. 
\end{proof}
\begin{prop}\label{prop:close} Let $f_{0}$ satisfy Assumption \ref{ci}, $\dd \in (0,\dd_{0}]$ be such that $\dd\,a_{\dd} <1$, and $k > k_{\dd}^{\dagger}$.  Let $f$ be a solution to \eqref{LFDeq}.  Assume that there exists $\delta > 0$ such that
$$\|f(t)-\M_{\dd}\|_{L^{2}(\langle \cdot \rangle^{k})} \leq \delta\,,  \qquad \forall\, t \geq t_{0}\geq0,$$
and that
\begin{equation}\label{eq:boundclose}
\sup_{t\geq t_{0}}\|f(t)\|_{{H^{4}_{k+4\g+8}}} \leq C_{t_{0}}\,.
\end{equation}
Then, there is $C >0$ depending on $t_{0}$, $C_{t_{0}}$, $\delta$, $k$, $\g$ such that for any $\lambda_{\star}\in(0,\lambda_{\gamma}(\dd))$,
$$\|f(t)-\M_{\dd}\|_{L^{2}(\langle \cdot \rangle^{k})} \leq C\exp(-\lambda_{\star}t)\|f(t_{0}) - \M_{\dd}\|_{L^{2}(\langle \cdot \rangle^{k})}, \qquad \forall\, t \geq t_{0}.$$
\end{prop}
\begin{proof} Set $g=f - \M_{\dd}$.  Since $g$ is a solution to \eqref{eq:partgt}, according to Duhamel formula for any $t_{0} \geq0$,
$$g(t)=\mathcal{S}_{\dd}(t-t_{0})g(t_{0})+\int_{t_{0}}^{t}\mathcal{S}_{\dd}(t-s)\bm{\Gamma}\big(g(s)\big)\d s$$
as soon as $g(t_{0}) \in L^{2}(\langle \cdot \rangle^{k})$ and $\bm{\Gamma}\big(g(s)\big) \in L^{2}(\langle \cdot \rangle^{k})$ for any $s \in [t_{0},t)$.  Notice that  both $f(t_{0})$ and $\M_{\dd}$ satisfy \eqref{hyp:mass} while $\Q$ preserves mass, momentum and energy.  Therefore,
$$\mathbb{P}_{\dd}g(t_{0})=0 \qquad \text{ and } \qquad \mathbb{P}_{\dd}\bm{\Gamma}\big(g(s)\big)=0 \quad  \forall\, s \geq t_0.$$ In particular, one deduces from  Theorem \ref{theo:enlarge} that for $\dd \in (0,\dd^{\dagger})$ with $\dd\,a_{\dd} <1$ it holds 
\begin{equation*}
\|g(t)\|_{L^{2}(\langle \cdot \rangle^{k})} \leq C\exp(-\lambda_{\star}(t-t_{0}))\|g(t_{0})\|_{L^{2}(\langle \cdot \rangle^{k})} 
+C\int_{t_{0}}^{t}\exp(-\lambda_{\star}(t-s))\|\bm{\Gamma}\big(g(s)\big)\|_{L^{2}(\langle \cdot \rangle^{k})}\d s\,,
\end{equation*}
with $\lambda_{\star}\in(0,\lambda_{\gamma}(\dd))$.  According to Lemma \ref{lem:estG}, we get, for $t \geq t_{0}$
\begin{multline*}
\|g(t)\|_{L^{2}(\langle \cdot \rangle^{k})} \leq C\exp(-\lambda_{\star}(t-t_{0}))\|g(t_{0})\|_{L^{2}(\langle \cdot \rangle^{k})}  \\
+C_{k,\g}\int_{t_{0}}^{t}\exp(-\lambda_{\star}(t-s))\left(\|g(s)\|_{L^{1}_{\g+2}}+\dd\big\|g(s)^{2}\big\|_{L^{1}_{\g+2}}\right)\|g(s)\|_{H^{2}_{{2(k+\g+2)}}}\d s \\
+C_{k,\g}\dd\int_{t_{0}}^{t}\exp(-\lambda_{\star}(t-s))\big\|g(s)^{2}\big\|_{L^{1}_{\g+2}}\d s.
\end{multline*}
Now, since 
$$\dd\|g(s)^{2}\|_{L^{1}_{\g+2}}=\dd\|g(s)\|^2_{L^{2}_{\g+2}}\leq \|g(s)\|^{\frac{3}{2}}_{L^{2}(\langle\cdot\rangle^k)} \|g(s)\|_{H^{4}_{k+4\g+8}}^{\frac{1}{2}},$$
and $\|g(s)\|_{L^\infty} \leq \|f(s)\|_{L^\infty}+\|\M_{\dd}\|_{L^\infty} \leq \frac{2}{\dd}$ we get 
\begin{multline*}
\|g(t)\|_{L^{2}(\langle \cdot \rangle^{k})} \leq C\exp(-\lambda_{\star}(t-t_{0}))\|g(t_{0})\|_{L^{2}(\langle \cdot \rangle^{k})}\\
+3C_{k,\g}\int_{t_{0}}^{t}\exp(-\lambda_{\star}(t-s)) \|g(s)\|_{L^{1}_{\g+2}}\,\|g(s)\|_{H^{2}_{k+2\g+4}}\d s\\
+C_{k,\g}\int_{t_{0}}^{t}\exp(-\lambda_{\star}(t-s)) \|g(s)\|^{\frac{3}{2}}_{L^{2}_{\g+2}} \|g(s)\|_{H^{4}_{k+4\g+8}}^{\frac{1}{2}}\d s.
\end{multline*}
Notice that as soon as $k > 7+2\g$ one has  
$$\|g(s)\|_{L^{1}_{\g+2}} \leq \|g(s)\|_{L^{2}(\langle \cdot \rangle^{k})}\left(\int_{\R^{3}}\langle v\rangle^{2\g+4-k}\d v\right)^{\frac{1}{2}}=C(k,\g)\|g(s)\|_{L^{2}(\langle\cdot\rangle^{k})}.$$
Also, using the interpolation inequality, valid for any $\theta \in (0,1)$, $p=\theta\,p_{1}+(1-\theta)\,p_{2}$, $\ell=\theta\,\ell_{1}+(1-\theta)\ell_{2}$,
$$\|g\|_{H^{p}_{\ell}} \leq C_{p,\ell,\theta}\|g\|_{H^{p_{1}}_{\ell_{1}}}^{\theta}\,\|g\|_{H^{p_{2}}_{\ell_{2}}}^{1-\theta}$$
with $p=2$, $\ell=k+2\g+4$, $p_{1}=0$, $\ell_{1}=k$ and, say $\theta=\frac{1}{2}$.  Thus, $p_{2}=4$ and $\ell_{2}=2(k+2\g+4)$.  We obtain that
$$\|g\|_{H^{2}_{k+2\g+4}} \leq C_{k,\g}\|g\|_{L^{2}(\langle \cdot \rangle^{k})}^{\frac{1}{2}}\,\|g\|_{H^{4}_{k+4\g+8}}^{\frac{1}{2}}.$$
And therefore, for any $s \in (t_{0},t)$
$$ \|g(s)\|_{L^{1}_{\g+2}}\,\|g(s)\|_{H^{2}_{k+2\g+4}} \leq C(k,\g)\|g(s)\|_{L^{2}(\langle \cdot \rangle^{k})}^{\frac{3}{2}}\,\|g(s)\|_{H^{4}_{k+4\g+8}}^{\frac{1}{2}}$$
which result from Lemma \ref{lem:HessM}, under the assumptions of the Proposition and noticing that \eqref{eq:boundclose} yields a similar bound for $g(t)$, that
\begin{equation*}
\|g(t)\|_{L^{2}(\langle \cdot \rangle^{k})} \leq C\exp(-\lambda_{\star}(t-t_{0}))\|g(t_{0})\|_{L^{2}(\langle \cdot \rangle^{k})}  + 
 \tilde{C}(k,\g,t_{0})\delta^{\frac{1}{2}}\int_{t_{0}}^{t}\exp(-\lambda_{\star}(t-s))\|g(s)\|_{L^{2}(\langle \cdot \rangle^{k})}\d s\end{equation*}
for some positive constant $\tilde{C}(k,\g,t_{0})$ depending only on $k,\g,t_{0}$. We conclude as in \cite[Lemma 4.5]{mouhotcmp}, provided $\delta$ and, consequently, $\|g(t_{0})\|_{L^{2}(\langle \cdot \rangle^{k})}$ are sufficiently small, that
$$\|g(t)\|_{L^{2}(\langle \cdot \rangle^{k})} \leq C(k,\g,\delta,t_{0})\exp(-\lambda_{\star}(t-t_{0}))\|g(t_{0})\|_{L^{2}(\langle \cdot \rangle^{k})}\,.$$ 
The proof is achieved.
\end{proof}
\begin{rmq}\label{rmq:ab1} Using Lemma \ref{lem:abeq2} one proves the threshold value 
$$\dd_{1}=\left(\frac{3}{5}\right)^{\frac{5}{2}}\frac{(2\pi\,E)^{\frac{3}{2}}}{\varrho}\,,$$
for which $\dd\,a_{\dd} <1$ for $\dd \in (0,\dd_{1}).$
\end{rmq}
\subsection{Entropy/Entropy production}\label{sec:entropy} Recall the definition \eqref{eq:FDentro} of the entropy ${\mathcal{S}}_{\dd}(f)$ for Fermi-Dirac particles. Under Assumption \ref{ci} and for $\dd \in (0,\dd_{0}]$, observe that the function $t\longmapsto {\mathcal{S}}_{\dd}(f)(t)$ is non-decreasing for any smooth solution $f=f(t)$ to the LFD equation \eqref{LFDeq}.  Indeed, for any smooth function $g$ with $0 < g < \dd^{-1}$ note that
\begin{equation}\label{eq:Xidd}\begin{split}
\bm{\Xi}_{\dd}[g](v,\vet)&=\Pi(v-\vet)\big(g_{\ast}(1-\dd g_{\ast})\nabla g - g(1-\dd g)\nabla g_{\ast}\big)\left(\frac{\nabla g}{g(1-\dd g)}-\frac{\nabla g_{\ast}}{g_{\ast}(1-\dd g_{\ast})}\right)\\
&=gg_{\ast}(1-\dd g)(1-\dd g_{\ast})\left|\Pi(v-\vet)\left(\frac{\nabla g}{g(1-\dd g)}-\frac{\nabla g_{\ast}}{g_{\ast}(1-\dd g_{\ast})}\right)\right|^{2} \geq 0\,.
\end{split}
\end{equation}
Then, for a smooth solution to $f=f(t)$ of \eqref{LFD}, the evolution of the entropy is given by
\begin{equation*}
\begin{split}
\dfrac{\d}{\d t}&\mathcal{S}_{\dd}(f(t))=-\int_{\R^{3}}\partial_{t}f(t,v)\big(\log(\dd f(t,v))-\log(1-\dd f(t,v))\big)\d v\\
&=\int_{\R^{6}}\Psi(v-\vet)\Pi(v-\vet)\bigg[\fet (1-\dd \fet)\nabla f -f(1-\dd f)\nabla \fet\bigg]\left[\frac{\nabla f}{f(1-\dd f)}\right]\d v\d\vet\\
&=\frac{1}{2}\int_{\R^{3}\times \R^{3}}\Psi(v-\vet)\bm{\Xi}_{\dd}[f(t)](v,\vet)\d v\d\vet \geq0.
\end{split}
\end{equation*}
Thus, we define the entropy production functional $\mathscr{D}_{\dd,\Psi}$ as
\begin{equation}\label{eq:product}
\mathscr{D}_{\dd,\Psi}(g):=\frac{1}{2}\int_{\R^{3}\times\R^{3}}\Psi(v-\vet)\bm{\Xi}_{\dd}[g](v,\vet)\d v\d\vet\,,\qquad \Psi(z)=|z|^{\g+2}\,,
\end{equation}
for any smooth function $0 < g < \dd^{-1}$.  Therefore,
\begin{equation}\label{eq:entro}
\dfrac{\d}{\d t}\mathcal{S}_{\dd}(f(t))=\mathscr{D}_{\dd,\Psi}(f(t)), \qquad \forall t \geq0\,.
\end{equation}
When the choice of $\Psi(z)=|z|^{\g+2}$ is clear from the context, we will simply write $\mathscr{D}_{\dd}$ instead of $\mathscr{D}_{\dd,\Psi}$.
We begin with a comparison between the entropy production $\mathscr{D}_{\dd}$ for the Landau-Fermi-Dirac operator and that of the Landau operator $\mathscr{D}_{0}.$ This comparison is valid for functions $f$ satisfying a suitable lower bound.
\begin{lem}\label{lem:comparison} Fix $\dd >0$ and let $0 \leq f \leq \dd^{-1}$ be a function such that
\begin{equation}\label{eq:lowek0}
\inf_{v \in \R^{3}}\big(1-\dd\,f(v)\big)=\kappa_{0}>0\,.
\end{equation}
Then, 
\begin{equation}\label{eq:D0Da}
\kappa_{0}^{2}\mathscr{D}_{0}(f) \leq 2\mathscr{D}_{\dd}(f) + \frac{4\dd^{2}}{\kappa_{0}}\,\int_{\R^{6}}f\,f_{\ast}\,|v-v_{\ast}|^{\g+2}\,\left|\nabla f(v)\right|^{2}\,\d v\d v_{\ast}.\end{equation}
\end{lem}
\begin{proof} Recall the representation of $\mathscr{D}_{\dd}(f)$ as
$$\mathscr{D}_{\dd}(f)=\frac{1}{2}\int_{\R^{3}\times \R^{3}}|v-\vet|^{\gamma+2}\,\bm{\Xi}_{\dd}[f](v,\vet)\d v\d\vet$$
with $\bm{\Xi}_{\dd}$ defined in \eqref{eq:Xidd},
$$\bm{\Xi}_{\dd}[f](v,\vet)=F\,F_{\ast}\left|\Pi(v-\vet)\left(\frac{\nabla f}{F}-\frac{\nabla f_{\ast}}{F_{\ast}}\right)\right|^{2}$$
where we used the shorthand notation $F=F^{(\dd)}=f(1-\dd\,f)$.  Notice in particular that such representation is valid for any $\dd\geq0$ and, in particular, for $\dd=0$, $F^{(0)}=f$. For $f$ satisfying \eqref{eq:lowek0}, it follows that
\begin{align}\label{eq:Xi0}
\begin{split}
\bm{\Xi}_{0}[f](v,\vet)&=f\,f_{\ast}\left|\Pi(v-\vet)\left(\frac{\nabla f}{f}-\frac{\nabla f_{\ast}}{f_{\ast}}\right)\right|^{2}\\
&\hspace{2cm}\leq \kappa_{0}^{-2}F\,F_{\ast}\left|\Pi(v-\vet)\left(\frac{\nabla f}{f}-\frac{\nabla f_{\ast}}{f_{\ast}}\right)\right|^{2}.
\end{split}
\end{align}
Writing $f=\frac{F}{1-\dd\,f}$, we obtain that
\begin{align*}
\left(\frac{\nabla f}{f}-\frac{\nabla f_{\ast}}{f_{\ast}}\right)=\bigg((1-\dd\,f)\frac{\nabla f}{F} &- (1-\dd\,f_{\ast})\frac{\nabla f_{\ast}}{F_{\ast}}\bigg)
\\
&=\left(\frac{\nabla f}{F}-\frac{\nabla f_{\ast}}{F_{\ast}}\right)-\dd\,\left(\frac{\nabla f}{1-\dd\,f}-\frac{\nabla f_{\ast}}{1-\dd\,f_{\ast}}\right).\end{align*}
Thus, 
\begin{align*}
\left|\Pi(v-\vet)\left(\frac{\nabla f}{f}-\frac{\nabla f_{\ast}}{f_{\ast}}\right)\right|^{2} &\leq 2\left|\Pi(v-\vet)\left(\frac{\nabla f}{F}-\frac{\nabla f_{\ast}}{F_{\ast}}\right)\right|^{2}\\
& \hspace{1.5cm} +  2 \dd^{2}\left|\Pi(v-\vet)\left(\frac{\nabla f}{1-\dd\,f}-\frac{\nabla f_{\ast}}{1-\dd\,f_{\ast}}\right)\right|^{2}
\end{align*}
where the last term can be estimated as
\begin{eqnarray*}
\left|\Pi(v-\vet)\left(\frac{\nabla f}{1-\dd\,f}-\frac{\nabla f_{\ast}}{1-\dd\,f_{\ast}}\right)\right|^{2} & \leq & 
2\left|\frac{\nabla f}{1-\dd\,f}\right|^{2}+2\left|\frac{\nabla f_{\ast}}{1-\dd\,f_{\ast}}\right|^{2} \\
&  \leq &  2 \kappa_{0}^{-1}\left(\frac{|\nabla f|^{2}}{1-\dd\,f}+\frac{|\nabla f_{\ast}|^{2}}{1-\dd\,f_{\ast}}\right).
\end{eqnarray*}
Multiplying these last two inequalities by $F\,F_{\ast}$ and inserting this in \eqref{eq:Xi0}, it follows that
\begin{equation*}
\kappa_{0}^{2}\, \bm{\Xi}_{0}[f](v,\vet) \leq 2\,\bm{\Xi}_{\dd}[f](v,\vet) + \frac{4\dd^{2}}{\kappa_{0}}f\,f_{\ast}\big(|\nabla f|^{2}+|\nabla f_{\ast}|^{2}\big).
\end{equation*}
Multiplying now by $\frac{1}{2}|v-\vet|^{\g+2}$ and integrating over $\R^{6}$ yields the result.
\end{proof}
\begin{prop}\label{prop:DoDaka} Consider $0\leq f_0\in L^{1}_{s_0}(\R^{3})$, with $s_0>2$, satisfying \eqref{hypci} and a solution $f(t,v)$ to \eqref{LFDeq} with $\dd \in (0,\dd_{0}]$ given by Theorem \ref{results}. Then, for any $\kappa_{0} \in (0,1)$ there exists $\dd_{\star} \in (0,\dd_{0}]$ such that
\begin{equation}\label{eq:lower}
\kappa_{0}^{2}\,\mathscr{D}_{0}(f(t)) \leq 2\mathscr{D}_{\dd}(f(t)) + C_{1}\dd^{2}, \qquad \forall\, \dd \in (0,\dd_{\star})\,,\; t \geq 1,
\end{equation}
for some constant $C_{1} >0$.
\end{prop}
\begin{proof} According to Lemma \ref{lem:comparison} and Corollary \ref{cor:Linfty}, for any $\kappa_{0} \in (0,1)$ there is $\dd_{\star} >0$ such that for any $\dd \in (0,\dd_{\star})$, the solution $f=f(t)$ to \eqref{LFDeq} satisfies \eqref{eq:D0Da} for $t \geq 1$. Now, one has that
\begin{multline*}
\int_{\R^{6}}f(t,v)\,f(t,v_{\ast})\,|v-v_{\ast}|^{\g+2}\,\left|\nabla f(t,v)\right|^{2}\,\d v\d v_{\ast}\\
 \leq \|f(t)\|_{\infty}\left(\int_{\R^{3}}f(t,v_{\ast})\langle v_{\ast}\rangle^{\g+2}\d \vet\right)
\left(\int_{\R^{3}}\langle v\rangle^{\g+2}\left|\nabla f(t,v)\right|^{2}\d v\right)\,.
\end{multline*}
All terms are bounded, uniformly with respect to $\dd$, as soon as $t \geq 1$ according to Theorem \ref{smoothness} and Corollary \ref{cor:Linfty}.
\end{proof}
\noindent
Now we compare the relative entropies.  Given nonnegative $f,\,g \in L^1_2(\R^3)$ with $0 \leq f \leq \dd^{-1}$ and  $0 \leq g \leq \dd^{-1}$, set
$$\mathcal{H}_{\dd}(f|g)=-\mathcal{S}_{\dd}(f)+\mathcal{S}_{\dd}(g)$$
and 
$$\mathcal{H}_{0}(f|g)=H(f)-H(g)=\int_{\R^{3}}f\log f \d v -\int_{\R^{3}}g\log g\, \d v.$$
The Fermi-Dirac entropy $\mathcal{S}_{\dd}(f)$ \emph{does not} converges, as $\dd \to 0$, towards the classical entropy $H(f)$.  It actually diverges like $|\log \dd|$, however, the \emph{relative entropies} satisfy the following property.
\begin{lem}\label{lem:HddH0}
Let $f,g \in L^{1}(\R^{3})$ with $\int_{\R^{3}}g(v)\d v=\int_{\R^{3}}f(v)\d v$ and $ 0 \leq f,g \leq \dd^{-1}$. Then
$$\bigg|\mathcal{H}_{\dd}(f|g)-\mathcal{H}_{0}(f|g)\bigg| \leq \dd\,\max(\|f\|_{L^2}^{2},\|g\|_{L^2}^{2}).$$
\end{lem}
\begin{proof} It is easy to check that
\begin{multline*}
\mathcal{H}_{\dd}(f|g)=\mathcal{H}_{0}(f|g) + \log \dd\,\left(\int_{\R^{3}}f\d v - \int_{\R^{3}}g\d v\right) \\
+\frac{1}{\dd}\int_{\R^{3}}\bigg[(1-\dd\,f)\log(1-\dd\,f) -(1-\dd\,g)\log(1-\dd\,g)\bigg]\d v.\end{multline*}
Thus, if the masses of $f$ and $g$ concide we obtain that
\begin{equation*}
\mathcal{H}_{\dd}(f|g)-\mathcal{H}_{0}(f|g)=\frac{1}{\dd}\int_{\R^{3}}\bigg[(1-\dd f)\log(1-\dd f) - (1-\dd g)\log(1-\dd g)\bigg]\d v.\end{equation*}
Using inequality $r-r^{2}  \leq -(1-r)\log(1-r) \leq r$, for any $r \in (0,1)$, we get that
$$\int_{\R^{3}}(g-f)\d v-\dd\,\int_{\R^{3}}g^{2}\d v \leq \mathcal{H}_{\dd}(f|g)-\mathcal{H}_{0}(f|g) \leq \dd \int_{\R^{3}}f^{2}\d v +\int_{\R^{3}}(g-f)\d v\,.$$
The result follows as $f$ and $g$ share the same mass.
\end{proof}
\begin{prop}\label{prop:compHaH0Ma}  Consider $0\leq f_0\in L^{1}_{s_0}(\R^{3})$, with $s_0>2$, satisfying \eqref{hypci}.  Then, for any $\dd \in (0,\dd_{0}]$ and solution $f(t,v)$ to \eqref{LFDeq} given by Theorem \ref{results} it holds that
\begin{equation}\label{eq:entropy}
\bigg|\mathcal{H}_{\dd}(f(t)|{\M}_{\dd})-\mathcal{H}_{0}(f(t)|\M_{\dd})\bigg| \leq C\dd, \qquad \forall\, t \geq 1,
\end{equation}
for a constant $C >0$ independent of $\dd$.  The function $\M_{\dd}$ is the Fermi-Dirac distribution with same mass, momentum, and energy as $f_{0}$.
\end{prop}
\begin{proof} The proof is a direct consequence of Lemma \ref{lem:HddH0}, the uniform bound on $\sup_{t\geq 1}\|f(t)\|_{L^2}$ given in Theorem \ref{smoothness}, and the bound $\sup_{0< \dd <1}\|\M_{\dd}\|_{L^2}$ given in Lemma \ref{lem:HessM}.
\end{proof}
\begin{rmq} It is possible to replace in \eqref{eq:entropy} $\M_{\dd}$ for the Maxwellian distribution $\M_{0}$ with same mass, momentum, and energy as $f_{0}$ as long as $\dd\leq {(2\pi E)^{3/2}}{\varrho}^{-1}$ where $\varrho$, $E>0$ were defined through \eqref{hyp:mass}. The inequality  $\dd\leq {(2\pi E)^{3/2}}{\varrho}^{-1}$ ensures that $0\leq \M_{0}\leq  \dd^{-1}$.
\end{rmq}
\begin{prop}\phantomsection\label{theo:Vill} For a given nonnegative function $f \in L^{1}_{2}(\R^{d}) \cap L^{2}(\R^{d})$ sufficiently smooth, let $M_f$ denote the Maxwellian function with the same mass $\varrho_{f}$, momentum $\bm{u}_{f}$, and energy $E_{f}$ as $f$. Then, there exist two constant $\lambda_{1}$, $\lambda_{2}$ depending on $f$ only through its mass and energy such that
\begin{equation*}
\mathscr{D}_{0}(f) \geq \min\left(\lambda_{1}\mathcal{H}_{0}(f|M_{f})\,;\,\lambda_{2}\mathcal{H}_{0}(f|M_{f})^{1+\frac{\g}{2}}\right).
\end{equation*}
\end{prop}
\begin{proof} See \cite[Theorem 8]{DeVi2}.  To exhibit the role played by the parameter $\g \in (0,1)$ in the estimates, we introduce, for any $s \geq-2$ the entropy production $\mathscr{D}_{0,s}(f)$ for Landau operator corresponding to $\Psi(v-\vet)=|v-\vet|^{s+2}$, that is
\begin{equation*}
\begin{split}
\mathscr{D}_{0,s}(f)&=\frac{1}{2}\int_{\R^{6}}|v-\vet|^{s+2}\bm{\Xi}_{0}[f](v,\vet)\d v\d\vet\\
&=\frac{1}{2}\int_{\R^{6}}|v-\vet|^{s+2}f\,f_{\ast}\left|\Pi(v-\vet)\left(\frac{\nabla f}{f}-\frac{\nabla f_{\ast}}{f_{\ast}}\right)\right|^{2}\d v\d\vet\end{split}\end{equation*}
For any $\g >0$, the estimate
$$|v-\vet|^{\g+2} \geq \delta^{\g}|v-\vet|^{2} -\delta^{\g+2}\mathbf{1}_{|v-\vet| \leq \delta}\,, \qquad \forall\, \delta >0, \:\:(v,\vet) \in \R^{6}$$
yields
$$\mathscr{D}_{0,\g}(f) \geq \delta^{\g}\mathscr{D}_{0,0}(f) -\delta^{\g+2}\mathscr{D}_{0,-2}(f), \qquad \forall\, \delta >0.$$
Let us first bound $\mathscr{D}_{0,-2}(f)$ from above.  We define $\overline{\Pi}=\Pi \ast f$ and check that
\begin{equation*}
\mathscr{D}_{0,-2}(f)=\int_{\R^{3}}\big\langle \overline{\Pi}(v)\,\frac{\nabla f}{f}\,,\,\nabla f\big\rangle \d v - 2\int_{\R^{6}}|v-\vet|^{-2}f(v)f(\vet)\d v\d\vet.
\end{equation*}
The last integral is nonpositive, while $\overline{\Pi} \leq \|f\|_{L^1}\mathbf{Id}$ in the sense of matrices so that
$$\mathscr{D}_{0,-2}(f) \leq \|f\|_{L^1}\int_{\R^{3}}\frac{|\nabla f|^{2}}{f}\d v=\varrho_{f}\,I(f)$$
where $I(f)$ stands for the Fisher information.  According to \cite[Theorem 1]{DeVi2},
$$\mathscr{D}_{0,0}(f) \geq \lambda(f)\,\big( I(f)-I(M_{f}) \big)$$
for some $\lambda(f) >0$ which depends explicitly on the mass, momentum, and energy of $f$ as well as on $\int_{\R^{3}}f(v)v_{i}\,v_{j}\d v$, with $i,\,j=1,2,3$.  This leads to
\begin{align*}
\mathscr{D}_{0,\g}(f) &\geq \delta^{\g}\,\lambda(f)\big(I(f)-I(M_{f})\big)-\varrho_{f}\delta^{\g+2}\,I(f)\\
&\geq \lambda(f)\big(\delta^{\g}-\frac{\delta^{\g+2}\varrho_{f}}{\lambda(f)}\big)\big(I(f)-I(M_{f})\big) - C_{f}\delta^{\g+2}
\end{align*}
where $C_{f}=\varrho_{f}I(M_{f})$ depends only on $\varrho_{f},\bm{u}_{f}$ and $E_{f}$. Picking then $\delta >0$ so that
$$\delta^{2}=\lambda(f)\min\bigg(\frac{1}{2\varrho_{f}}\,,\,\frac{1}{4C_{f}}\big(I(f)-I(M_{f})\big)\bigg)$$
we get that
\begin{equation}\label{eq:d0g}
\mathscr{D}_{0,\g}(f) \geq \frac{\lambda(f)}{4}\left(I(f)-I(M_{f})\right)\delta^{\g}.\end{equation}
Hence, 
$$\mathscr{D}_{0,\g}(f) \geq \frac{\lambda(f)}{4}\big(I(f)-I(M_{f})\big)\min\bigg(\Big(\frac{\lambda(f)}{2\varrho_{f}}\Big)^{\frac{\g}{2}}\,;\,\Big(\frac{\lambda(f)}{4C_{f}}\Big)^{\frac{\g}{2}}\,\big(I(f)-I(M_{f})\big)^{\frac{\g}{2}}\bigg).$$
Finally, using the Logarithmic Sobolev inequality
$$I(f)-I(M_{f}) \geq \frac{2}{E_{f}}\mathcal{H}_{0}(f|M_{f})$$ 
we get the result.
\end{proof}
\begin{theo}\label{theo:prod}
Consider $0\leq f_0\in L^{1}_{s_0}(\R^{3})$, with $s_0>2$, satisfying \eqref{hypci} and a solution $f(t,v)$ to \eqref{LFDeq} with $\dd \in (0,\dd_{0}]$ given by theorem \ref{results}.  Then, there exist $\dd_{\star} \in (0,\dd_{0}]$, $\lambda_{0} >0$, and $C_{0} >0$ such that
\begin{equation}\label{eq:dissiHa}
\dfrac{\d}{\d t}\mathcal{H}_{\dd}(f(t)|\M_{\dd})\leq -\lambda_{0}\min\left( \mathcal{H}_{\dd}(f(t)|\M_{\dd})\,;\,\mathcal{H}_{\dd}(f(t)|\M_{\dd})^{1+\frac{\g}{2}}\right) + C_{0}\dd^{1+\frac{\g}{2}}, \qquad \forall\, t \geq 1.
\end{equation}
for any $\dd \in (0,\dd_{\star})$.  As a consequence, there is a positive constant $C_{1} >0$ such that
\begin{equation}\label{eq:estimaHa}
\mathcal{H}_{\dd}(f(t)|\M_{\dd}) \leq C_{1}\left((1+t)^{-\frac{2}{\g}}+\dd^{1+\frac{2}{\g}}\right) \qquad \forall\, t \geq 1, \;\; \dd \in (0,\dd_{\star}).
\end{equation}
\end{theo} 
\begin{proof} The proof of \eqref{eq:dissiHa} is a direct application of propositions \ref{prop:DoDaka}, \ref{prop:compHaH0Ma} and \ref{theo:Vill}.   Indeed, with the notations of such propositions, for any $\kappa_{0} \in (0,1)$ there is $\dd_{\star} >0$ such that, as soon as $\dd \in (0,\dd_{\star})$, we have for $t \geq 1$
\begin{align*}
\dfrac{\d}{\d t}\mathcal{H}_{\dd}(f(t)|\M_{\dd})=-\mathscr{D}_{\dd}(f(t)) &\leq -\frac{\kappa_{0}^{2}}{2}\,\mathscr{D}_{0}(f(t))  + \frac{C_{1}}{2}\dd^{2}\\
&\leq  -\frac{\kappa_{0}^{2}}{2}\min\left(\lambda_{1}\mathcal{H}_{0}(f(t)|\M_{0})\,;\,\lambda_{2}\mathcal{H}_{0}(f(t)|\M_{0})^{1+\frac{\g}{2}}\right) + \frac{C_{1}}{2}\dd^{2}
\end{align*}
where $\M_{0}$ is the Maxwellian with same mass, momentum, and energy as the initial datum $f_{0}$.  In addition, 
$$\Big|\mathcal{H}_{0}(f(t)|\M_{0})- \mathcal{H}_{\dd}(f(t)|\M_{\dd})\Big|
\leq \Big|\mathcal{H}_{0}(\M_{\dd}|\M_{0})\Big|+ \Big|\mathcal{H}_{0}(f(t)|\M_{\dd})- \mathcal{H}_{\dd}(f(t)|\M_{\dd})\Big|.$$
Then, \eqref{eq:dissiHa} follows from Proposition \ref{prop:compHaH0Ma} and Lemma \ref{lem:relative}.  Estimate \eqref{eq:estimaHa} follows after integration of \eqref{eq:dissiHa}.
\end{proof}
\begin{cor}\label{cor:converg}
Under the assumptions of Theorem \ref{smoothness}, for any $t_{0} \geq 1$, $\ell >0$, there exists $\dd^{\star} > 0$ such that
$$\left\|f(t)-\M_{\dd}\right\|_{L^{2}(\langle \cdot \rangle^{\ell})} \leq C\left(\left(1+t\right)^{-\frac{1}{5\gamma}}+\dd^{\frac{\g+2}{10\g}}\right), \qquad \forall\, \dd \in (0,\dd^{\star}),\;\; t \geq t_{0}$$
for some constant $C=C(t_{0},\ell)>0$.
\end{cor}
\begin{proof} Using a version of Csisz\'ar-Kullback inequality for the Fermi-Dirac entropy, derived in \cite[Theorem 3]{LW}, we have that
$$\|f(t)-\M_{\dd}\|_{L^{1}}^{2} \leq 2\varrho\,\mathcal{H}_{\dd}(f(t)|\M_{\dd}), \qquad \forall\, t \geq0\,.$$
Thus, according to Theorem \ref{theo:prod}, for $\dd \in (0,\dd_{\star})$
$$\|f(t)-\M_{\dd}\|_{L^{1}} \leq C_{2}\left((1+t)^{-\frac{1}{\g}}+\dd^{\frac{1}{2}+\frac{1}{\g}}\right)$$
for some constant $C_{2} >0$.  Now, recalling Nash's inequality
$$\|u\|_{L^2} \leq  C\|u\|_{L^1}^{\frac{2}{5}}\,\|\nabla u\|_{L^2}^{\frac{3}{5}}$$
for some universal positive constant $C>0$ and applying it to $u=g\,\langle \cdot \rangle^{\frac{\ell}{2}}$ with $\ell >0$, one has  
\begin{equation*}\label{eq:interpol}
\|g\|_{L^{2}(\langle \cdot \rangle^{\ell})} \leq C\|g\|_{L^{1}_{\frac{\ell}{2}}}^{\frac{2}{5}}\,\|g\|_{H^{1}_{\ell}}^{\frac{3}{5}} \leq \bar{C}_{\ell}\|g\|_{L^{1}_{\ell}}^{\frac{1}{5}}\,\|g\|_{H^{1}_{\ell}}^{\frac{3}{5}}\,\|g\|_{L^{1}}^{\frac{1}{5}}\end{equation*}
for some positive constant $\bar{C}_{\ell}$ since $\|\nabla u\|_{L^2} \leq (1+\ell)\|g\|_{H^{1}_{\ell}}$ and $\|g\|_{L^{1}_{\frac{\ell}{2}}} \leq \|g\|_{L^{1}_{\ell}}^{\frac{1}{2}}\,\|g\|_{L^{1}}^{\frac{1}{2}}$. We deduce from Theorem \ref{smoothness} that 
$$\|f(t)-\M_{\dd}\|_{L^{2}(\langle \cdot \rangle^{\ell})} \leq C(t_{0},\ell)\left((1+t)^{-\frac{1}{\g}}+\dd^{\frac{1}{2}+\frac{1}{\g}}\right)^{\frac{1}{5}}\,,\qquad \forall\, t\geq t_0$$
which gives the conclusion.
\end{proof}
\noindent
We are in position to prove Theorem \ref{theo:converg}.
\begin{proof}[Proof of Theorem \ref{theo:converg}] Let $\dd \in (0,\dd_{0}]$ be such that $\dd\,a_{\dd} <1$, $k > k_{\dd}^{\dagger}$, and $\delta >0$ be the small parameter appearing in Proposition \ref{prop:close}. In Corollary \ref{cor:converg}, we can pick $t_{0} >0$ sufficiently large and construct $\dd^{\ddagger}$ sufficiently small such that 
$$\left\|f(t)-\M_{\dd}\right\|_{L^{2}(\langle \cdot \rangle^{\ell})} \leq \delta \qquad \forall\, t \geq t_{0}.$$
Applying Proposition \ref{prop:close}, using Theorem \ref{smoothness} to ensure \eqref{eq:boundclose}, we get that, for any $\lambda_{\star}\in(0,\lambda_\gamma(\dd))$,
$$\|f(t)-\M_{\dd}\|_{L^{2}(\langle \cdot \rangle^{k})} \leq C\exp(-\lambda_{\star}t)\|f(t_{0})-\M_{\dd}\|_{L^{2}(\langle \cdot \rangle^{k})}, \qquad \forall\, t \geq t_{0}.$$
Apply in such estimate the uniform bound on $\|f(t)-\M_{\dd}\|_{L^{2}(\langle \cdot\rangle^{k})}$ on $(0,t_{0}]$ given in Theorem \ref{smoothness} and the control of $\|f(t)-\M_{\dd}\|_{L^{1}_{2}}$ by $\|f(t)-\M_{\dd}\|_{L^{2}(\langle \cdot \rangle^{k})}$ to conclude.\end{proof}

\appendix

\section{About Fermi-Dirac statistics}\label{app:FD}

Assume that the initial condition $f_0$ satisfies \eqref{hypci}-\eqref{hyp:mass}.  For any $\dd\leq \dd_{0}$, let $\M_{\dd}$ be the Fermi-Dirac statistics 
$$\M_{\dd}(v)=\frac{a_{\dd}\exp(-b_{\dd}|v|^{2})}{1+\dd\,a_{\dd}\exp(-b_{\dd}|v|^{2})}$$
with $a_{\dd} \geq 0$ $b_{\dd} \geq 0$ such that \eqref{FD_moments} holds. We collect here several results concerning the behaviour of $\M_{\dd}$ as $\dd \to 0$. 
\begin{lem}\label{lem:abeq2} For any $\dd\leq\bar{\dd}=\left(\frac{2}{5}\right)^{\frac{5}{2}}\frac{(6\pi\,E)^{\frac{3}{2}}}{\varrho}$ it follows that, 
\begin{equation}\label{abeq2}
\frac{3b_{\dd}}{5} \leq \frac{1}{2E} \leq  \frac{4}{3}b_{\dd}\,,\qquad\qquad  \left(\frac{3}{5}\right)^{\frac{5}{2}}a_{\dd} \leq \frac{\varrho}{(2\pi\,E)^{\frac{3}{2}}}\leq \left(\frac{4}{3}\right)^{\frac{3}{2}}\,a_{\dd}\,.
\end{equation}
\end{lem}
\begin{proof} {We first recall that, according to \cite[Eq. (5.2)]{Lu}, there exists some (explicit) strictly increasing mapping $\mathbf{\Phi}\::\R^{+} \to \R^{+}$ such that
\begin{equation}\label{eq:LU}
\mathbf{\Phi}\left(\frac{1}{\dd\,a_{\dd}}\right)= 3\,E \left(\frac{4\pi}{\dd\varrho}\right)^{\frac{2}{3}}\end{equation}
with moreover
$$\lim_{t\to0^{+}}\mathbf{\Phi}(t)=\frac{3}{5}\,3^{2/3}, \qquad \lim_{t\to\infty}\mathbf{\Phi}(t)=\infty.$$
In particular, since \eqref{eq:LU} implies that $\lim_{\dd\to0^{+}} \mathbf{\Phi}\left(\frac{1}{\dd\,a_{\dd}}\right)=+\infty$, we first observe that
\begin{equation}\label{eq:ddadd}
\lim_{\dd \to 0}\dd\,a_{\dd}=0.\end{equation}}
For notational simplicity set
$$\varrho_{\dd}=a_{\dd}\left(\frac{\pi}{b_{\dd}}\right)^{\frac{3}{2}},\qquad  \qquad E_{\dd}=\frac{1}{2b_{\dd}}\,,$$
so that
$$M(v)=a_{\dd}\exp\left(-b_{\dd}|v|^{2}\right)=\frac{\varrho_{\dd}}{(2\pi\,E_{\dd})^{3/2}}\exp\Big(-\frac{|v|^{2}}{2E_{\dd}}\Big)$$
is a Maxwellian with mass $\varrho_{\dd}$ and energy $3\varrho_{\dd}\,E_{\dd}$. Recalling that $\M_{\dd}(v)=\frac{M(v)}{1+\dd M(v)} \leq M(v)$, we get
\begin{equation*}
\begin{split}
\varrho &= \int_{\R^{3}}\M_{\dd}(v)\d v \leq \int_{\R^{3}}M(v)\d v=\varrho_{\dd}\\
3\varrho E &= \int_{\mathbb{R}^{3}} \M_{\dd}(v) |v|^{2}\d v \leq \int_{\R^{3}}M(v)|v|^{2}\d v=3\varrho_{\dd}\,E_{\dd}=\frac{3}{2}\pi^{\frac{3}{2}}a_{\dd}b_{\dd}^{-\frac{5}{2}}\,.
\end{split}\end{equation*}
Set $\phi(x) = \frac{1}{1+\dd x}$ with $x>0$.  Note that $\phi$ is convex.  Then, by Jensen's inequality,
\begin{equation*}
\varrho =\int_{\R^{3}}\phi(M(v))M(v)\d v \geq \varrho_{\dd}\phi\left(\varrho_{\dd}^{-1}\int_{\R^{3}}M(v)^{2}\d v\right)\,.
\end{equation*}
Since $M(v)^{2}$ is the Maxwellian associated with coefficients $a_{\dd}^{2}$ and $2b_{\dd}$ we get
$$\int_{\R^{3}}M(v)^{2}\d v=a_{\dd}^{2}\left(\frac{\pi}{2b_{\dd}}\right)^{\frac{3}{2}}=a_{\dd}\dfrac{\varrho_{\dd}}{2^{\frac{3}{2}}}$$
which results in 
$$\varrho \geq \varrho_{\dd}\phi\left(a_{\dd}2^{-\frac{3}{2}}\right)=\frac{\varrho_{\dd}}{1+2^{-\frac{3}{2}}\dd\,a_{\dd}}.$$
%\begin{equation*} \frac{a}{b^{3/2}}\int_{\mathbb{R}^{3}}\phi(e^{-|v|^{2}})e^{-|v|^{2}}\text{d}v\geq \frac{c_0\, a}{b^{3/2}}\phi\Big( c^{-1}_{0}\int_{\mathbb{R}^{3}}e^{-2|v|^{2}}\text{d}v\Big) = \frac{c_0\, a}{b^{3/2}}\frac{1}{1+2^{-3/2}\varepsilon a}\,.
%\end{equation*}
Similarly,
\begin{equation*}
3\varrho E =\int_{\R^{3}}\phi(M(v))M(v)|v|^{2}\d v \geq 3\varrho_{\dd}E_{\dd}\phi\left(\frac{1}{3\varrho_{\dd}E_{\dd}}\int_{\R^{3}}M^{2}(v)|v|^{2}\d v\right)\end{equation*}
with 
$$\int_{\R^{3}}M^{2}(v)|v|^{2}\d v=3a_{\dd}^{2}\left(\frac{\pi}{2b_{\dd}}\right)^{\frac{3}{2}}\frac{1}{4b_{\dd}}=a_{\dd}\dfrac{3\varrho_{\dd}E_{\dd}}{2^{\frac{5}{2}}}.$$
Thus,
\begin{equation}\label{roE}
3\varrho\,E \geq 3\varrho_{\dd}E_{\dd}\phi\left(a_{\dd}2^{-\frac{5}{2}}\right)=\frac{3\varrho_{\dd}E_{\dd}}{1+2^{-\frac{5}{2}}\dd\,a_{\dd}}\,.
\end{equation}
This gives the following set of inequalities, in terms of $a_{\dd},b_{\dd}$
\begin{equation}\label{eq:ab1}
\pi^{\frac{3}{2}}\frac{a_{\dd}b_{\dd}^{-\frac{3}{2}}}{1+2^{-\frac{3}{2}}\dd\,a_{\dd}} \leq \varrho \leq \pi^{\frac{3}{2}}a_{\dd}b_{\dd}^{-\frac{3}{2}}, \qquad \qquad  	 \pi^{\frac{3}{2}}\frac{a_{\dd}b_{\dd}^{-\frac{5}{2}}}{1+2^{-\frac{5}{2}}\dd\,a_{\dd}} \leq  2\varrho\,E \leq  \pi^{\frac{3}{2}}a_{\dd}b_{\dd}^{-\frac{5}{2}}.\end{equation}
Little algebra yields 
\begin{equation}\begin{split}\label{eq:estimbapsi}
\frac{b_{\dd}}{1+2^{-\frac{3}{2}}\dd\,a_{\dd}} &\leq \frac{1}{2E} \leq b_{\dd}(1+2^{-\frac{5}{2}}\dd\,a_{\dd}),\\
\frac{a_{\dd}}{(1+2^{-\frac{3}{2}}\dd\,a_{\dd})^{\frac{5}{2}}}&\leq \frac{\varrho}{(2\pi\,E)^{\frac{3}{2}}} \leq a_{\dd}(1+2^{-\frac{5}{2}}\dd\,a_{\dd})^{\frac{3}{2}}.\end{split}\end{equation}
The left hand side of the second inequality reads
$$\psi(\dd\,a_{\dd}) \leq \frac{\dd \varrho}{(2\pi\,E)^{\frac{3}{2}}}, \qquad \text{ where } \quad \psi(x)=\frac{x}{(1+2^{-\frac{3}{2}}x)^{\frac{5}{2}}}, \qquad \forall\,  x >0.$$
Notice that $\psi$ has a unique maximum point at $\bar{x}=\frac{2^{\frac{5}{2}}}{3}$ with value $\psi(\bar{x})=\left(\frac{2}{5}\right)^{\frac{5}{2}}3^{\frac{3}{2}}$. Let $\bar{\dd} >0$ be choosen such that $\frac{\bar{\dd} \varrho}{(2\pi\,E)^{\frac{3}{2}}}=\psi(\bar{x})$. Using the fact that the mapping $\dd >0 \mapsto \dd\,a_{\dd}$ is continuous and goes to zero as $\dd \to 0$ according to \eqref{eq:ddadd}, we deduce from a continuity argument that
$$\dd < \bar{\dd} \Longrightarrow \dd a_{\dd} < \bar{x}=\frac{2^{\frac{5}{2}}}{3}.$$
Using \eqref{eq:estimbapsi}, this easily leads to \eqref{abeq2}.
\end{proof}
\begin{rmq}\label{rmq:bounded} Notice that, combining \eqref{eq:ddadd} with \eqref{eq:estimbapsi}, one sees that 
$\limsup_{\dd \to 0}b_{\dd} \leq \frac{1}{2E}$ and $\limsup_{\dd \to 0}a_{\dd} \leq \frac{\varrho}{(2\pi\,E)^{\frac{3}{2}}}.$ In particular, both $(a_{\dd})_{\dd}$ and $(b_{\dd})_{\dd}$ are bounded.\end{rmq}
\begin{rmq}\label{rmq:abeq2} Lemma \ref{lem:abeq2} allows to explicit the value $\dd^{\dagger} >0$ such that the spectral gap $\lambda_{2}(\dd) >0$ for $\dd \in (0,\dd^{\dagger})$ in Proposition \ref{maxwell_gap}. Indeed, recall that $\lambda_{2}(\dd) >0$ as soon as 
$$\varrho\,E > \frac{3\dd^{2}}{C_{\mathrm{P}}(\dd)}a_{\dd}^{3}\left(\frac{\pi}{3b_{\dd}}\right)^{\frac{3}{2}}=\frac{3\dd^{2}}{2b_{\dd}^{\frac{5}{2}}}\left(\frac{\pi}{3}\right)^{\frac{3}{2}}a_{\dd}^{3}(1+\dd\,a_{\dd})^{4}.$$
Notice that, from \eqref{eq:ab1}, $\pi^{\frac{3}{2}}a_{\dd}b_{\dd}^{-\frac{5}{2}} \leq 2\varrho\,E(1+2^{-\frac{5}{2}}\dd\,a_{\dd}) \leq 2\varrho\,E\left(1+2^{-\frac{5}{2}}\bar{x}\right)$ for $\dd < \bar{\dd}.$ This means that, to get the above estimate, for $\dd \in (0,\bar{\dd})$, it suffices that
$$1> \frac{\dd^{2}a_{\dd}^{2}}{\sqrt{3}}\left(1+2^{-\frac{5}{2}}\bar{x}\right)(1+\bar{x})^{4}={\dd^{2}a_{\dd}^{2}}\frac{4(3+4\sqrt{2})^{4}}{3^{\frac{11}{2}}}$$
or equivalently, 
$$\frac{3^{\frac{11}{2}}}{4(3+4\sqrt{2})^{4}}a_{\dd}^{-2} > \dd^{2}.$$
According to \eqref{abeq2}, this holds as soon as
$$\dd < \frac{5^{-\frac{5}{2}}\,3^{\frac{21}{4}}}{(3+4\sqrt{2})^{2}}\frac{(2\pi\,E)^{\frac{3}{2}}}{2\varrho}:=\dd^{\dagger}.$$
\end{rmq}
\begin{rmq}\label{rmq:spec}
The same consideration also allows to estimate $\lambda_{2}(\dd)$ in Proposition \ref{maxwell_gap} yielding for instance
$$\lambda_{2}(\dd) > \frac{\nu}{E}\left(\frac{3^{4}}{4(3+4\sqrt{2})^{4}}-\dd^{2}\;\frac{\varrho^2\,5^{\frac{15}{2}}}{(\pi\,E)^{3}\,3^{\frac{21}{2}}}\right), \qquad \forall\, \dd \in (0,\dd^{\dagger}).$$
Recalling that 
$$\nu=\frac{\varrho^{2}E^{2}}{2C_{a,b}(1+\dd\,a_{\dd})}=\frac{\varrho^{2}E^{2}}{2\varrho_{\dd}E_{\dd}(1+\dd\,a_{\dd})}$$
and since, for $\dd < \bar{\dd}$, $\varrho\,E \geq \varrho_{\dd}E_{\dd} \left(1+2^{-\frac{5}{2}}\bar{x}\right)^{-1}=\frac{3}{4}\varrho_{\dd}E_{\dd}$ by \eqref{roE}, we get that
$\nu \geq \frac{3}{8(1+\bar{x})}\varrho\,E$
resulting in
$$\lambda_{2}(\dd) > \frac{9\varrho}{8(3+4\sqrt{2})}\left(\frac{3^{4}}{4(3+4\sqrt{2})^{4}}-\dd^{2}\frac{5^{\frac{15}{2}}\,\varrho^2}{3^{\frac{21}{2}}(\pi\,E)^{3}}\right), \qquad \forall\, \dd \in (0,\dd^{\dagger}).$$
In particular, $\lim_{\dd \to 0}\lambda_{2}(\dd) \geq \frac{3^{6}}{32(3+4\sqrt{2})^{5}}\varrho \approx 4.686 \times 10^{-4} \varrho$.
\end{rmq}
\noindent
With the notations of the previous proof, we also have
$$\left|\varrho_{\dd}-\varrho\right|=\left|\int_{\R^{3}}M(v)\d v-\int_{\R^{3}}\M_{\dd}(v)\d v\right|=\dd\,\int_{\R^{3}}M\,\M_{\dd}\d v$$
since $\M_{\dd}=\frac{M}{1+\dd\,M}$.  Consequently, using that for $\dd\in(0,\bar{\dd})$, $M(v) \leq a_{\dd} \leq \left(\frac{5}{3}\right)^{\frac{5}{2}}\frac{\varrho}{(2\pi E)^{\frac{3}{2}}}$, we obtain that
\begin{equation}\label{eq:varrhoa}
\left|\varrho_{\dd}-\varrho\right| \leq \dd\; \left(\frac{5}{3}\right)^{\frac{5}{2}}\frac{\varrho^2}{(2\pi E)^{\frac{3}{2}}}\,,\quad\text{that is }\; \left|\varrho_{\dd}-\varrho\right| \leq C_{0}\,\dd
\end{equation}
for some positive constant $C_{0}$ depending only on $E$ and $\varrho$.  In the same way,
$$\left|\varrho_{\dd}E_{\dd}-\varrho\,E\right|=\frac{1}{3}\,\left|\int_{\R^{3}}M(v)|v|^{2}\d v-\int_{\R^{3}}\M_{\dd}(v)|v|^{2}\d v\right|=\frac{\dd}{3}\int_{\R^{3}}M\,\M_{\dd}\,|v|^{2}\d v\,.$$
Thus, for $\dd\in(0,\bar{\dd})$,
$$\left|\varrho_{\dd}E_{\dd}-\varrho\,E\right| \leq \dd\; \left(\frac{5}{3}\right)^{\frac{5}{2}}\frac{\varrho^2}{(2\pi)^{\frac{3}{2}}\sqrt{E}}\,.$$
Since
$$\left|E_{\dd}-E\right|=\frac{1}{\varrho}|\varrho\,E_{\dd}-\varrho\,E| \leq \frac{E_{\dd}}{\varrho}\,|\varrho-\varrho_{\dd}| + \frac{1}{\varrho}\left|\varrho_{\dd}E_{\dd}-\varrho\,E\right|\,,$$
we deduce from \eqref{abeq2}, recalling that $E_{\dd}=\frac{1}{2b_{\dd}}$, that
$$\left|E_{\dd}-E\right| \leq \frac{1}{\varrho}\left(\frac{4E}{3}|\varrho-\varrho_{\dd}| + \dd \left(\frac{5}{3}\right)^{\frac{5}{2}}\frac{\varrho^2}{(2\pi)^{\frac{3}{2}}\sqrt{E}} \right).$$
This estimate and \eqref{eq:varrhoa} results in
\begin{equation}\label{eq:Edd}
\left|E_{\dd}-E\right| \leq C_{2}\dd, \qquad \forall \dd\in(0,\bar{\dd}),
\end{equation}
for some positive constant $C_{2}$ depending only on $\varrho$ and $E$. We deduce from this the following lemma.
\begin{lem}\label{lem:relative}
Denote by $\M_{0}(v)$ the Maxwellian 
$$\M_{0}(v)=\frac{\varrho}{(2\pi\,E)^{3/2}}\exp\left(-\frac{|v|^{2}}{2E}\right), \qquad \forall v \in \R^{3}\,.$$
There exists a positive constant $C$ depending only on $\varrho$ and $E$ such that
$$\left|\mathcal{H}_{0}(\M_{0}|\M_{\dd})\right|=\left|H(\M_{0})-H(\M_{\dd})\right| \leq C\,\dd \qquad \forall\, \dd \in (0,\bar{\dd}]\,.$$
\end{lem}
\begin{proof} Writing $\M_{\dd}(v)=\frac{M}{1+\dd\,M}$ with
$$M(v)=\frac{\varrho_{\dd}}{(2\pi\,E_{\dd})^{3/2}}\exp\left(-\frac{|v|^{2}}{2E_{\dd}}\right)\,,$$
we obtain that
$$
\mathcal{H}_{0}(\M_{0}|\M_{\dd})=\varrho\,\left(\log\varrho-\log \varrho_{\dd}-\tfrac{3}{2}\Psi\left(\frac{E}{E_{\dd}}\right)\right)+\int_{\R^{3}}\M_{\dd}\log(1+\dd\,M)\d v$$
with
$$\Psi(x)=\log x+1-x \leq 0, \qquad \forall x >0.$$
Now, we have
$$0 \leq \int_{\R^{3}}\M_{\dd}\log(1+\dd\,M)\d v \leq \dd \int_{\R^{3}}\M_{\dd}\,M\d v \leq \dd\,a_{\dd}\varrho$$
where $a_{\dd}$ is bounded, see Remark \ref{rmq:bounded}. This shows that there exists $C >0$ such that
\begin{equation}\label{eq:H0M0}
\left|\mathcal{H}_{0}(\M_{0}|\M_{\dd})\right| \leq C\dd + \varrho\,\left(\left|\log\varrho-\log \varrho_{\dd}\right|+\tfrac{3}{2}\left|\Psi\left(\frac{E}{E_{\dd}}\right)\right|\right).
\end{equation}
Additionally,
$$\left|\log \varrho-\log \varrho_{\dd}\right| \leq \max\left\{\frac{1}{\varrho},\frac{1}{\varrho_{\dd}}\right\}|\varrho-\varrho_{\dd}|$$
and using the elementary inequality $0\leq x-1-\log x \leq \frac{ (x-1)^2}{x}\;$ for $x >0$, we conclude that
$$\left|\Psi\left(\frac{E}{E_{\dd}}\right)\right| \leq \frac{1}{E\,E_{\dd}}\left(E-E_{\dd}\right)^{2}=\frac{2b_{\dd}}{E}\left(E-E_{\dd}\right)^{2}\,.$$
Using the fact that $\varrho_{\dd}=\dfrac{a_{\dd}\pi^{3/2}}{b_{\dd}^{3/2}}$ and $b_{\dd}$ are bounded according to \eqref{eq:ab1} and Remark \ref{rmq:bounded}, we deduce the result from  \eqref{eq:estimbapsi}--\eqref{eq:H0M0}.
\end{proof}
\begin{rmq} We deduce from Lemma \ref{lem:relative} and the Csiszar-Kullback inequality that
$$\left\|\M_{\dd}-\M_{0}\right\|_{L^{1}}^{2} \leq 2\left|\mathcal{H}_{0}(\M_{0}|\M_{\dd})\right| \leq 2\,C\,\dd, \qquad \forall \dd\in(0,\overline{\dd}).$$
\end{rmq}
\begin{lem}\label{lem:HessM} For any $\ell > 0$ there is $C_{\ell} >0$ such that 
$$\sup_{0 < \dd < 1}\left(\|\nabla^{2}\M_{\dd}\|_{L^{2}(\langle \cdot \rangle^{\ell})} + \|\M_{\dd}\|_{L^{2}(\langle \cdot \rangle^{\ell})}\right) \leq C_{\ell}.$$
More generally, for any $k \in \mathbb{N}$, $s \geq0$, 
$$\sup_{0 < \dd < 1}\|\M_{\dd}\|_{H^{k}_{s}} < \infty.$$
\end{lem}
\begin{proof} For the computation of the $L^{2}(\langle \cdot\rangle^{\ell})$ norm, one simply notice that
$$\|\M_{\dd}\|_{L^{2}(\langle \cdot \rangle^{\ell})}^{2} \leq \int_{\R^{3}} M^{2}(v) \langle v \rangle^{\ell}\d v=\frac{\varrho_{\dd}^{2}}{(2\pi\,E_{\dd})^{3}}\int_{\R^{3}}\exp\left(-\frac{|v|^{2}}{E_{\dd}}\right)\langle v\rangle^{\ell}\d v$$
which depends only on $\ell,\varrho_{\dd}$ and $E_{\dd}.$ In particular, it is uniformly bounded with respect to $\dd \in (0,\dd_0].$ In the same way, since
$$\nabla^{2}\M_{\dd}(v)=\frac{1}{E_{\dd}^{2}}\frac{M(v)}{(1+\dd\,M(v))^{3}}\bigg(\big(1-\dd\,M(v)\big) v \otimes v - E_{\dd}\big(1+\dd\,M(v)\big)\mathbf{Id}\bigg)$$
the same reasoning shows that $\|\nabla^{2}\M_{\dd}\|_{L^{2}(\langle\cdot\rangle^{\ell})}$ can be bounded with bound depending only on $\ell,\,E_{\dd}$, and $\varrho_{\dd}$. The proof for general Sobolev weighted estimates follows by induction.
\end{proof}

\section{Factorization and enlargement}\label{app:kleb}
Recall the notations introduced in the proof of Theorem \ref{theo:enlarge}, namely, we set
\begin{eqnarray*}
\bm{A}_{\dd}g& =& \Q(g,\M_{\dd}) + \U\chi_{R}g-2\dd\,\Q(g\,\M_{\dd},\M_{\dd})\\
& =& \sum_{i,j=1}^{3}\left({a}_{ij}\ast g\right)\partial^2_{ij}\M_{\dd} - (c\ast g)\M_{\dd} + \U\,\chi_{R}g-2\dd\,\Q(g\,\M_{\dd},\M_{\dd})
\end{eqnarray*}
and
$$\bm{B}_{\dd}g=\mathscr{L}_{\dd}g-\bm{A}_{\dd}g=\sum_{i,j=1}^{3}\left({a}_{ij}\ast \M_{\dd}\right)\partial^2_{ij}g-(c\ast\M_{\dd})\,g -\dd\,\Q(\M_{\dd}^{2},g)-\U\chi_{R}g$$
where $\Q$ denotes the bilinear Landau collision operator, $\chi_{R}=\chi(R^{-1}\cdot)$ with $R >1$ and $\chi$ a suitable smooth cut-off function, and $\U >0$.  

\subsection{Dissipativity properties} We begin with the study of the disipativity properties of $\bm{B}_{\dd}$.  The proof of the following lemma is a direct consequence of \cite[Lemma 2.5]{Kleber}, recall that the mass of $\M_{\dd}$ is $\varrho$.
\begin{lem}\label{lem:Jp} For any $v \in \R^{3}$, set 
$$J_{p}(v)=\frac{1}{\varrho}\int_{\R^{3}}|v-w|^{p}\M_{\dd}(w)\d w\,, \quad \text{ and }  \quad \mu_{p}(\dd)=\frac{1}{\varrho}\int_{\R^{3}}|v|^{p}\M_{\dd}(v)\d v, \quad 0 \leq p \leq 3.$$ Then, for any $v \in \R^{3}$ it holds:
\begin{enumerate}
\item $J_{0}(v)=1$ and $J_{2}(v)=|v|^{2}+\mu_{2}(\dd)$,
\item $J_{p}(v) \leq |v|^{p}+\mu_{p}(\dd)$ for any $0 < p \leq 1$,
\item $J_{p}(v) \leq |v|^{p}+\mu_{2}(\dd)^{\frac{p}{2}}$ for any $1 \leq p < 2$,
\item {$J_{p}(v) \leq |v|^{p} +(6\mu_{2}(\dd))^{\frac{p}{4}}|v|^{\frac{p}{2}}+\mu_{4}(\dd)^{\frac{p}{4}}$ for $2 < p \leq 3.$}
\end{enumerate}
\end{lem}
\noindent
Also, the following lemma holds.  It is proven in \cite[Proposition 4.10]{lemou}, for $\dd=1$, and \cite[Lemma 2.7]{Kleber}. Recall that $\mM=\M_{\dd}(1-\dd \M_{\dd})$.
\begin{lem}\label{lem:aij} For any $f >0$ radially symmetric with $f \in L^{1}_{\g+2}(\R^{3})$, the matrix  
$$\bm{\sigma}[f](v)=\big(\bm{\sigma}_{ij}[f](v)\big)_{ij}=\big(a_{ij}\ast f(v)\big)_{ij}$$
has a simple eigenvalue $\lambda_{1}[f](v) >0$ associated to the eigenvector $v$ and a double eigenvalue $\lambda_{2}[f](v) >0$ associated to the eigenspace $v^{\perp}$.  They are given by
$$\lambda_{1}[f](v)=\int_{\R^{3}}\left(1-\left\langle \frac{v}{|v|}\,,\,\frac{w}{|w|}\right\rangle^{2}\right)|w|^{\g+2}f(v-w)\d w,$$
$$\lambda_{2}[f](v)=\int_{\R^{3}}\left(1-\frac{1}{2}\left|\frac{v}{|v|}\times \frac{w}{|w|}\right|^{2}\right)|w|^{\g+2}f(v-w)\d w,$$
and satisfy, for $|v| \to \infty$
$$\lambda_{1}[f](v) \simeq\frac{2}{3}|v|^{\g}\int_{\R^{3}}|w|^{2}f(w)\d w, \qquad \lambda_{2}[f](v) \simeq |v|^{\g+2}\int_{\R^{3}} f(w)\d w$$
with 
$$\min\big(\lambda_{1}[f](v),\lambda_{2}[f](v)\big) \geq \lambda_{\dd}[f], \qquad \forall\, v \in \R^{3}$$
for some $\lambda_{\dd} >0.$ Moreover, the function $\bm{\sigma}_{ij}[f]$ is smooth and, for any multi-index $\bm{\beta} \in \N^{3}$, 
$$\left|\partial^{\bm{\beta}}\,\bm{\sigma}_{ij}[f](v)\right| \lesssim \langle v\rangle^{\g+2-|\bm{\beta}|}\,,$$
and
$$\sum_{i,j=1}^{3}\bm{\sigma}_{ij}[f](v)\xi_{i}\xi_{j}=\lambda_{1}[f](v)\left|P_{v}(\xi)\right|^{2}+\lambda_{2}[f](v)\big|\xi-P_{v}(\xi)\big|^{2} \qquad \forall\, \xi \in \R^{3}\,.$$
$P_{v}$ is the projection on $\mathrm{Span}(v)$. Finally,
$$\mathrm{trace}\big(\bm{\sigma}[f](v)\big)=2\int_{\R^{3}}|v-\vet|^{\g+2}f(\vet)\d \vet,$$ 
and,
\begin{equation}\label{eq:bimM}
\bm{b}_{i}[\mM](v)=-2b_{\dd}\,v_{i}\lambda_{1}\big[\mM-2\dd\,\mM\,\M_{\dd}\big](v) \qquad \forall\, i=1,2,3.\end{equation}
\end{lem}
\begin{proof} The first part of the statement is a general property of the matrix $(a_{ij})_{ij}$, see also \cite{lemou,degond}. The computation of the trace of $\bm{\sigma}[f](v)$ is as in \cite[Lemma 2.7]{Kleber}. Let us compute $\bm{b}_{i}[\mM](v)$.  Note that
$$\bm{b}_{i}[\mM](v)=\sum_{j=1}^{3}\partial_{j}a_{ij}\ast \mM(v)=\sum_{j=1}^{3}a_{ij}\ast \partial_{j}\mM.$$
Since $\partial_{j}\mM=\partial_{j}\M_{\dd}(1-2\dd\,\M_{\dd})=-2b_{\dd}v_{j}\mM(v)+4\dd\,b_{\dd}v_{j}\mM\,\M_{\dd}$, we get that
$$\bm{b}_{i}[\mM](v)=-2b_{\dd}\sum_{j}\bm{\sigma}_{ij}[\mM](v)v_{j}+4\dd\,b_{\dd}\sum_{j=1}^{3}\bm{\sigma}_{ij}[\mM\,\M_{\dd}](v)v_{j}$$
which gives the result.
\end{proof}
\noindent
The key point in the sequel is the fact that, since $\mM=\M_{\dd}-\dd\M_{\dd}^{2}$ it follows that
$$\bm{B}_{\dd}g=\sum_{i,j=1}^{3}\bm{\sigma}_{ij}[\mM](v)\partial^2_{ij}g-\bm{c}[\mM](v)\,g(v) -\U\chi_{R}g$$
where $\bm{c}[\mM]=c \ast \mM$. The computations of \cite[Lemma 2.8, Eq. (2.19)]{Kleber}  give the following lemma.
\begin{lem}\label{lem:weightm}
For any $\theta \in \R$, $p >1$, and positive weight $\varpi(v) >0$, it follows that
\begin{multline*}
\sum_{i,j=1}^{3}\int_{\R^{3}}\bm{\sigma}_{ij}[\mM](v)\partial^2_{ij}g\,|g(v)|^{p-1}\mathrm{sign}(g(v)) \varpi(v) \d v-\int_{\R^{3}}\bm{c}[\mM](v)|g(v)|^{p}\varpi(v)\d v\\
=-(p-1)\sum_{i,j=1}^{3}\int_{\R^{3}}\bm{\sigma}_{ij}[\mM](v)\partial_{i}\left(\varpi^{\theta}g\right)\,
\partial_{j}\left(\varpi^{\theta}g\right)\varpi^{p(1-\theta)-1}|g(v)|^{p-2}\d v\\
+\int_{\R^{3}}\varphi_{\varpi,p,\theta}(v)|g(v)|^{p}\varpi(v)\d v\,,
\end{multline*}
with
\begin{equation*}
\begin{split}
\varphi_{\varpi,p,\theta}(v)&=\eta_{1}(p,\theta)\sum_{ij}\bm{\sigma}_{ij}[\mM](v)\frac{\partial^2_{ij}\varpi}{\varpi}+\eta_{2}(p,\theta)\sum_{ij}\bm{\sigma}_{ij}[\mM](v)\frac{\partial_{i}\varpi}{\varpi}\frac{\partial_{j}\varpi}{\varpi}\\
&\hspace{1cm} + \left(\frac{1}{p}+\eta_{1}(p,\theta)\right)\sum_{i=1}^{3}\bm{b}_{i}[\mM](v)\frac{\partial_{i}\varpi}{\varpi} -\frac{p-1}{p}\bm{c}[\mM](v)
\end{split}
\end{equation*}
where $\eta_{1}(p,\theta)=\frac{1-2\theta(p-1)}{p}$, and $\eta_{2}(p,\theta)=-\theta\, p\, \eta_{1}(p,\theta)+\theta[1-\theta(p-1)]$.
\end{lem}
\noindent
In particular, for $\varpi=\varpi_{k}(v)=\langle v\rangle^{k}$ since 
\begin{equation*}
\frac{\partial_{i}\varpi}{\varpi}=kv_{i}\langle v\rangle^{-2}, \qquad \frac{\partial^2_{ij}\varpi}{\varpi}={\delta}_{ij}k\langle v\rangle^{-2}+k(k-2)v_{i}v_{j}\langle v\rangle^{-4}\,,
\end{equation*}
we deduce from Lemma \ref{lem:aij} that
\begin{equation*}\begin{split}
\sum_{ij}\bm{\sigma}_{ij}[\mM](v)\frac{\partial^2_{ij}\varpi}{\varpi}&=k\,\mathrm{trace}\left(\bm{\sigma}[\mM](v)\right)\langle v\rangle^{-2}+k(k-2)\langle v\rangle^{-4}\sum_{ij}\bm{\sigma}_{ij}[\mM](v)v_{i}v_{j}\\
&=2k\,\langle v\rangle^{-2}\int_{\R^{3}}|v-\vet|^{\g+2}\mM(\vet)\d\vet+k(k-2)\lambda_{1}[\mM](v)|v|^{2}\langle v\rangle^{-4}
\end{split}\end{equation*}
and, in the same way
\begin{equation*}\begin{split}
\sum_{ij}\bm{\sigma}_{ij}[\mM](v)\frac{\partial_{i}\varpi}{\varpi}\frac{\partial_{j}\varpi}{\varpi}&=\lambda_{1}[\mM](v)k^{2}|v|^{2}\langle v\rangle^{-4}\\
\sum_{i=1}^{3}\bm{b}_{i}[\mM](v)\frac{\partial_{i}\varpi}{\varpi}&=k\langle v\rangle^{-2}\sum_{i=1}^{3}\bm{b}_{i}[\mM](v)v_{i}.
\end{split}\end{equation*}
Moreover, one notices that
$$\bm{c}[\mM](v)=-2(\g+3)\int_{\R^{3}}|v-\vet|^{\g}\mM(\vet)\d \vet.$$
Using \eqref{eq:bimM} and Lemma \ref{lem:weightm} with $\theta=\frac{1}{2}$ and $p=2$, one is led to the following lemma. 
\begin{lem}\label{lem:vk}
Setting $\varpi_{k}=\langle v\rangle^{k}$ we have 
\begin{multline}\label{eq:Bagg}
\int_{\R^{3}}\bm{B}_{\dd}g(v)\,g(v)\varpi_{k}(v)\d v
=-\sum_{i,j=1}^{3}\int_{\R^{3}}\bm{\sigma}_{ij}[\mM](v)\partial_{i}\left(\varpi_{\frac{k}{2}}g\right)\, \partial_{j}\left(\varpi_{\frac{k}{2}}g\right)\d v\\
+\int_{\R^{3}}\left(\Phi_{k}(v)-\U\chi_{R}(v)\right)g^{2}(v)\varpi_{k}(v)\d v\end{multline}
with
\begin{multline}
\label{eq:phik}
\Phi_{k}(v)=\frac{k^2}{4}\lambda_{1}[\mM](v)|v|^{2}\langle v\rangle^{-4}
-b_{\dd}\,k|v|^{2}\langle v\rangle^{-2}\big(\lambda_{1}[\mM](v)-2\dd\,\lambda_{1}[\mM\,\M_{\dd}](v)\big) \\
+ (\g+3)\int_{\R^{3}}|v-\vet|^{\g}\mM(\vet)\d \vet.
\end{multline}
\end{lem}
\begin{prop}\label{prop:Badiss} Fix $\dd > 0$ such that $\dd\,a_{\dd} < 1$.  For any $a \in \R$ and
$$k > k_{\dd}:={\varrho(\g+3)}\bm{\zeta}_{\dd}^{-1}$$
there exist $\U,\,R$ sufficiently large so that $\bm{B}_{\dd}$ generates a $C_{0}$-semigroup $\big(U_{\dd}(t)\big)_{t\geq 0}$ in $L^{2}(\varpi_{k})$
with
$$\|U_{\dd}(t)\|_{\mathscr{B}(L^{2}(\varpi_{k}))} \leq \exp(a t), \qquad t \geq 0\,.$$
\end{prop}
\begin{proof} Following \cite[Lemma 2.8]{Kleber}, the proof consists in identifying the dominant terms in $\Phi_{k}(v)$ for large $|v|$. According to Lemma \ref{lem:aij}, as $|v| \to \infty$ the first term in $\Phi_{k}(v)$ converges to zero while
\begin{multline*}
-b_{\dd}\,k|v|^{2}\langle v\rangle^{-2}\big(\lambda_{1}[\mM](v)-2\dd\,\lambda_{1}[\mM\,\M_{\dd}](v)\big)
\simeq -\frac{2b_{\dd}k}{3}|v|^{\g}\int_{\R^{3}}|w|^{2}\big(1-2\dd\M_{\dd}(w)\big)\mM(w)\d w\,,
\end{multline*}
and since $\mM \leq \M_{\dd}$
$$\int_{\R^{3}}|v-\vet|^{\g}\mM(\vet)\d \vet \leq \varrho\,J_{\g}(v)\,.$$
Here we used the notations of Lemma \ref{lem:Jp}. Therefore,
$$\limsup_{|v| \to \infty}\,|v|^{-\g}\Phi_{k}(v) \leq  -\bigg(\frac{2b_{\dd}k}{3}\int_{\R^{3}}|w|^{2}\mM(w)\left(1-2\dd\M_{\dd}(w)\right)\d w -\varrho(\g+3)\bigg)$$
and as soon as $k > k_{\dd}$, it follows that for any $a \in \R$ one can choose $\U,\, R$ sufficiently large such that
$$\Phi_{k}(v)-\U\chi_{R}(v) \leq a \qquad \forall\, v \in \R^{3}.$$
Since the first term on the right-hand side of \eqref{eq:Bagg} is nonpositive due to the ellipticity of $\bm{\sigma}[\mM](v)$, we conclude that
$$\int_{\R^{3}}\bm{B}_{\dd}g\,g\,\varpi_{k}(v)\d v \leq a\int_{\R^{3}}g^{2}(v)\varpi_{k}(v)\d v\,,$$
which proves the result.
\end{proof}
 
\subsection{Regularization} 
Let us now investigate the regularization properties of $\bm{A}_{\dd}$. Introducing, for any $g=g(v)$
$$G_{\dd}(v)=g(v)\big(1-2\dd\M_{\dd}(v)\big)$$
one has
$$\bm{A}_{\dd}g(v)=\sum_{i,j=1}^{3}\left({a}_{ij}\ast G_{\dd}\right)\partial^2_{ij}\M_{\dd} - (c\ast G_{\dd})\M_{\dd} + \U\,\chi_{R}g$$
where $\U,\,R$ have been chosen sufficiently large so that Proposition \ref{prop:Badiss} holds. Notice that, for any positive weight function $\varpi$ and $q \leq 2$, the multiplication operator 
$$\U\chi_{R}\::g \mapsto \U\chi_{R}g$$
is bounded from $L^{2}(\varpi)$ to $L^{q}(\M_{\dd}^{-1})$.  Thus, we focus on the operator
$$\widetilde{\bm{A}}_{\dd}g=\sum_{i,j=1}^{3}\left({a}_{ij}\ast G_{\dd}\right)\partial^2_{ij}\M_{\dd} - (c\ast G_{\dd})\M_{\dd}.$$
\begin{lem}\label{lem:Abounded}
For any $k \geq 0$, there is $C_{1}(k,\dd) >0$ such that
\begin{equation}\label{eq:tildeAa}
\|\widetilde{\bm{A}}_{\dd}g\|_{L^{2}(\M_{\dd}^{-1})} \leq C_{1}(k,\dd)\|G_{\dd}\|_{L^{1}_{\g+2}} \leq C_{1}(k,\dd)\|1-2\dd\M_{\dd}\|_{L^\infty}\|g\|_{L^{1}_{\g+2}}.\end{equation}
Consequently, for any $k > 2\g+7$, it holds
$$\widetilde{\bm{A}}_{\dd} \in \mathscr{B}(L^{1}_{\g+2}(\R^{3}),L^{2}(\M_{\dd}^{-1})) \cap \mathscr{B}(L^{2}(\varpi_{k},L^2(\M_{\dd}^{-1})).$$
\end{lem}
\begin{proof} Recall from \cite[Lemma 2.10]{Kleber} that for any multi-index $\bm{\beta} \in \N^{3}$ with $|\bm{\beta}| \leq 2$, it holds that
\begin{equation*}
\left|\partial_{\bm{\beta}}\left(a_{ij}\ast f\right)(v)\right| \lesssim \min\left(\langle v\rangle^{\gamma+2}\|\partial_{\bm{\beta}}f\|_{L^{1}_{\g+2}},\langle v\rangle^{\gamma+2-|\bm{\beta}|}\|f\|_{L^{1}_{\g+2-|\bm{\beta}|}}\right)\,.
\end{equation*}
Thus, as in \cite[Lemma 2.11]{Kleber},
$$\left\|\left(a_{ij}\ast G_{\dd}\right)\partial^2_{ij}\M_{\dd}\right\|_{L^{2}(\M_{\dd}^{-1})}^{2} \leq C\,\|G_{\dd}\|_{L^{1}_{\g+2}}^{2}\,\int_{\R^{3}}\langle v\rangle^{2\g+4}\left|\partial^2_{ij}\M_{\dd}(v)\right|^{2}\,\M_{\dd}^{-1}\d v\,,$$
that is,
$$\bigg\|\sum_{ij}\left(a_{ij}\ast G_{\dd}\right)\partial^2_{ij}\M_{\dd}\bigg\|_{L^{2}(\M_{\dd}^{-1})} \leq C(k,\dd)\|G_{\dd}\|_{L^{1}_{\g+2}}\,.$$
Since 
$$\|(c \ast G_{\dd})\M_{\dd}\|_{L^{2}(\M_{\dd}^{-1})}\leq 2(\g+3)\|G_{\dd}\|_{L^{1}_{\g}}\|\M_{\dd}\|_{L^{1}(\varpi_{2\g})}\,,$$
we get \eqref{eq:tildeAa} and $\widetilde{\bm{A}}_{\dd} \in \mathscr{B}(L^{1}_{\g+2},L^{2}(\M_{\dd}^{-1}))$.  Now, by Cauchy-Schwarz inequality, it holds that for $s > \frac{3}{2}$
$$\|g\|_{L^{1}_{\g+2}} \leq \|g\|_{L^{2}(\varpi_{2\g+4+2s})}\,\|\langle \cdot \rangle^{-s}\|_{L^{2}}=C_{s}\,\|g\|_{L^{2}(\varpi_{2\g+4+2s})}$$
which proves that for $k > 2\g+4+3$ one also has $\widetilde{\bm{A}}_{\dd} \in \mathscr{B}(L^{2}(\varpi_{k}),L^2(\M_{\dd}^{-1})$.
\end{proof}
\noindent
Combining Lemma \ref{lem:Abounded} and Proposition \ref{prop:Badiss}, we prove the following lemma as in \cite[Corollary 2.12]{Kleber}.
\begin{lem}\label{lem:AU} Fix $\dd > 0$ such that $\dd\,a_{\dd}<1$ and $k > \max(k_{\dd},2\g+7)$.
Then, for any $a \in \R$ there exists $C(a,k,\dd) >0$ such that
$$\left\|\bm{A}_{\dd}U_{\dd}(t)\right\|_{\mathscr{B}(L^{2}(\varpi_{k}),L^{2}(\M_{\dd}^{-1}))} \leq C(a,k,\dd)\exp(at) \qquad \forall\, t \geq0.$$
\end{lem}

\end{document}